\documentclass[preprint]{elsarticle}       
\usepackage{lipsum}
\makeatletter
\def\ps@pprintTitle{%
 \let\@oddhead\@empty
 \let\@evenhead\@empty
 \def\@oddfoot{}%
 \let\@evenfoot\@oddfoot}
\makeatother

\usepackage{lscape}
\usepackage{lineno,hyperref}
\usepackage{rotating}
\usepackage{makeidx}
\usepackage{float}
\usepackage{epsfig}
\usepackage{amsmath,amssymb}
\usepackage[latin1]{inputenc}
\usepackage{epic,eepic}
\usepackage{overpic}
\usepackage{color}
\usepackage{amsfonts}
\usepackage{graphicx}
\usepackage{mwe} 
\usepackage{subcaption}

\usepackage{algorithm}
\usepackage{algorithmicx}
\usepackage{algpseudocode}
\usepackage{enumitem}

\usepackage{listings}
            {\hfill\penalty10000\raisebox{-.09em}{\large\bf\rm $\blacksquare$}\par\medskip}

\newtheorem{theorem}{Theorem}[section]
\newtheorem{lemma}[theorem]{Lemma}
\newtheorem{proposition}[theorem]{Proposition}

\newtheorem{remark}{Remark}
\newtheorem{proof}{Proof}
[section]

\journal{}

\begin{document}

\small

\begin{frontmatter}

\title{Data-dependent approximation through RBF}
\tnotetext[label1]{The  fourth author has been supported through project CIAICO/2021/227 (Proyecto financiado por la Conselleria de Innovaci\'on, Universidades, Ciencia y Sociedad digital de la Generalitat Valenciana), by grant PID2020-117211GB-I00 and by PID2023-146836NB-I00 funded by MCIN/AEI/10.13039/501100011033.}

\author[UPCT]{Jos\'e Kuruc}
\ead{Jose.kuruc@hotmail.com}
\author[TAU]{David Levin}
\ead{levindd@gmail.com}
\author[UV]{Pep Mulet}
\ead{pep.mulet@uv.es}
\author[UPCT]{Juan Ruiz-\'Alvarez}
\ead{juan.ruiz@upct.es}
\author[UV]{Dionisio F. Y\'a\~nez}
\ead{Dionisio.Yanez@uv.es}

\date{Received: date / Accepted: date}

\address[UPCT]{Departamento de Matem\'atica Aplicada y Estad\'istica. Universidad  Polit\'ecnica de Cartagena, Cartagena (Spain).}
\address[TAU]{School of Mathematical Sciences. Tel-Aviv University, Tel-Aviv (Israel).}
\address[UV]{Departamento de Matem\'aticas. Universidad de Valencia, Valencia (Spain).}


\begin{abstract}
In this article we present a modification of classical Radial Basis Function (RBF) interpolation techniques aimed at reducing oscillations near discontinuities in one and two dimensions. Our approach introduces an adaptive mechanism by varying the shape parameter of the RBFs and making it data-dependent, forcing it to tend to infinity in the vicinity of discontinuities. This modification results in kernel functions that locally resemble 
delta functions, effectively minimizing spurious oscillations.

To detect discontinuities, we employ smoothness indicators: for grid-based data, these are computed as undivided second-order differences squared. For scattered data, we use least squares approximations of the Laplacian multiplied by the square of the mean local separation of the stencil points, and then squared. These indicators guide the adaptive adjustment of the shape parameter.

We prove the invertibility of the resulting interpolation matrix and propose a solution strategy that maintains the condition number comparable to that of a system where points near discontinuities are excluded. Numerical experiments in one and two dimensions demonstrate that the proposed method significantly reduces oscillations near discontinuities across various kernel types, whether locally or globally supported. At the same time, the interpolation accuracy and matrix conditioning in smooth regions remain essentially unchanged, as measured by the infinity norm of the error and the condition number.
\end{abstract}
\end{frontmatter}

\section{Introduction and review}

Radial Basis Function  interpolation is a powerful method for approximating multivariate functions. It was initially proposed by R. Hardy in \cite{Hardy1971} using multiquadric functions and lately advanced by C. A. Micchelli in \cite{Micchelli1986}, who provided rigorous conditions for the solvability of RBF interpolation problems. In a simple way, the problem can be stated as follows: Given a set of data points $\{(\mathbf{x}_i, y_i)\}_{i=1}^N\subset \mathbb{R}^n\times \mathbb{R}$, the goal is to find a function $s(\mathbf{x})$ that interpolates the data, i.e., $s(\mathbf{x}_i) = y_i$ for all $i$. The interpolant is constructed as a linear combination of radial basis functions $\phi(\varepsilon\|\mathbf{x} - \mathbf{x}_i\|)$ centered at the data points:

\begin{equation}\label{interp1}
s(\mathbf{x}) = \sum_{i=1}^N \lambda_i \phi(\varepsilon\|\mathbf{x} - \mathbf{x}_i\|),
\end{equation}
where $\lambda_i$ are the coefficients to be determined, $\phi$ is the chosen radial basis function, and $\varepsilon$ is the {\it shape parameter}, that determines its effective support.


To find the coefficients $\lambda_i$ in RBF interpolation, we solve the following system of equations. Given the interpolation condition $s(\mathbf{x}_i) = y_i$ for all $i$, we have:
\begin{equation}\label{interp2}
s(\mathbf{x}_i) = \sum_{j=1}^N \lambda_j \phi(\varepsilon\|\mathbf{x}_i - \mathbf{x}_j\|) = y_i \quad \text{for } i = 1, 2, \ldots, N.
\end{equation}
This leads to the linear system of equations:
$$
\begin{pmatrix}\label{sys}
\phi(\varepsilon\|\mathbf{x}_1 - \mathbf{x}_1\|) & \phi(\varepsilon\|\mathbf{x}_1 - \mathbf{x}_2\|) & \cdots & \phi(\varepsilon\|\mathbf{x}_1 - \mathbf{x}_N\|) \\
\phi(\varepsilon\|\mathbf{x}_2 - \mathbf{x}_1\|) & \phi(\varepsilon\|\mathbf{x}_2 - \mathbf{x}_2\|) & \cdots & \phi(\varepsilon\|\mathbf{x}_2 - \mathbf{x}_N\|) \\
\vdots & \vdots & \ddots & \vdots \\
\phi(\varepsilon\|\mathbf{x}_N - \mathbf{x}_1\|) & \phi(\varepsilon\|\mathbf{x}_N - \mathbf{x}_2\|) & \cdots & \phi(\varepsilon\|\mathbf{x}_N - \mathbf{x}_N\|)
\end{pmatrix}
\begin{pmatrix}
\lambda_1 \\
\lambda_2 \\
\vdots \\
\lambda_N
\end{pmatrix}
=
\begin{pmatrix}
y_1 \\
y_2 \\
\vdots \\
y_N
\end{pmatrix}.
$$


In matrix form, this can be written as:
\begin{equation}\label{sistema}
\mathbf{A} \boldsymbol{\lambda} = \mathbf{y},
\end{equation}
where $\mathbf{A}$ is the $N \times N$ matrix with entries $A_{ij} = \phi(\varepsilon\|\mathbf{x}_i - \mathbf{x}_j\|)$, $\boldsymbol{\lambda} = (\lambda_1, \lambda_2, \ldots, \lambda_N)^T$ is the vector of coefficients, and $\mathbf{y} = (y_1, y_2, \ldots, y_N)^T$ is the vector of data values. All the functions $\phi$ used in this paper determine a strictly positive definite matrix $\mathbf{A}$ (see e.g.\cite{Fasshauer2007,Wendland}).


The choice of the radial basis function $\phi(r\varepsilon)$ significantly affects the properties of the interpolant. Common choices include those presented in Table \ref{tabla1nucleos}.

\begin{table}[h!]
\centering
\begin{tabular}{lll}
\hline
$\phi(r\varepsilon)$& RBF &                                                          \\ \hline
$e^{-  (\varepsilon r)^2}$ & Gaussian $\mathcal{C}^\infty$ & G                                          \\
$\left(1 + (\varepsilon r)^2 \right)^{-1/2}$ & Inverse MultiQuadratic $C^\infty$ & IMQ  \\
$e^{-  \varepsilon r} \left( 1 +   \varepsilon r \right)$ & Mat\'ern $\mathcal{C}^2$ & M2 \\
$e^{-  \varepsilon r} \left( 3 + 3  \varepsilon r +   (\varepsilon r)^2 \right)$ & Mat\'ern $\mathcal{C}^4$ & M4 \\
$(1 -  \varepsilon r)^4_+ \left( 4 \varepsilon r + 1 \right)$ & Wendland $\mathcal{C}^2$ & W2 \\
$(1 -  \varepsilon r)^6_+ \left( 35 (\varepsilon r)^2 + 18 \varepsilon r + 3 \right)$ & Wendland $\mathcal{C}^4$ & W4 \\ \hline
\end{tabular}
\caption{Examples of RBFs that resemble a discrete delta function when $\varepsilon\to\infty$.}\label{tabla1nucleos}
\end{table}

In Table \ref{tabla1nucleos}, $\varepsilon$ is a {\it shape parameter} that controls the width of the basis function. The variable $r$ represents the Euclidean distance $\|\mathbf{x} - \mathbf{x}_i\|$ between points in any dimensional space. As in \cite{Fasshauer2007}, we use the shape parameter $\varepsilon$ applied directly to $r$, which has the effect of returning {\it flat} basis functions if $\varepsilon$ is decreased or more peaked/localized basis functions if it is increased. It can be easily checked that this idea applies to all the $\phi(r\varepsilon)$ functions presented in Table \ref{tabla1nucleos}.


The smoothness of the interpolant depends on the smoothness of the chosen basis function $\phi$. For example:

\begin{itemize}
    \item \textbf{Gaussian and Multiquadric:} These functions are infinitely differentiable, leading to very smooth interpolants.
    \item \textbf{Wendland and Mat\'ern:} These functions are only finitely differentiable, resulting in less smooth interpolants but with compact support, which can be advantageous for large datasets or data with discontinuities.
\end{itemize}

The order of accuracy in smooth regions is also influenced by the choice of $\phi$. Infinitely differentiable functions like the Gaussian can achieve high accuracy, while compactly supported functions like the Wendland basis may have lower accuracy but better computational efficiency.

In this manuscript we are interested in accurately reconstructing functions with jump discontinuities, which is a challenge that has long been recognized in numerical analysis. The main reason is that classical global or high-order linear approximation methods, including those based on RBFs, are known to suffer from the Gibbs phenomenon when applied to such functions. This results in oscillatory artifacts near discontinuities, which degrade the quality of the approximation. Over the past few decades, a wide range of strategies have been developed to address this issue (see, e.g., \cite{Amir2018,Arandiga2005,Bozzini2013,Bozzini2014,Crampton2005,Gout2008,Lenarduzzi2017}). While many of these approaches focus on general interpolation or surface recovery, relatively few have targeted RBF-based methods specifically. Notable exceptions include the work of Jung \cite{Jung2007}, who demonstrated that using multiquadric RBFs with adaptively vanishing shape parameters near discontinuities can significantly reduce oscillations. This technique effectively linearizes the interpolant locally, improving stability and accuracy. A more advanced solution involves the use of variably scaled RBFs, where discontinuous scaling functions are introduced near jump locations. This leads to interpolants that are themselves discontinuous, allowing for more faithful recovery of the underlying function \cite{DeMarchi2019,Romani2018,Erb2019}. However, these methods typically require prior knowledge of the discontinuity locations, making edge detection a necessary preprocessing step. In parallel, data-dependent reconstruction techniques such as ENO and WENO have been adapted to the RBF framework. Guo and Jung \cite{Guo2017a,Guo2017b} proposed RBF-ENO and RBF-WENO schemes that optimize the shape parameter to enhance performance near discontinuities. Building on these ideas, Ar\`andiga et al. \cite{Arandiga2020} developed a multiquadric RBF-WENO method that combines multiple RBFs (2 or 3) to achieve high-order accuracy while controlling oscillations.


This paper focuses on the adaption of general RBF interpolation to the presence of discontinuities in the data using a strategy that is different from all the ones mentioned previously. 

The article is organized as follows. Section 2 introduces the data-dependent modification applied to the shape parameter, analyzes its impact on the affected kernel functions, and discusses the required properties of both the smoothness indicators and the composing function used to adapt the interpolation near discontinuities. It also presents the final form of the interpolation expression that achieves this adaptivity. Section 3 provides a theoretical analysis of the invertibility of the modified interpolation matrix. Section 4 explores the transformation of the interpolation system into a block matrix and examines its condition number. 
Section 5 discusses the design of smoothness indicators in one and two dimensions, both for grid-based and mesh-free data. Section 6 presents numerical experiments that support the theoretical findings, demonstrating the effectiveness of the proposed method. Finally, Section 7 summarizes some conclusions.

%
%
%
%
%
%
%
%

\section{A new data-dependent approach}



In this section, we introduce an automatic strategy to locally modify the shape parameter $\varepsilon$ based on the smoothness of the data. The goal is to reduce the effective support of the kernel functions near discontinuities (where oscillations typically arise in the classical approach) while preserving the standard behavior in smooth regions. This adaptive modification aims to suppress interpolation artifacts without affecting accuracy elsewhere. A key challenge of this approach lies in ensuring the invertibility of the resulting interpolation matrix while keeping the condition number largely unaffected. Since the shape parameter is modified depending on the smoothness of the data, the structure of the matrix changes, and its original properties, such as symmetry and positive definiteness, may no longer hold. We address this issue in detail in the next section.

To begin, we consider the general RBF interpolation framework introduced in Equation (\ref{interp1}) of the previous section, now explicitly incorporating a point-dependent shape parameter $\tilde{\varepsilon}_i$. Accordingly, we propose solving the system corresponding to the following data-dependent interpolation formulation:

\begin{equation}\label{int}
\tilde s(\mathbf{x}) = \sum_{i=1}^N \lambda_i \phi(\tilde\varepsilon_i\|\mathbf{x} - \mathbf{x}_i\|).
\end{equation}

The central idea behind the modified RBF interpolation proposed in this work is to restrict the influence of kernel functions near discontinuities, effectively limiting their contribution to the point at which they are centered. This is achieved by introducing a data-dependent modification of the shape parameter $\varepsilon$, which is designed to tend to infinity at locations close to discontinuities, while remaining unmodified at smooth zones. As a result, the affected kernels become highly localized, resembling 
delta functions, and their interaction with neighboring points is significantly reduced. This localization leads to a structural change in the interpolation matrix: the columns corresponding to these modified kernels are replaced by basis functions with minimal support. In all the kernels presented in Table \ref{tabla1nucleos}, this results in columns that are zero everywhere except for a single one on the diagonal. Although this modification helps suppress oscillations near discontinuities, as shown in the numerical experiments, it may raise concerns about the invertibility of the interpolation matrix and the preservation of its original properties. These issues are examined in detail in the following section.

The modification of the shape parameter that we propose is
\begin{equation}\label{tildegamma}
\tilde\varepsilon_i=\varepsilon\frac{1}{c+\psi(I_i)},
\end{equation}
with $\psi: \mathbb{R}\to \mathbb{Z}$ defined as
\begin{equation}\label{fI}
\psi(x)=\text{round}\left(e^{-(C x)^t}\right),
\end{equation}
where $\text{round}$ is the rounding to the nearest integer operation, $c$ is a parameter to avoid divisions by zero, $C>0$ and $t\ge 1$ are real constants that can be tuned, and that help to attain kernels that resemble the 
delta function:
\begin{equation}\label{deltafunction}
\delta(x) =
\begin{cases}
1, & x = 0, \\
0, & x \neq 0,
\end{cases}
\end{equation}
close to the discontinuities, while leaving the shape parameter essentially the same at smooth zones. In the numerical experiments we set $c=10^{-16}$, $C=10$ and $t=2$. The $I_i$ are smoothness indicators, i.e. a local numerical measure of the smoothness of the data around $\mathbf{x}_i$. The points used to calculate the smoothness indicator for each $i\in\{1,\hdots,N\}$ are denoted by stencil $\mathcal{S}_i\subset \{\mathbf{x}_j\}_{j=1}^N.$
 Let $h$ denote the fill distance (see, e.g., \cite{Fasshauer2007,Wendland}). Then, in simple terms, $I_i$ is a locally supported operator that should satisfy the following properties:

\begin{enumerate}[label={\bfseries P\arabic*}]
\item\label{P1sm1d} The order of a smoothness indicator that is free of discontinuities is $h^m$, i.e.
$$I_{i}=\mathcal{O}(h^m) \,\, \text{if}\,\, f \,\, \text{is smooth in the considered support of $I_i$}.$$
\item\label{P2sm1d}   When a discontinuity crosses the stencil $\mathcal{S}_{i}$ then
$I_{i} =\mathcal{O}(1) \,\, \text{as}\,\, h\to 0.$

\end{enumerate}


Now we can state the following proposition:

\begin{proposition}\label{t1}
Assuming the expression for $\tilde\varepsilon$ given in (\ref{tildegamma}) with the properties $P1$ and $P2$ given for the smoothness indicator $I$, and $C=\frac{1}{h},$ we choose RBF kernels that satisfy:
\begin{itemize}
\item The function $\phi(r\tilde\varepsilon_i)\approx\phi(r\varepsilon)$  at smooth zones.
\item The function $\phi(r\tilde\varepsilon_i)\approx\delta(r)$, if the support of the smoothness indicator crosses a jump discontinuity in the function, where $\delta(r)$ is defined in Eq. \eqref{deltafunction}.


\end{itemize}
\end{proposition}
Examples of kernels satisfying these conditions are presented in Table \ref{tabla1nucleos}, also introduced and analyzed in \cite{Fasshauer2007}. For some of these functions we need to introduce the cutoff function $(\cdot)_+:\mathbb{R} \to \mathbb{R}$, that is defined as:
 \begin{equation*}
 (x)_+=\max(x, 0).
 \end{equation*}

As it has been already mentioned, in radial basis function  interpolation, the choice of the shape parameter $\varepsilon$ critically affects the behavior of the interpolant. When $\varepsilon$ is large, the basis functions become highly localized, leading to an interpolant that resembles a series of sharp spikes centered at the data points, a phenomenon known as the \textit{bed of nails effect}. A detailed discussion of this behavior can be found in \cite{Fasshauer2007}. Thus, a problem that can be found with the interpolation in (\ref{int}) is that it typically leads to {\it nails} close to the discontinuity. Thus, once we have solved the system, we propose to eliminate those nails, so the resulting approximation will only interpolate at the rest of the points, where the smoothness indicators detect a smooth zone. To do so, after solving the system for \eqref{int}, we propose the following interpolation:

\begin{equation}\label{int1}
\tilde s_{f,X}(\mathbf{x})  = \sum_{i=1}^N \lambda_i \psi( I_i)\phi(\tilde\varepsilon_i\|\mathbf{x} - \mathbf{x}_i\|).
\end{equation}
The resulting interpolation maintains desirable properties even near discontinuities. By locally increasing the shape parameter $\varepsilon$ in regions close to a discontinuity, the corresponding radial basis functions become highly localized, limiting their influence to a small neighborhood around their centers. As a result, removing these localized kernels has minimal impact on the overall approximation, since their contribution is confined to a very narrow region. The main idea behind the modification proposed is to remove the kernels that are close to the discontinuities and that contribute the most to the oscillations. 

The first property that can be deduced from the expression of the new data-dependent approximant is that it preserves the continuity of the classical operator. This is stated in the following proposition.

\begin{proposition}
Let $\phi(\|\cdot\|)\in \mathcal{C}^\nu$ be a radial basis function. Then the operator defined in Eq. \eqref{int1}, $\tilde s_{f,X}\in \mathcal{C}^\nu$. 
\end{proposition}
The proof is direct from the expression of $\tilde s_{f,X}$, as it is a combination of functions $\phi$ evaluated at different points. Also, the modifications introduced in the data-dependent operator do not depend on the variable $\mathbf{x}$.

In Section~\ref{si_sec}, we present several strategies for designing smoothness indicators and describe in detail the specific indicators used in the numerical experiments.


\section{Analysis of the invertibility of the new interpolation matrix}
As described in previous section, we plan to make the radial basis function  interpolation adaptive near discontinuities by modifying the shape parameter locally in a non-linear way.  As a consequence, some of the original kernels (which correspond to columns of the interpolation matrix) are replaced by modified basis functions with locally adjusted shape parameters. In the limit, the modified shape parameter will force the kernels placed close to the discontinuity to resemble the delta function in (\ref{deltafunction}), which will be in the end removed, enabling the interpolant to approximate better the non-smooth behavior while remaining within the RBF framework.

It is known \cite{Micchelli1986, Fasshauer2007} that all the kernels presented in Table \ref{tabla1nucleos} have properties that ensure the invertibility of the corresponding interpolation matrix, guaranteeing the existence of a unique solution to the interpolation problem. However, when the shape parameter is altered locally, i.e., different values are used for different data points, this effectively modifies certain columns of the interpolation matrix. Such changes may disrupt the strict positive definiteness of the matrix, potentially affecting its invertibility. Therefore, it becomes essential to analyze whether the locally modified interpolation matrix retains its non-singular nature, and under what conditions this property is preserved.

\subsection{Invertibility of the system}

We prove the following lemma to demonstrate that our new system has a unique solution. The main idea of our new algorithm is to modify several columns of the matrix depending on the smoothness of the underlying data. As a result, the columns corresponding to data where the smoothness indicator is large will be zero except at the diagonal. This occurs because the modification of the shape parameter near discontinuities aims to reduce the effective support of the kernels close to the discontinuities, effectively transforming them into delta functions. The use of the smoothness indicator function in (\ref{tildegamma}) and (\ref{fI}) with a sufficiently large constant $C$ (we use $C=10$ in our experiments) guaranties this behavior.

\begin{lemma}\label{lema1}
Let $\mathbf{A}$ be a square matrix $N\times N$ defined in Eq. \eqref{sistema} (symmetric and positive definite) and $\mathbf{A}(:,j)$ the columns of the matrix, with $j=1,\hdots,N$ and let $\mathbf{B}$ be the matrix with same components of $\mathbf{A}$ except one column which is modified by zero elements in all the entries different of the diagonal and a non-zero element in the diagonal, i.e. there exist $j_0\in\{1,\dots,N\}$ and $\lambda\neq 0$ such that
$$\mathbf{B}(i,j_0)=0,\,\, i\neq j_0, \quad \mathbf{B}(j_0,j_0)=\lambda,\quad \mathbf{B}(:,j)=\mathbf{A}(:,j),\,\, j\neq j_0,$$
then the determinant of the matrix $\mathbf{B}$ is different from zero, i.e.
$$|\mathbf{B}|\neq 0.$$
\end{lemma}
\begin{proof}
We develop the determinant of the matrix $\mathbf{B}$ by the column $j_0$. For this, we denote:
$$\mathbf{A}^{j_0}=\begin{pmatrix}
\phi(\|\mathbf{x}_1 - \mathbf{x}_1\|) & \phi(\|\mathbf{x}_1 - \mathbf{x}_2\|) & \cdots & \phi(\|\mathbf{x}_1 - \mathbf{x}_{j_0-1}\|) & \phi(\|\mathbf{x}_1 - \mathbf{x}_{j_0+1}\|) & \cdots& \phi(\|\mathbf{x}_1 - \mathbf{x}_N\|) \\
\phi(\|\mathbf{x}_2 - \mathbf{x}_1\|) & \phi(\|\mathbf{x}_2 - \mathbf{x}_2\|) & \cdots & \phi(\|\mathbf{x}_2 - \mathbf{x}_{j_0-1}\|) & \phi(\|\mathbf{x}_2 - \mathbf{x}_{j_0+1}\|) & \cdots& \phi(\|\mathbf{x}_2 - \mathbf{x}_N\|) \\
\vdots & \vdots & \ddots & \vdots & \vdots & \ddots & \vdots \\
\phi(\|\mathbf{x}_{j_0-1} - \mathbf{x}_1\|) & \phi(\|\mathbf{x}_{j_0-1} - \mathbf{x}_2\|) & \cdots & \phi(\|\mathbf{x}_{j_0-1} - \mathbf{x}_{j_0-1}\|) & \phi(\|\mathbf{x}_{j_0-1} - \mathbf{x}_{j_0+1}\|) & \cdots& \phi(\|\mathbf{x}_{j_0-1} - \mathbf{x}_N\|) \\
\phi(\|\mathbf{x}_{j_0+1} - \mathbf{x}_1\|) & \phi(\|\mathbf{x}_{j_0+1} - \mathbf{x}_2\|) & \cdots & \phi(\|\mathbf{x}_{j_0+1} - \mathbf{x}_{j_0-1}\|) & \phi(\|\mathbf{x}_{j_0+1} - \mathbf{x}_{j_0+1}\|) & \cdots& \phi(\|\mathbf{x}_{j_0+1} - \mathbf{x}_N\|) \\
\vdots & \vdots & \ddots & \vdots & \vdots & \ddots & \vdots \\
\phi(\|\mathbf{x}_N - \mathbf{x}_1\|) & \phi(\|\mathbf{x}_N - \mathbf{x}_2\|) & \cdots & \phi(\|\mathbf{x}_N - \mathbf{x}_{j_0-1}\|) & \phi(\|\mathbf{x}_N - \mathbf{x}_{j_0+1}\|) & \cdots& \phi(\|\mathbf{x}_N - \mathbf{x}_N\|) 
\end{pmatrix}.$$
Note that $\mathbf{A}^{j_0}$ is the matrix used when we solve the RBF problem taking the nodes $\{\mathbf{x}_{1},\dots,\mathbf{x}_{j_0-1},\mathbf{x}_{j_0+1},\dots,\mathbf{x}_{N}\}$ and therefore it is invertible. Thus,
\begin{equation*}
\begin{split}
|\mathbf{B}|=\lambda |\mathbf{A}^{j_0}|\neq 0.
\end{split}
\end{equation*}
\end{proof}
\begin{proposition}\label{prop1}
Let $\mathbf{A}$ be a square matrix $N\times N$ defined in Eq. \eqref{sistema} (symmetric and positive definite) and $\mathbf{A}(:,j)$ the columns of the matrix, with $j=1,\hdots,N$ and let $\mathbf{B}$ be the matrix with same components of $\mathbf{A}$ except some columns  which are modified by zero elements in all the entries different of the diagonal and non-zero elements in the diagonal, i.e. there exist $j_m$, and $\lambda_m\neq 0$, $m=1,\dots,k$, $k\in\{1,\dots,N\}$ such that
$$\mathbf{B}(i,j_m)=0,\,\, i\neq j_m, \quad \mathbf{B}(j_m,j_m)=\lambda_m,\quad \mathbf{B}(:,j)=\mathbf{A}(:,j),\,\, j\neq j_m,\,\,m=1,\dots,k,$$
then the determinant of the matrix $\mathbf{B}$ is different from zero, i.e.
$$|\mathbf{B}|\neq 0.$$
\end{proposition}
\begin{proof}
It is direct by induction and Lemma \ref{lema1}.
\end{proof}


\section{Analysis of the condition number of the matrix in the new system} 

Let us consider what happens when modifying the interpolation matrix by replacing certain columns (originally corresponding to the kernels in Table \ref{tabla1nucleos}) with columns of zeros except for a one at the diagonal. As we have mentioned before, these columns are selected based on a smoothness indicator: when the indicator is large, the shape parameter of the Gaussian is increased significantly, resulting in a kernel that behaves like a delta function, effectively producing a column with zeros and a single one on the diagonal.

Since the RBF interpolation is a linear combination of coefficients $\lambda_i$ multiplied by the kernel functions evaluated at each point, we can rearrange the interpolation system so that the points (and thus the Gaussian kernels) associated with large smoothness indicators are placed at the end of the linear combination. This leads to a block structure in the global interpolation matrix, where the identity matrix appears in the lower-right corner, and the upper-left block represents the RBF interpolation matrix excluding the points affected by large smoothness indicators. The resulting matrix can be expressed as
\begin{equation}\label{block1}
{\bf M} = \begin{bmatrix}
{\bf \tilde A} & 0 \\
{\bf C} & {\bf I}
\end{bmatrix},
\end{equation}
where: 
\begin{itemize}
    \item ${\bf \tilde A}$ is the interpolation matrix without considering those points where the smoothness indicator is large. Thus, this matrix has the invertibility properties associated with the chosen kernel. For example, for the Gaussian kernel, it is symmetric and positive definite.
    \item ${\bf C}$ is the result of particularizing the kernels at smooth zones at the points where the smoothness index is large. 
    \item $\bf I$ is the identity matrix.
\end{itemize}

From the expression in (\ref{block1}) we can also see that the determinant is different from zero since $\det(\bf{M})=\det( \bf{\tilde{A}})\det(\bf{I})=\det( \bf{\tilde{A}})$.

The structure of {\bf M} in (\ref{block1}) also allows us 
to analyze the conditioning of the system and to simplify the solution process by decoupling the well-behaved identity block from the more complex interpolation block. Consider the system that we want to solve, where the block matrix in (\ref{block1}) is involved
\[
{\bf M} \begin{bmatrix} \mathbf{u} \\ \mathbf{u}' \end{bmatrix} = \begin{bmatrix} \mathbf{z} \\ \mathbf{z}' \end{bmatrix}.
\]
In this case, we can 
decouple the system. First,  we solve for $\mathbf{u}$ from the first block row:
\[
{\bf \tilde A} \mathbf{u} = \mathbf{z}  \quad \Rightarrow \quad \mathbf{u} = {\bf \tilde A}^{-1} \mathbf{z} .
\]
Then, we replace into the second block row:
\begin{equation}\label{flechas}
\bf{C} \mathbf{u} + \bf{I} \mathbf{u}' = \mathbf{z}' \quad \Rightarrow \quad \bf{C} {\bf \tilde A}^{-1} \mathbf{z} + \mathbf{u}' = \mathbf{z}' 
\quad \Rightarrow \quad \mathbf{u}' = \mathbf{z}'- \bf{C} {\bf \tilde A}^{-1} \mathbf{z} \Rightarrow \quad \mathbf{u}' = \mathbf{z}' - \bf{C} \boldsymbol{u},
\end{equation}
where $\boldsymbol{u}={\bf \tilde A}^{-1} \mathbf{z}$ is the solution of the RBF interpolation problem,
\begin{equation}\label{sistema_b}
\mathbf{\tilde A} \boldsymbol{u} = \mathbf{z},
\end{equation}
considering only those data points at smooth zones, i.e. where the smoothness indicator is small (of order $\mathcal{O}(h^n)$).
Now it must be clear that:
\begin{itemize}
    \item The reduced system for $\bf{u}'$ is perfectly conditioned.
    \item Decoupling the system ensures that the modifications introduced in the RBF interpolation matrix (specifically the replacement of certain kernel columns with identity-like columns based on a smoothness indicator) allow us to focus solely on the condition number of a reduced RBF interpolation matrix ${\bf \tilde A}$. This reduced matrix corresponds to the interpolation problem where data points with large smoothness indicators have been excluded. If this matrix has good invertibility properties, then the new system inherits these properties.
\end{itemize}

\begin{lemma}
  The condition of the matrix $M$ in the $\infty$-norm is bounded
  as
  \begin{equation*}
    \kappa_{\infty}({\bf M})\leq \kappa_{\infty}({\bf\tilde A})
    \max(1, \frac{1+\Vert {\bf C}\Vert_{\infty}}{\Vert{\bf \tilde
        A}\Vert_{\infty}})
    \max(1, \Vert {\bf C}\Vert_{\infty}+\frac{1}{\Vert{\bf \tilde
        A}^{-1}\Vert_{\infty}})
  \end{equation*}    
  \end{lemma}
  \begin{proof}
    Based on the fact that 
\begin{equation*}
{\bf M}^{-1} = \begin{bmatrix}
{\bf \tilde A}^{-1} & 0 \\
-{\bf C} {\bf \tilde A}^{-1}& {\bf I}
\end{bmatrix},
\end{equation*}
and that $\Vert (A_{i,j})\Vert_{\infty}\leq \max_{i}\sum_{j}\Vert
A_{i,j}\Vert_{\infty}$, where $(A_{i,j})$ denotes a block matrix.
\end{proof}

\begin{remark}
In all the numerical experiments that we have performed we have found that it is not necessary to use the block structure of the interpolation matrix, as the system matrix with the data-dependent modification presents condition numbers that are very similar to those of the original matrix. The reader can refer to the numerical experiments to check the condition numbers obtained.
\end{remark}

\section{Discussion about the smoothness indicators in one and several dimensions, and for regular and mesh-free data}\label{si_sec}

A smoothness indicator must satisfy the two conditions P1 and P2 described in Section 2, which ensure that the indicator vanishes for smooth data and grows appropriately near discontinuities. A valid choice for uniform univariate data in this context is the squared undivided second-order difference, given by
\begin{equation}\label{ud1}
\beta = \left( f(x+h) - 2f(x) + f(x-h) \right)^2,
\end{equation}
which has size $\mathcal{O}(h^4)$ for sufficiently smooth functions. In the same sense, for data in two dimensions using a uniform grid, we can use the squared five points stencil Laplacian as smoothness indicator, which is given by,
\begin{equation}\label{ud2}
\beta = \left( f(x+h, y) + f(x-h, y) + f(x, y+h) + f(x, y-h) - 4f(x, y) \right)^2.
\end{equation}
For non-uniform data in one or several dimensions, one can use a different strategy. 
Let $X = \{\mathbf{x}_i\}_{i=1}^N \subset \Omega \subset \mathbb{R}^2$ be a set of scattered nodes and let $u : \Omega \to \mathbb{R}$ denote pointwise samples of a sufficiently smooth function $u$. For each centre $\mathbf{x}_i \in X$ choose a local stencil of $K$ neighbours $S(\mathbf{x}_i) = \{\mathbf{x}_{i_j}\}_{j=1}^K\subset X$. We seek weights $w_j$ such that
\[
\Delta u(\mathbf{x}_i) \approx \sum_{j=1}^K w_j u(\mathbf{x}_{i_j}),
\]
and that this formula is exact for a finite polynomial space. A similar idea is used, for example, in \cite{LiIto2006} (see pages 68-69) to compute surface derivatives in several dimensions. Let $P_m$ denote the polynomials in two variables of total degree $\leq m$. Imposing exactness on a basis $\{p_\ell\}_{\ell=1}^M$ of $P_m$ yields the moment conditions
\[
\sum_{j=1}^K w_j p_\ell(\mathbf{x}_{i_j}) = \Delta p_\ell(\mathbf{x}_i), \quad \ell = 1, \dots, M,
\]
where $M = \dim P_m$. In matrix form this constraint reads
\[
\mathbf{V} \mathbf{w} = \mathbf{b},
\]
with $\mathbf{V} \in \mathbb{R}^{M \times K}$, $\mathbf{V}_{\ell,j} = p_\ell(x_j)$, $\mathbf{w} = (w_j)_{j=1}^K$ and $b_\ell = \Delta p_\ell(x_0)$. When $K = M$ and $\mathbf{V}$ is nonsingular, the weights are obtained by direct solve $w = \mathbf{V}^{-1} \mathbf{b}$. More robustly, one typically takes $K < M$ and computes weights by solving the normal equations
\[
\mathbf{V}^\top\mathbf{V} \mathbf{w} = \mathbf{V}^\top\mathbf{b}, \]
provided that $\text{rank}(V^\top V)=K$. Small Tikhonov regularization may be added if $\mathbf{V}^\top \mathbf{V}$ is ill-conditioned (a tolerance for the condition number of $10^{-10}$ is used in the numerical experiments):
\[
(\mathbf{V}^\top \mathbf{V} + \lambda \mathbf{I}) \mathbf{w} = \mathbf{V}^\top\mathbf{b},\quad \lambda>0. \]
For $m = 2$ and univariate data, the monomial basis that we use is $\{1, x-x_0, (x-x_0)^2\}$. Then we have $M = 3$ and the right-hand side is
\[
\mathbf{b} = ( 0, 0, 2)^\top.
\]
Analogously, for data in 2D and $m = 2$, we use the monomial basis $\{1, (x-x_0), (y-y_0), (x-x_0)^2, (x-x_0)(y-y_0), (y-y_0)^2\}$, we have $M = 6$ and the right-hand side is
\[
\mathbf{b} = (0, 0, 0, 2, 0, 2)^\top,
\]
because $\Delta 1 = \Delta x = \Delta y = 0$, $\Delta x^2 = \Delta y^2 = 2$, $\Delta(xy) = 0$. The resulting weight vector $w$ is then applied to the data values $u(\mathbf{x}_{i_j})$ on the stencil to produce the discrete Laplacian approximation $\sum_j w_j u(\mathbf{x}_{i_j})$.

As we want to obtain something equivalent to the undivided second differences but for sparse data in several dimensions, we need to define a characteristic scale for the stencil. To do so, we compute a local fill-scale $h_{\text{loc}}(\mathbf{x}_i)$. A simple and effective choice is the arithmetic mean of the Euclidean distances from $\mathbf{x}_i$ to the other stencil nodes (excluding the zero self-distance):
\[
h_{\text{loc}}(\mathbf{x}_i) = \frac{1}{K-1} \sum_{j=1}^K \|\mathbf{x}_{i_j} - \mathbf{x}_i\|.
\]
This $h_{\text{loc}}$ approximates the typical spacing of the stencil. Now, we can multiply the computed $\Delta$-approximation by $h_{\text{loc}}^2$ to obtain
\[
I_i[u] := h_{\text{loc}}(\mathbf{x}_i)^4\left( \sum_{j=1}^K w_j u(\mathbf{x}_{i_j})\right)^2,
\]
which has the same units as a discrete undivided second-difference operator.

A good remark here is that using this approach with 3 points in the univariate case or 5 points in the bivariate case with a uniform grid will return the expressions of the smoothness indicators in (\ref{ud1}) and (\ref{ud2}) respectively. Also, a similar approach can be used for more than 2 dimensions.

\section{Numerical experiments}
In this section, we present a series of numerical experiments designed to evaluate the performance of the proposed data-dependent interpolation algorithms. We begin by considering smooth functions to assess the accuracy and stability of the method when using different RBF kernels, and to compare the results with those obtained using classical linear interpolation techniques. In addition to the interpolation accuracy using the infinity norm of the error, we analyze the condition number of the resulting interpolation matrices to understand the numerical stability of the approach. Subsequently, we examine the adaptive behavior of the data-dependent algorithms near jump discontinuities in both one and two dimensions. For all the experiments, we consider both gridded data and scattered data generated using Halton sequences, allowing us to test the robustness of the method across different point distributions.
\subsection{Behaviour at smooth zones}
Let us consider the smooth function defined by
\[
f(x) = 1+\sin(\pi x), \quad x \in \mathbb{R}.
\]
The domain of interest is the interval $[0, 1]$, but we perform the approximation in the interval $[-1,2]$ to avoid the numerical effects that might appear at the boundary. We begin with a set of $N = 3(2^l+1)$ initial nodes in the large interval, either uniform or Halton points. From these, we generate 10 intermediate evaluation points between each pair of consecutive initial nodes. Then, we analize the approximation in the domain of interest $[0, 1]$. The approximation is performed using two distinct strategies:
\begin{itemize}
    \item Classical RBF interpolation, denoted as RBF$_{\mathcal{H}}$.
    \item Data-dependent RBF interpolation, denoted as DD-RBF$_{\mathcal{H}}$.
\end{itemize}
Here, $\mathcal{H} \in \{\text{G}, \text{IMQ},\text{W2}, \text{W4}, \text{M2}, \text{M4} \}$ refers to the kernel type: Gaussian, inverse multiquadric, Wendland $\mathcal{C}^2$, Wendland $\mathcal{C}^4$, Mat\'ern $\mathcal{C}^2$ or Mat\'ern $\mathcal{C}^4$  respectively, as specified in Table \ref{tabla1nucleos}.

Let $\{z_j\}_{j=0}^{2^{l}}$ be a set of evaluation points in the interval $[0,1]$ defined by $z_j = \frac{j}{2^l}$. For each level $l$, we compute the interpolation $\mathcal{I}^l(z_j)$ using both classical and data-dependent RBF methods. The pointwise error is given by
\[
e^l_j = \left| f(z_j) - \mathcal{I}^l(z_j) \right|.
\]
We define the maximum absolute error 
as follows:
\begin{equation}
\label{eq:error_rate}
\begin{aligned}
\text{E}_l &= \max_{0 \leq j \leq 2^l} e^l_j.
\end{aligned}
\end{equation}

We adopt the following notation for the interpolation operators:
\begin{itemize}
    \item $\textsc{RBF}_{\mathcal{H}}$: classical RBF method and kernel $\mathcal{H}$.
    \item $\textsc{DD-RBF}_{\mathcal{H}}$: data-dependent variant of the above.
\end{itemize}
In the case of the classical methods, the shape parameter $\varepsilon^l$ is fixed for all $l$ and has been chosen so that the interpolation matrix is not close to singular and provides good accuracy results. In this experiment, we use the following shape parameters either for gridded data or Halton points:
\[
\varepsilon^l = 
\begin{cases}
\frac{0.8}{h_l} & \text{if } \mathcal{H} = \text{G or IMQ}, \\
\frac{0.1}{h_l} & \text{if } \mathcal{H} = \text{W2, W4, M2 or M4}.
\end{cases}
\]
For the data-dependent methods, the same shape parameter is used, and modified as described in previous sections relying on the numerical smoothness of the data. The centers used for interpolation coincide with the original data points. As mentioned before, we have set $C=10$ and $t=2$ in (\ref{fI}). For the configuration of the Halton set of points, we use the following sentence in Matlab:
\vspace{0.5cm}

\begin{lstlisting}
p = haltonset(1,'Skip',0,'Leap',38);
\end{lstlisting}
\vspace{0.5cm}

Table~\ref{texp1} presents the observed errors in the infinity norm for uniform grid spacing and Table \ref{texp2} for Halton points. In both cases, we can see that the performance of classical and data-dependent methods is similar at smooth zones. The condition number $\kappa$ of the interpolation matrices in (\ref{sistema})  is also presented in the tables, and we can see that they are similar for classical and data-dependent approaches.

\begin{table}[!ht]
\begin{center}
\begin{tabular}{lccccccccccc}
& \multicolumn{2}{c}{RBF$_{\text{G}}$} & &   \multicolumn{2}{c}{DD-RBF$_{\text{G}}$}\\ \cline{1-3} \cline{5-7} $l$ & $\text{E}_l$  &$\kappa$ & &$\text{E}_l$ &  $\kappa$ \\
7 &1.6081e-06 & 2.3617e+01 && 1.6081e-06 & 2.3617e+01 && \\
8 &1.6037e-06 & 2.3621e+01 && 1.6037e-06 & 2.3621e+01 && \\
9 &1.6027e-06 & 2.3621e+01 && 1.6027e-06 & 2.3621e+01 && \\
10 &1.6024e-06 & 2.3622e+01 && 1.6024e-06 & 2.3622e+01 && \\
\hline
\hline
& \multicolumn{2}{c}{RBF$_{\text{IMQ}}$} & &   \multicolumn{2}{c}{DD-RBF$_{\text{IMQ}}$}\\ \cline{1-3} \cline{5-7} $l$ & $\text{E}_l$ & $\kappa$ & &$\text{E}_l$ & $\kappa$ \\
7 &3.1180e-04 & 2.2682e+02 && 3.1180e-04 & 2.2682e+02 && \\
8 &2.6783e-04 & 2.5543e+02 && 2.6783e-04 & 2.5543e+02 && \\
9 &2.3499e-04 & 2.8403e+02 && 2.3499e-04 & 2.8403e+02 && \\
10 &2.0946e-04 & 3.1264e+02 && 2.0946e-04 & 3.1264e+02 && \\
\hline
\hline
& \multicolumn{2}{c}{RBF$_{\text{W2}}$} & &   \multicolumn{2}{c}{DD-RBF$_{\text{W2}}$}\\ \cline{1-3} \cline{5-7} $l$ & $\text{E}_l$ & $\kappa$ & &$\text{E}_l$ & $\kappa$ \\
7 &1.8735e-04 & 1.3331e+03 && 1.8735e-04 & 1.3331e+03 && \\
8 &1.8720e-04 & 1.3333e+03 && 1.8720e-04 & 1.3333e+03 && \\
9 &1.8716e-04 & 1.3334e+03 && 1.8716e-04 & 1.3334e+03 && \\
10 &1.8715e-04 & 1.3334e+03 && 1.8715e-04 & 1.3334e+03 && \\
\hline
\hline
& \multicolumn{2}{c}{RBF$_{\text{W4}}$} & &   \multicolumn{2}{c}{DD-RBF$_{\text{W4}}$}\\ \cline{1-3} \cline{5-7} $l$ & $\text{E}_l$ & $\kappa$ & &$\text{E}_l$ & $\kappa$ \\
7 &7.8142e-06 & 8.1500e+03 && 7.8142e-06 & 8.1500e+03 && \\
8 &7.8080e-06 & 8.1523e+03 && 7.8080e-06 & 8.1523e+03 && \\
9 &7.8065e-06 & 8.1529e+03 && 7.8065e-06 & 8.1529e+03 && \\
10 &7.8061e-06 & 8.1530e+03 && 7.8061e-06 & 8.1530e+03 && \\
\hline
\hline
& \multicolumn{2}{c}{RBF$_{\text{M2}}$} & &   \multicolumn{2}{c}{DD-RBF$_{\text{M2}}$}\\ \cline{1-3} \cline{5-7} $l$ & $\text{E}_l$ & $\kappa$ & &$\text{E}_l$ & $\kappa$ \\
7 &5.5208e-07 & 4.7533e+05 && 5.5208e-07 & 4.7533e+05 && \\
8 &5.2753e-07 & 4.7945e+05 && 5.2753e-07 & 4.7945e+05 && \\
9 &5.2158e-07 & 4.8057e+05 && 5.2158e-07 & 4.8057e+05 && \\
10 &5.2011e-07 & 4.8086e+05 && 5.2011e-07 & 4.8086e+05 && \\
\hline
\hline
& \multicolumn{2}{c}{RBF$_{\text{M4}}$} & &   \multicolumn{2}{c}{DD-RBF$_{\text{M4}}$}\\ \cline{1-3} \cline{5-7} $l$ & $\text{E}_l$ & $\kappa$ & &$\text{E}_l$ & $\kappa$ \\
7 &1.4238e-10 & 4.7319e+08 && 1.4238e-10 & 4.7319e+08 && \\
8 &1.3282e-10 & 4.7921e+08 && 1.3282e-10 & 4.7921e+08 && \\
9 &1.3057e-10 & 4.8087e+08 && 1.3057e-10 & 4.8087e+08 && \\
10 &1.3002e-10 & 4.8131e+08 && 1.3002e-10 & 4.8131e+08 && \\
\hline
\end{tabular}
\end{center}
\caption{Errors and condition numbers using classical and data-dependent RBF methods for the test function $f(x)=1+\sin(\pi x)$ in the interval $x\in[0,1]$ with a uniform grid.}\label{texp1}
\end{table}

\begin{table}[!ht]
\begin{center}
\begin{tabular}{lccccccccccc}
& \multicolumn{2}{c}{RBF$_{\text{G}}$} & &   \multicolumn{2}{c}{DD-RBF$_{\text{G}}$}\\ \cline{1-3} \cline{5-7} $l$ & $\text{E}_l$  &$\kappa$ & &$\text{E}_l$ &  $\kappa$ \\
7 &4.3581e-09 & 1.9497e+10 && 4.3581e-09 & 1.9497e+10 && \\
8 &2.8719e-10 & 1.0023e+12 && 2.8719e-10 & 1.0023e+12 && \\
9 &2.6317e-12 & 2.8770e+14 && 2.6317e-12 & 2.8770e+14 && \\
10 &2.1210e-12 & 4.0384e+14 && 2.1210e-12 & 4.0384e+14 && \\
\hline
\hline
& \multicolumn{2}{c}{RBF$_{\text{IMQ}}$} & &   \multicolumn{2}{c}{DD-RBF$_{\text{IMQ}}$}\\ \cline{1-3} \cline{5-7} $l$ & $\text{E}_l$ & $\kappa$ & &$\text{E}_l$ & $\kappa$ \\
7 &1.5617e-04 & 3.2110e+06 && 1.5617e-04 & 3.2110e+06 && \\
8 &1.4711e-04 & 9.0396e+06 && 1.4711e-04 & 9.0396e+06 && \\
9 &1.0478e-04 & 3.2937e+07 && 1.0478e-04 & 3.2937e+07 && \\
10 &7.3167e-05 & 3.7340e+07 && 7.3167e-05 & 3.7340e+07 && \\
\hline
\hline
& \multicolumn{2}{c}{RBF$_{\text{W2}}$} & &   \multicolumn{2}{c}{DD-RBF$_{\text{W2}}$}\\ \cline{1-3} \cline{5-7} $l$ & $\text{E}_l$ & $\kappa$ & &$\text{E}_l$ & $\kappa$ \\
7 &1.4794e-04 & 4.4056e+05 && 1.4794e-04 & 4.4056e+05 && \\
8 &1.9051e-04 & 6.2437e+05 && 1.9051e-04 & 6.2437e+05 && \\
9 &2.2980e-04 & 9.5161e+05 && 2.2980e-04 & 9.5161e+05 && \\
10 &1.9139e-04 & 9.5167e+05 && 1.9139e-04 & 9.5167e+05 && \\
\hline
\hline
& \multicolumn{2}{c}{RBF$_{\text{W4}}$} & &   \multicolumn{2}{c}{DD-RBF$_{\text{W4}}$}\\ \cline{1-3} \cline{5-7} $l$ & $\text{E}_l$ & $\kappa$ & &$\text{E}_l$ & $\kappa$ \\
7 &5.5502e-06 & 3.1982e+07 && 5.5502e-06 & 3.1982e+07 && \\
8 &7.0760e-06 & 5.3908e+07 && 7.0760e-06 & 5.3908e+07 && \\
9 &7.6237e-06 & 1.0139e+08 && 7.6237e-06 & 1.0139e+08 && \\
10 &6.3418e-06 & 1.0141e+08 && 6.3418e-06 & 1.0141e+08 && \\
\hline
\hline
& \multicolumn{2}{c}{RBF$_{\text{M2}}$} & &   \multicolumn{2}{c}{DD-RBF$_{\text{M2}}$}\\ \cline{1-3} \cline{5-7} $l$ & $\text{E}_l$ & $\kappa$ & &$\text{E}_l$ & $\kappa$ \\
7 &5.5926e-07 & 1.5037e+08 && 5.5926e-07 & 1.5037e+08 && \\
8 &5.7734e-07 & 2.2060e+08 && 5.7734e-07 & 2.2060e+08 && \\
9 &6.5517e-07 & 3.4031e+08 && 6.5517e-07 & 3.4031e+08 && \\
10 &5.3433e-07 & 3.4183e+08 && 5.3433e-07 & 3.4183e+08 && \\
\hline
\hline
& \multicolumn{2}{c}{RBF$_{\text{M4}}$} & &   \multicolumn{2}{c}{DD-RBF$_{\text{M4}}$}\\ \cline{1-3} \cline{5-7} $l$ & $\text{E}_l$ & $\kappa$ & &$\text{E}_l$ & $\kappa$ \\
7 &1.5038e-10 & 1.7881e+12 && 1.5038e-10 & 1.7881e+12 && \\
8 &1.2085e-10 & 3.1676e+12 && 1.2085e-10 & 3.1676e+12 && \\
9 &1.3127e-10 & 6.0655e+12 && 1.3127e-10 & 6.0655e+12 && \\
10 &1.0592e-10 & 6.1064e+12 && 1.0592e-10 & 6.1064e+12 && \\
\hline
\end{tabular}
\end{center}
\caption{Errors and condition numbers using classical and data-dependent RBF methods for the test function $f(x)=1+\sin(\pi x)$ in the interval $x\in[0,1]$ using Halton points.}\label{texp2}
\end{table}

\subsection{Mitigation of the oscillations close to jump discontinuities in the function for univariate data}
In this subsection, we perform an experiment using the piecewise smooth function:
\begin{equation}\label{funciong}
g(x)=\begin{cases}
\sin(\pi x), & x\leq  2/3, \\
1-\sin(\pi x), & x> 2/3.
\end{cases}
\end{equation}
As in the previous section, the domain of interest is the interval $[0, 1]$. In this case, we have performed the approximation directly in this interval. We use a set of $N = 32$ initial nodes, either uniform or Halton points. From these, we generate 10 intermediate evaluation points between each pair of consecutive initial nodes. 

As in the previous section, the shape parameter $\varepsilon$ is fixed and has been chosen so that the interpolation matrix is not close to singular. For the experiments conducted with gridded data, shown in Figure~\ref{exp1}, we use the following shape parameters:
\[
\varepsilon = 
\begin{cases}
\frac{0.5}{h} & \text{if } \mathcal{H} = \text{G or IMQ}, \\
\frac{0.1}{h} & \text{if } \mathcal{H} = \text{W2, W4,M2 or M4}.
\end{cases}
\]
Similarly, for the experiments with non-uniform data using Halton points, presented in Figure~\ref{exp2}, we apply the following shape parameter settings:
\[
\varepsilon = 
\begin{cases}
\frac{0.8}{h} & \text{if } \mathcal{H} = \text{G or IMQ}, \\
\frac{0.1}{h} & \text{if } \mathcal{H} = \text{W2, W4, M2 or M4}.
\end{cases}
\]
For the data-dependent methods, the same shape parameter is used, and modified as described in previous sections relying on the numerical smoothness of the data. As before, the centers used for interpolation coincide with the original data points and we have set $C=10$ and $t=2$ in (\ref{fI}). The configuration of the Halton points is the same as before.

Figure~\ref{exp1} presents the results obtained using both the classical and data-dependent RBF interpolation algorithms across all kernel types listed in Table~\ref{tabla1nucleos}. As mentioned before, these experiments were conducted using gridded data. In each plot, the original data centers are marked with solid red dots, while the output of the classical algorithm is shown with a solid blue line, following the notation specified in Table~\ref{tabla1nucleos} for the legend of each plot. The results of the data-dependent algorithm are displayed with a solid black line. Across all cases, it is evident that the oscillations present in the classical interpolation near discontinuities are significantly reduced when using the data-dependent approach. As a trade-off, a slight smoothing of the discontinuity can be observed, although this effect remains mild and does not substantially affect the overall accuracy of the interpolation. Additionally, it can be seen that the interpolation property is no longer strictly preserved near the discontinuity, as the data-dependent algorithms approximate rather than interpolate in those regions. This behavior is expected and was explicitly considered during the design of the adaptive strategy. Table~\ref{tabla_condicion_unif} displays the condition numbers corresponding to the experiments shown in Figure~\ref{exp1}. As observed, the condition numbers for the classical and data-dependent methods are comparable, with the data-dependent methods generally exhibiting slightly lower values.

Figure~\ref{exp2} shows the results of a similar experiment conducted with non-uniform data, specifically using Halton points. As in the case of gridded data, we observe that the data-dependent RBF algorithms effectively reduce the oscillations that appear near discontinuities when using the classical approach. The results are consistent across all kernel types, confirming the robustness of the adaptive strategy. As before, a mild smoothing of the discontinuity is present, which is an expected consequence of the data-dependent formulation. Table~\ref{tabla_condicion_no_unif} reports the condition numbers obtained from the experiments with non-uniform data shown in Figure~\ref{exp2}. As in the uniform case, the condition numbers for the data-dependent methods remain close to those of the classical counterparts, with a slight overall reduction observed in the data-dependent formulation.

	\begin{figure}[htbp!]
\begin{center}
		\begin{tabular}{cc}
	\includegraphics[width=8cm, height=5.15cm]{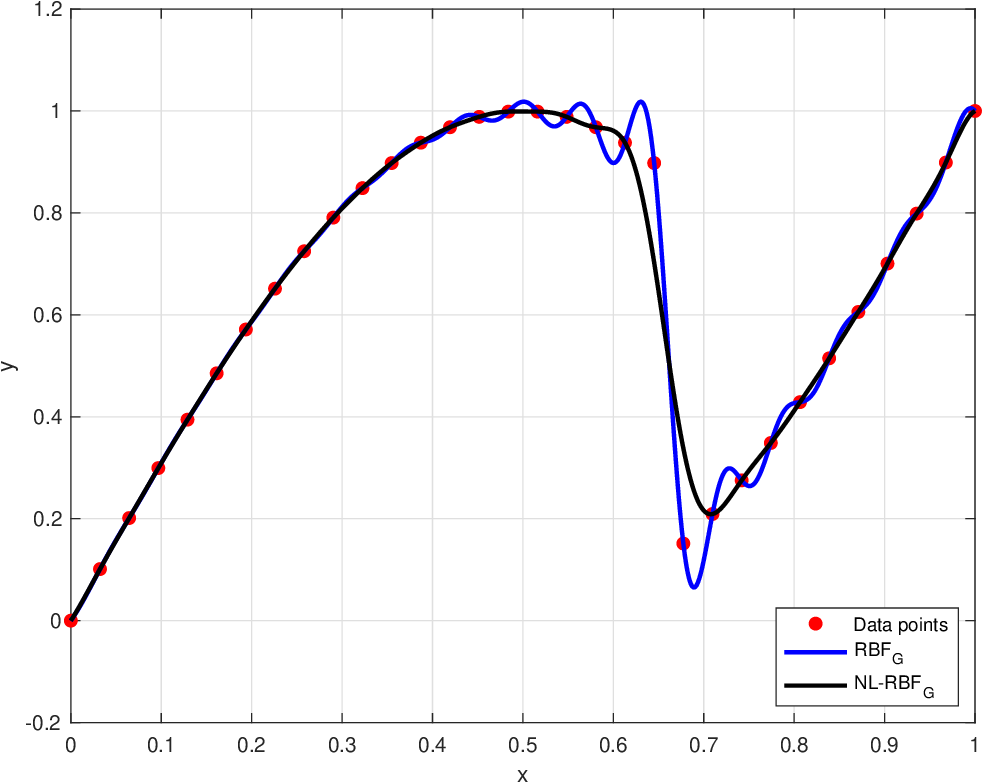} & 	\includegraphics[width=8cm, height=5.15cm]{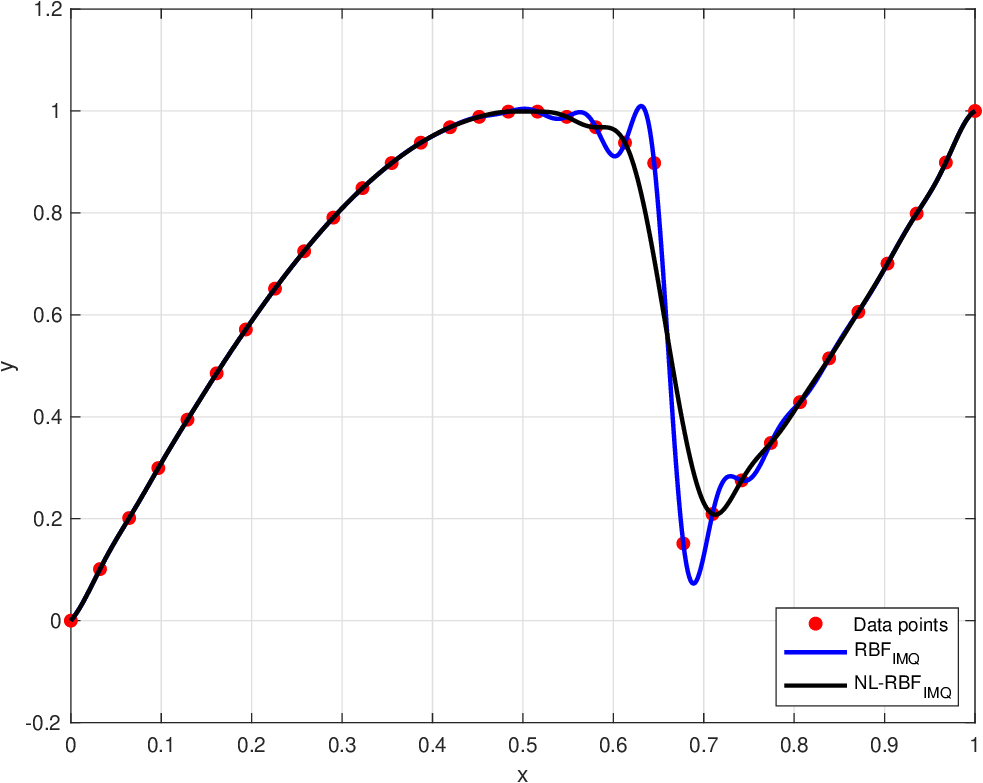}\\
	\includegraphics[width=8cm, height=5.15cm]{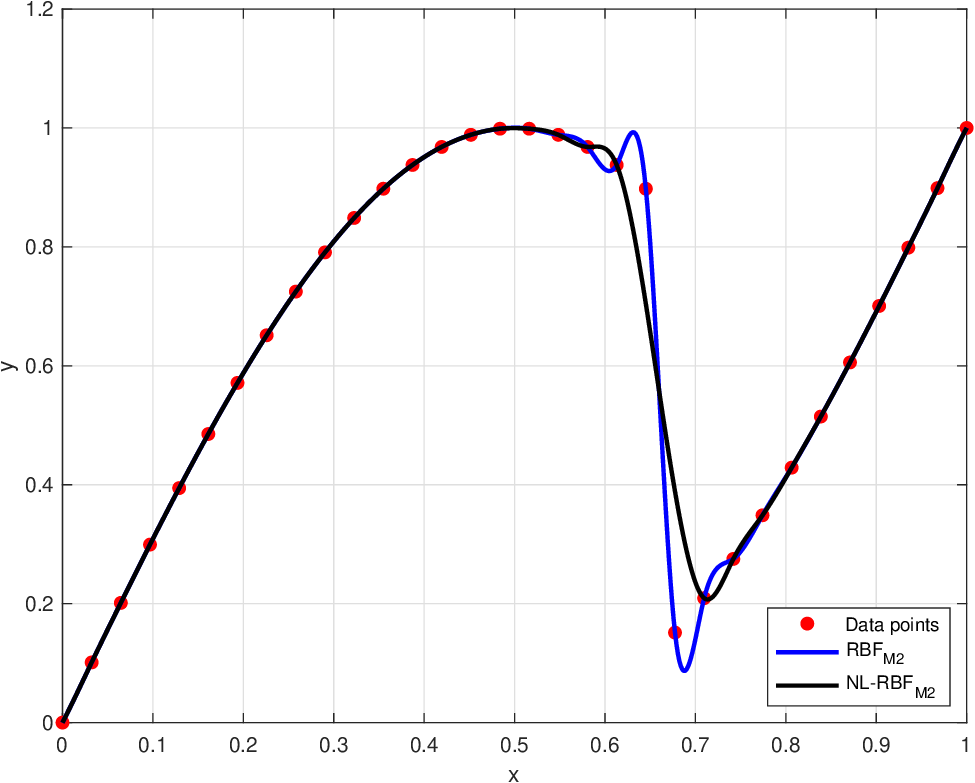} & 	\includegraphics[width=8cm, height=5.15cm]{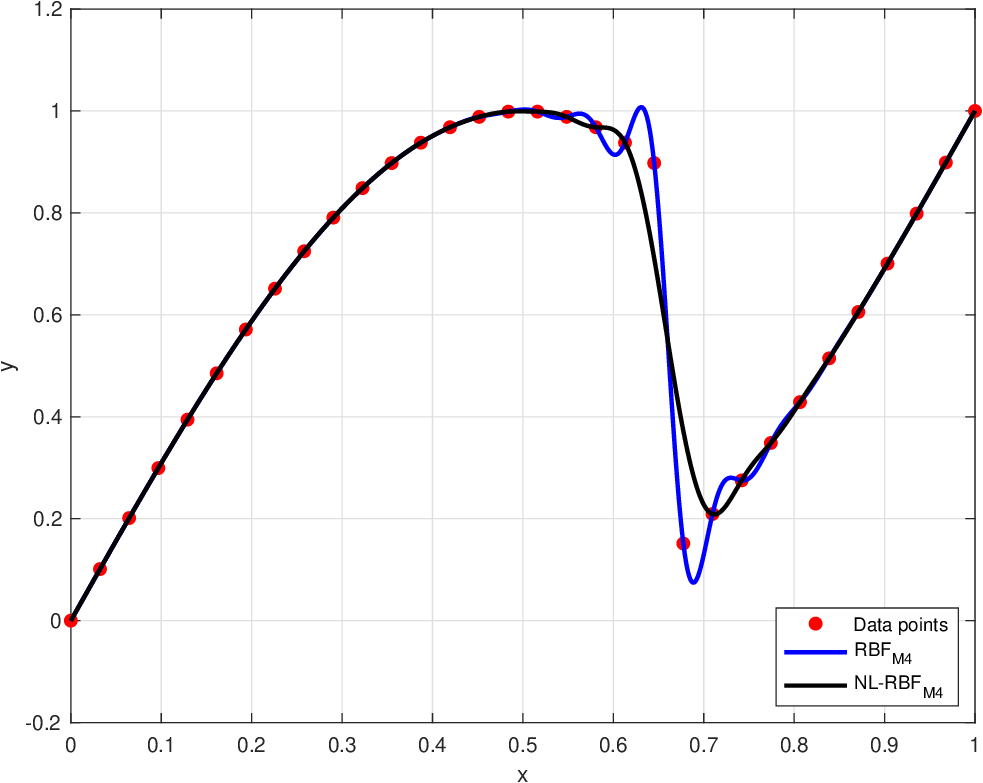}\\
	\includegraphics[width=8cm, height=5.15cm]{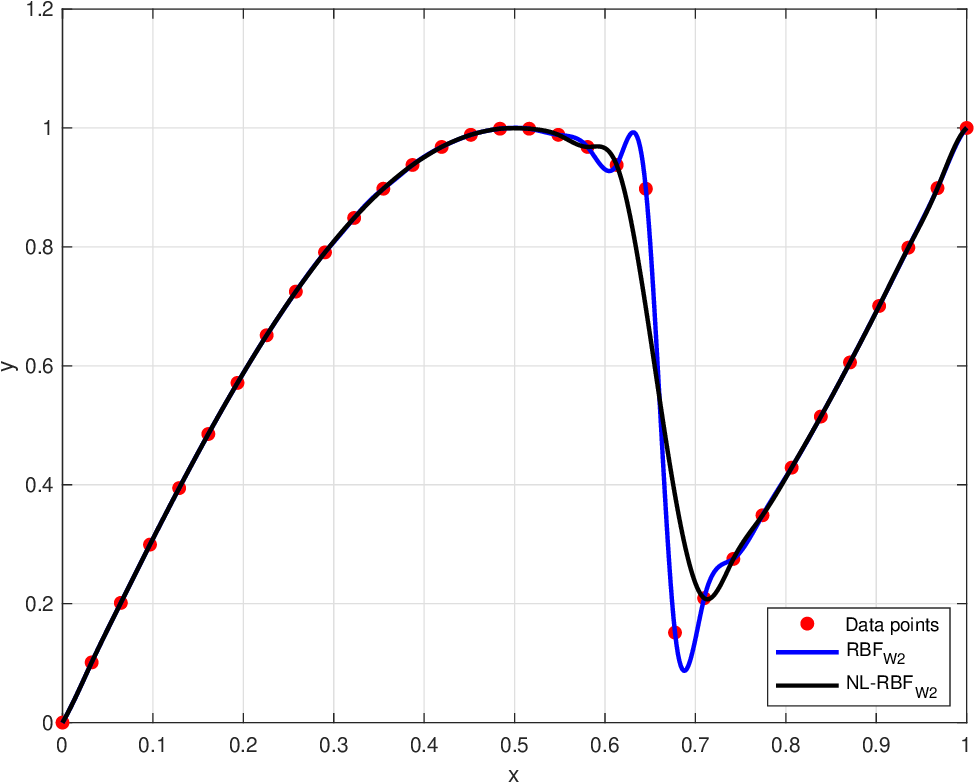} & 	\includegraphics[width=8cm, height=5.15cm]{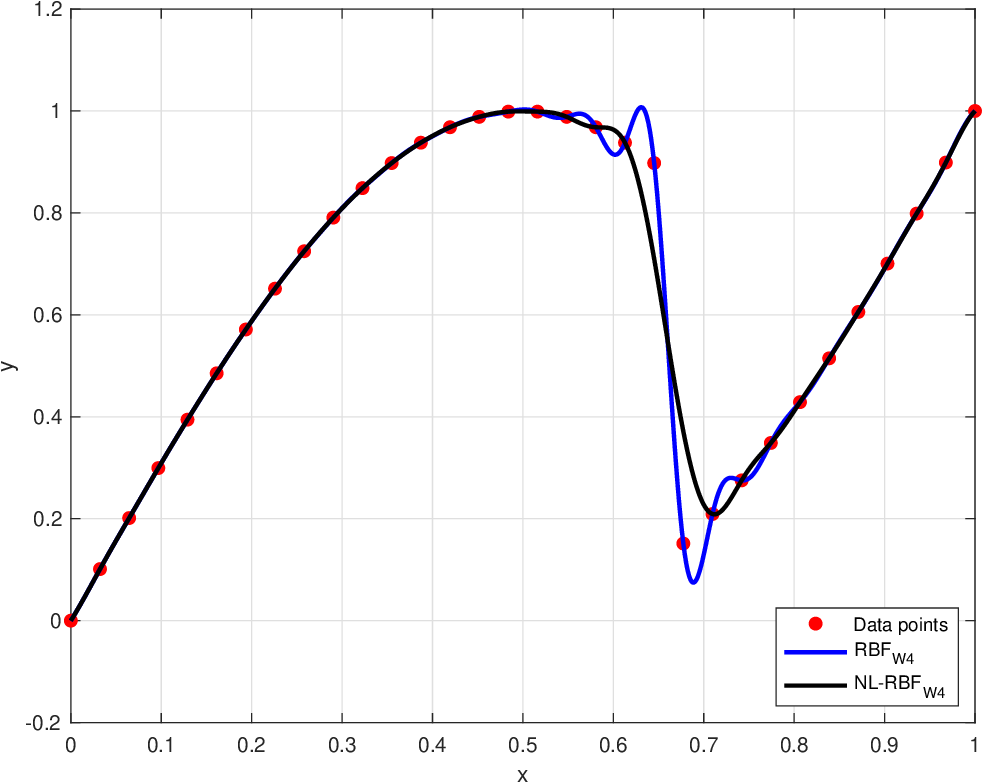}
		\end{tabular}
\end{center}
			\caption{Approximation to the function $g$ (red solid dots), Eq. \eqref{funciong}, using a uniform grid of $n=32$ points and generating ten evaluation points between each of them. In each plot the classical and data-dependent RBF algorithms have been used. From left to right and top to bottom these algorithms are: RBF$_{\text{G}}$ and RBF$_{\text{G}}$, RBF$_{\text{IMQ}}$ and DD-RBF$_{\text{IMQ}}$, RBF$_{\text{W2}}$ and  DD-RBF$_{\text{W2}}$,  RBF$_{\text{W4}}$ and  DD-RBF$_{\text{W4}}$,  RBF$_{\text{M2}}$ and DD-RBF$_{\text{M2}}$, RBF$_{\text{M4}}$ and DD-RBF$_{\text{M4}}$.}
		\label{exp1}
	\end{figure}

\begin{table}[ht]
\centering
\begin{tabular}{|l|c|c|}
\hline
\textbf{Kernel} & $\kappa$ \textbf{(Classical)} &$\kappa$ \textbf{(data-dependent)} \\\hline
RBF$_{\text{G}}$ & 7.8829e+03 & 5.8410e+03 \\
RBF$_{\text{IMQ}}$ & 1.3707e+03 & 1.2740e+03 \\
RBF$_{\text{W2}}$ & 1.2888e+03 & 1.2368e+03 \\
RBF$_{\text{W4}}$ & 7.7535e+03 & 7.2861e+03 \\
RBF$_{\text{M2}}$ & 2.7476e+05 & 2.6051e+05 \\
RBF$_{\text{M4}}$ & 2.3188e+08 & 2.1584e+08 \\
\hline
\end{tabular}
\caption{Condition numbers $\kappa$ for classical and data-dependent RBF interpolation methods across different kernels for the experiments shown in Figure \ref{exp1}, where the data is uniformly gridded.}
\label{tabla_condicion_unif}
\end{table}

	\begin{figure}[htbp!]
\begin{center}
		\begin{tabular}{cc}
	\includegraphics[width=8cm, height=5.15cm]{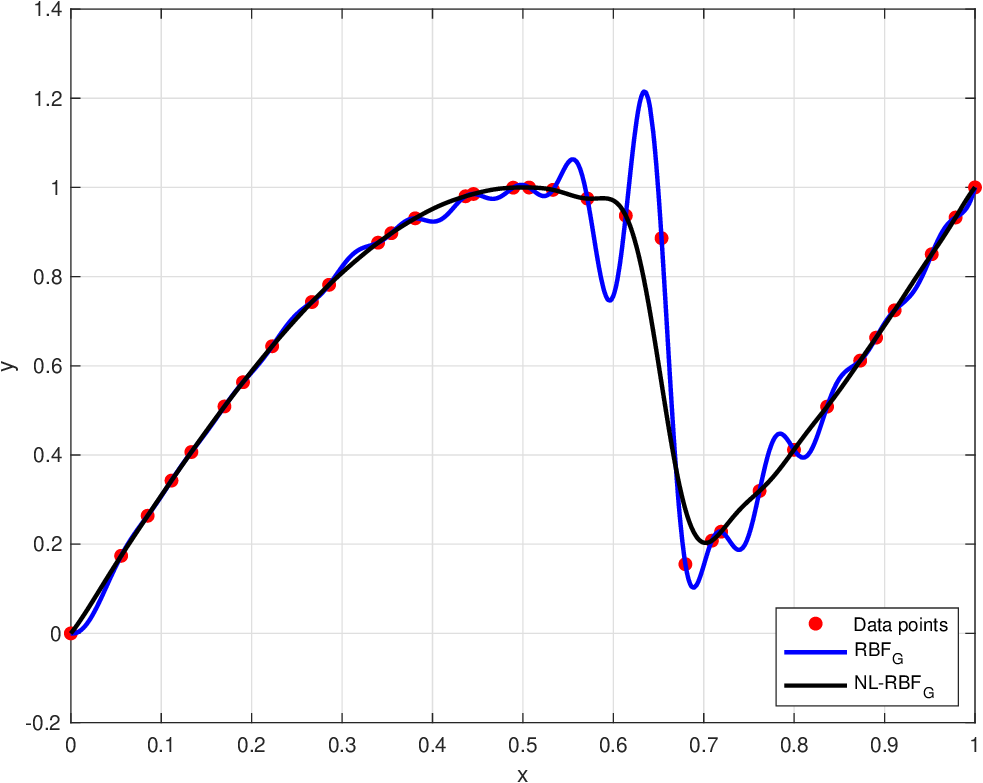} & 	\includegraphics[width=8cm, height=5.15cm]{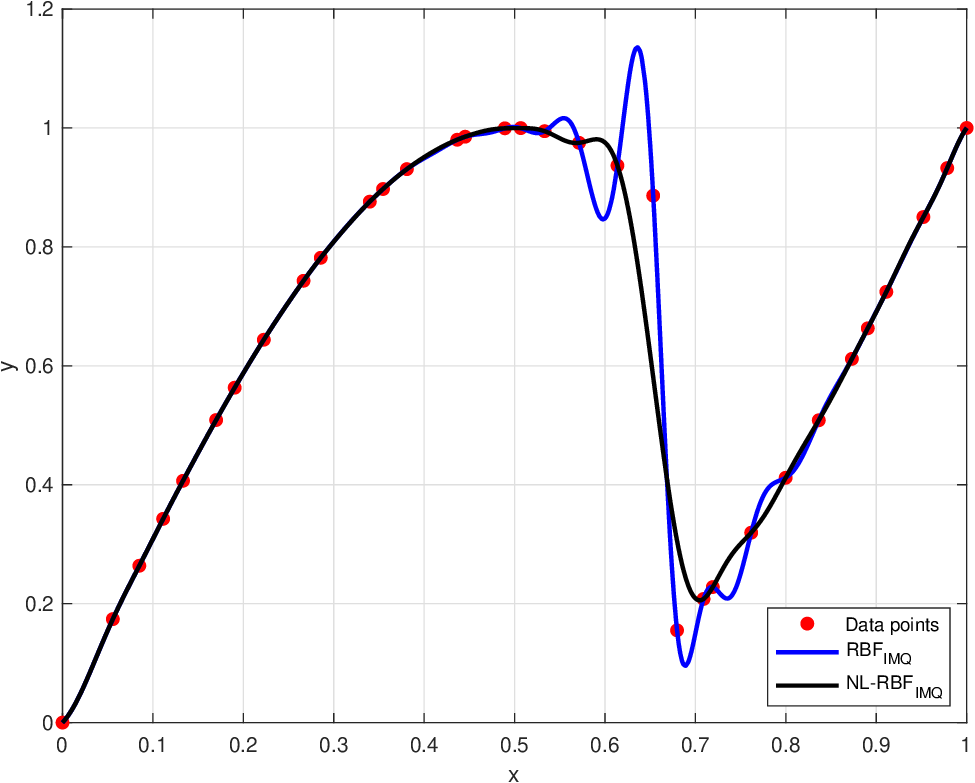}\\
	\includegraphics[width=8cm, height=5.15cm]{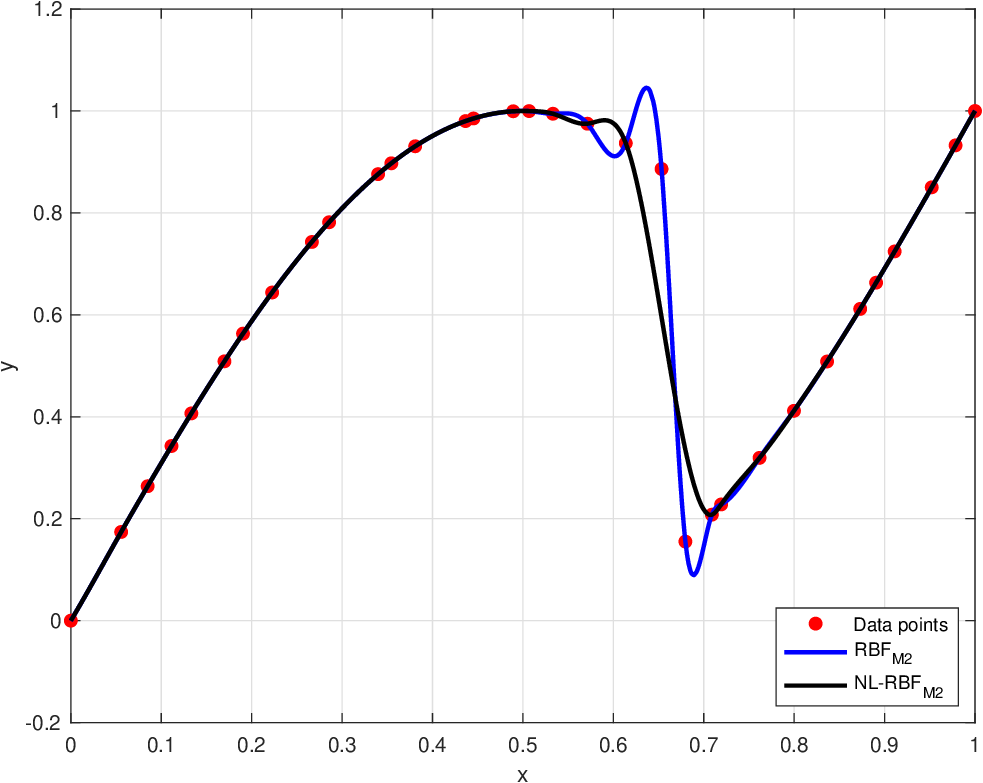} & 	\includegraphics[width=8cm, height=5.15cm]{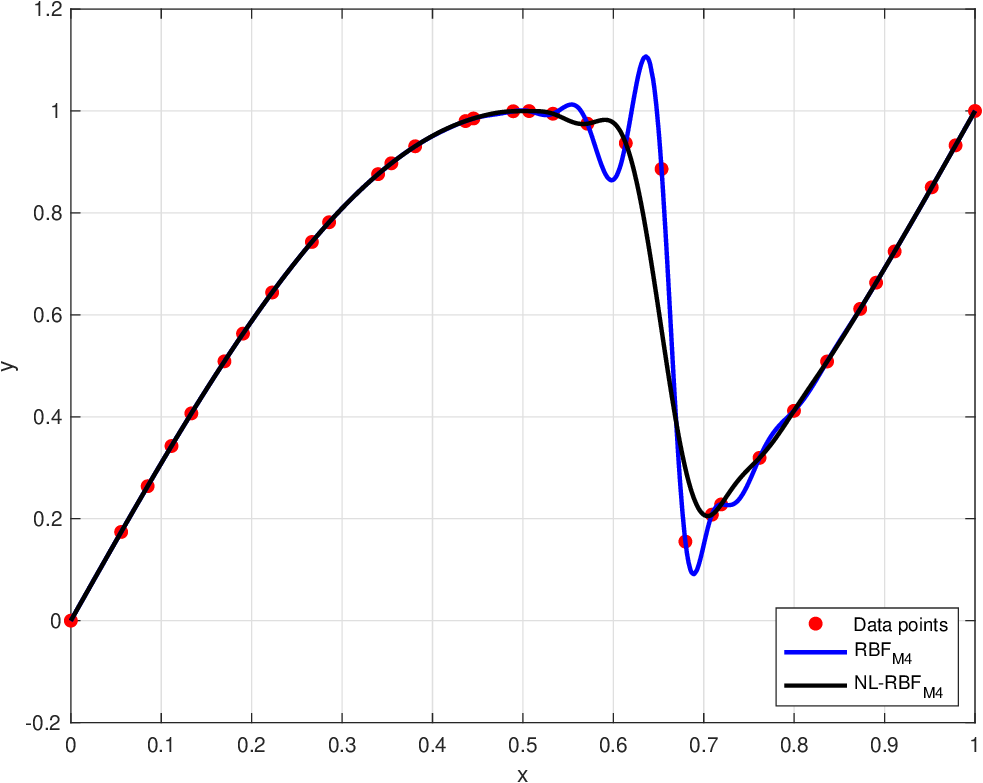}\\
	\includegraphics[width=8cm, height=5.15cm]{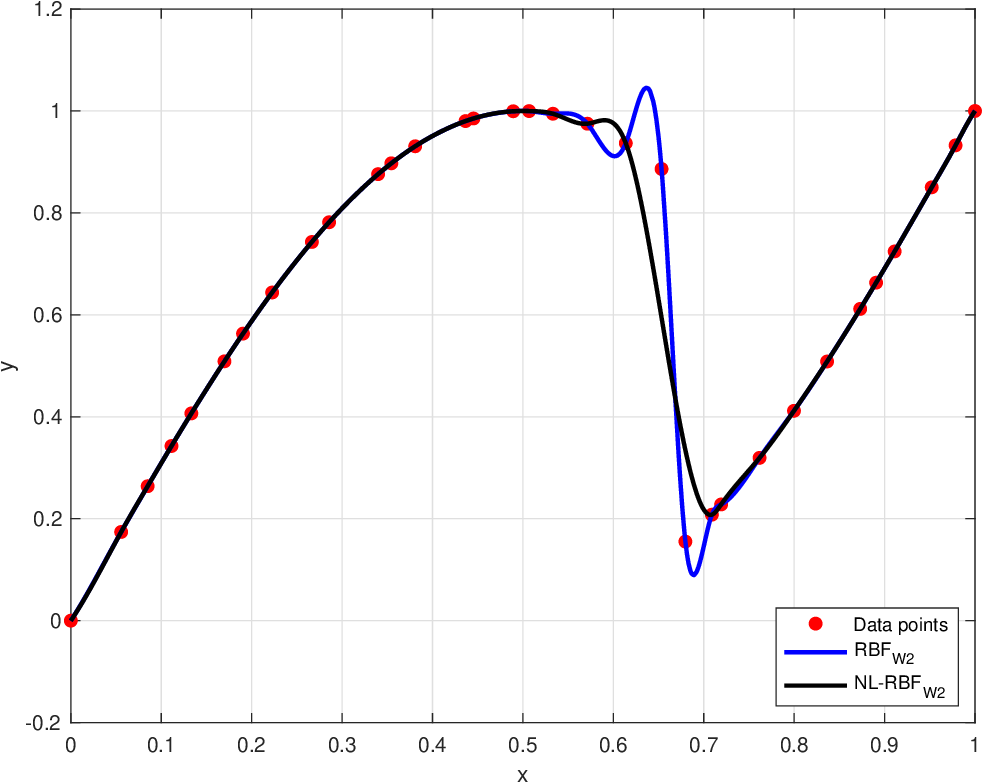} & 	\includegraphics[width=8cm, height=5.15cm]{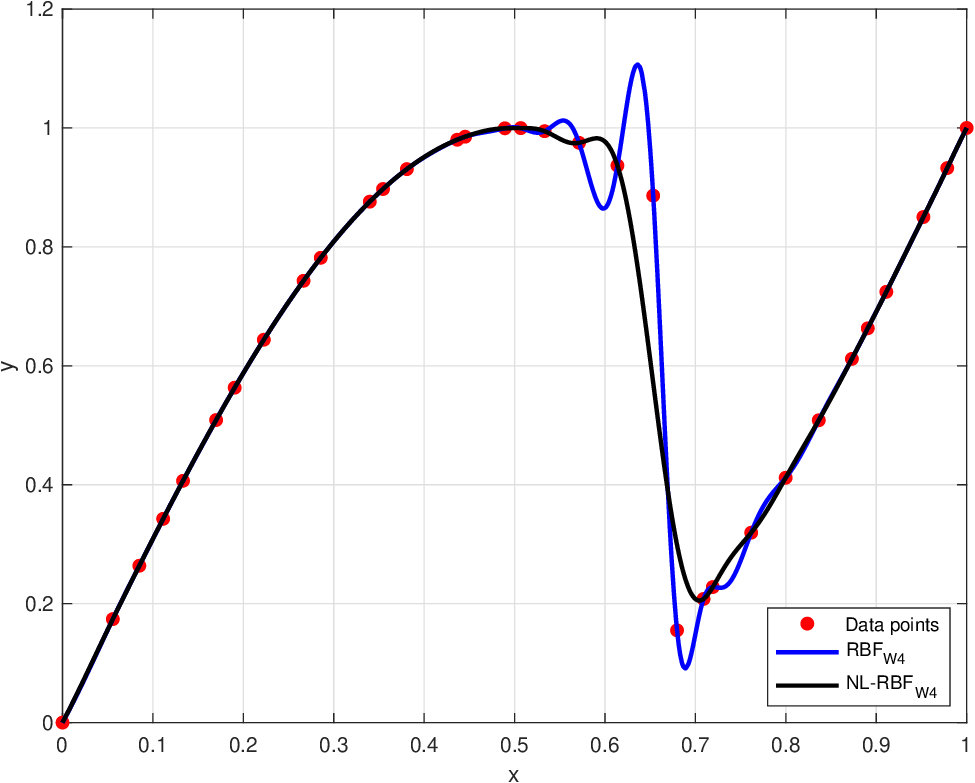}
		\end{tabular}
\end{center}
			\caption{Approximation to the function $g$ (red solid dots), Eq. \eqref{funciong}, using $n=32$ Halton points and generating ten Halton evaluation points between each of them. In each plot the classical and data-dependent RBF algorithms have been used. From left to right and top to bottom these algorithms are: RBF$_{\text{G}}$ and RBF$_{\text{G}}$, RBF$_{\text{IMQ}}$ and DD-RBF$_{\text{IMQ}}$, RBF$_{\text{W2}}$ and  DD-RBF$_{\text{W2}}$,  RBF$_{\text{W4}}$ and  DD-RBF$_{\text{W4}}$,  RBF$_{\text{M2}}$ and DD-RBF$_{\text{M2}}$, RBF$_{\text{M4}}$ and DD-RBF$_{\text{M4}}$.}
		\label{exp2}
	\end{figure}

\begin{table}[ht]
\centering
\begin{tabular}{|l|c|c|}
\hline
\textbf{Kernel} & $\kappa$ \textbf{(Classical)} & $\kappa$ \textbf{(data-dependent)} \\
\hline
RBF$_{\text{G}}$ & 2.6163e+05 & 1.3202e+05 \\
RBF$_{\text{IMQ}}$ & 1.2173e+04 & 9.4582e+03 \\
RBF$_{\text{W2}}$ & 3.9567e+04 & 3.6773e+04 \\
RBF$_{\text{W4}}$ & 8.9213e+05 & 7.3134e+05 \\
RBF$_{\text{M2}}$ & 6.1173e+06 & 5.7182e+06 \\
RBF$_{\text{M4}}$ & 1.8349e+10 & 1.5113e+10 \\
\hline
\end{tabular}
\caption{Condition numbers $\kappa$ for classical and data-dependent RBF interpolation methods across different kernels for the experiments shown in Figure \ref{exp2}, where we have used Halton points.}
\label{tabla_condicion_no_unif}
\end{table}

\subsection{Mitigation of the oscillations close to jump discontinuities in the function for bivariate data}

In this subsection, we analyze the performance of both classical and data-dependent algorithms when applied to piecewise smooth bivariate data. To do so, we use the well-known Franke's function \cite{Franke}, which is commonly employed as a benchmark in surface approximation problems:

\begin{equation}\label{frankesfunction}
\begin{split}
f(x,y) =\;& \frac{3}{4} e^{- \frac{1}{4}((9x - 2)^2 + (9y - 2)^2)} + \frac{3}{4} e^{- \frac{1}{49}(9x + 1)^2 - \frac{1}{10}(9y + 1)} \\
& + \frac{1}{2} e^{- \frac{1}{4}((9x - 7)^2 + (9y - 3)^2)} - \frac{1}{5} e^{-(9x - 4)^2 - (9y - 7)^2}.
\end{split}
\end{equation}
To introduce a jump discontinuity, we modify the function as follows:
\begin{equation}\label{frankesdisc}
f_1(x,y) = \left\{
\begin{array}{ll}
2 + f(x,y), & x^2 + y^2 - 0.3^2 \geq 0, \\
-1 + f(x,y), & x^2 + y^2 - 0.3^2 < 0.
\end{array}
\right.
\end{equation}
As in the previous univariate examples, we consider both gridded data and Halton sequences for sampling. In both cases, the data points are distributed within the square $[0,1]^2$, using an initial set of $50 \times 50$ points.

For the experiments using gridded data, illustrated in Figures \ref{exp1_2D} to \ref{exp6_2D}, we set the shape parameter $\varepsilon$ according to the type of radial basis function $\mathcal{H}$ as follows:
\[
\varepsilon = 
\begin{cases}
\frac{0.5}{h}, & \text{if } \mathcal{H} = \text{G or IMQ}, \\
\frac{0.1}{h}, & \text{if } \mathcal{H} = \text{W2, W4, M2 or M4}.
\end{cases}
\]
Similarly, for the experiments using non-uniform data sampled via Halton points (Figures~\ref{exp7_2D} to~\ref{exp12_2D}), we use the following shape parameter settings:
\[
\varepsilon = 
\begin{cases}
\frac{0.7}{h}, & \text{if } \mathcal{H} = \text{G or IMQ}, \\
\frac{0.1}{h}, & \text{if } \mathcal{H} = \text{W2, W4, M2 or M4}.
\end{cases}
\]

In all the experiments presented in this section (Figures~\ref{exp1_2D} to~\ref{exp12_2D}), we approximate the piecewise smooth function $f_1$ defined in Eq.~\eqref{frankesdisc}, using $n = 50 \times 50$ initial data points, either arranged on a regular grid (Figures \ref{exp1_2D} to \ref{exp6_2D}) or generated via Halton sequences (Figures \ref{exp7_2D} to \ref{exp12_2D}). To perform the approximation, we interpolate at six intermediate sites between each pair of initial data points in both directions, resulting in a total of $344 \times 344$ final evaluation points, which include the original ones. Each figure consists of two columns: the left column shows the results obtained using the classical $\textsc{RBF}_{\mathcal{H}}$ algorithm, while the right column corresponds to the data-dependent $\textsc{DD-RBF}_{\mathcal{H}}$ algorithm, following the notation introduced in Table~\ref{tabla1nucleos}. The first row in each figure displays the final approximation, and the second row shows the error distribution across the domain. For the experiments using gridded data (Figures~\ref{exp1_2D} to~\ref{exp6_2D}), we observe that the classical algorithm consistently introduces severe oscillations along the discontinuity curve. These oscillations are clearly visible in the error plots, where they propagate far from the discontinuity. In contrast, the data-dependent algorithm significantly reduces these oscillations, although some smearing near the discontinuity remains. As shown in Table~\ref{tabla_condicion_unif_2D}, the condition numbers for both algorithms are comparable, with slightly lower values for the data-dependent method.

Similar conclusions can be drawn from the experiments using Halton points (Figures~\ref{exp7_2D} to~\ref{exp12_2D}). In this case, the oscillations produced by the classical algorithm are even more pronounced, while the data-dependent algorithm effectively mitigates them. The condition numbers for these experiments, reported in Table~\ref{tabla_condicion_nounif_2D}, again show similar values for both algorithms, with a slight advantage for the data-dependent approach.


	\begin{figure}[htbp!]
\begin{center}
		\begin{tabular}{cc}
	\includegraphics[width=6cm]{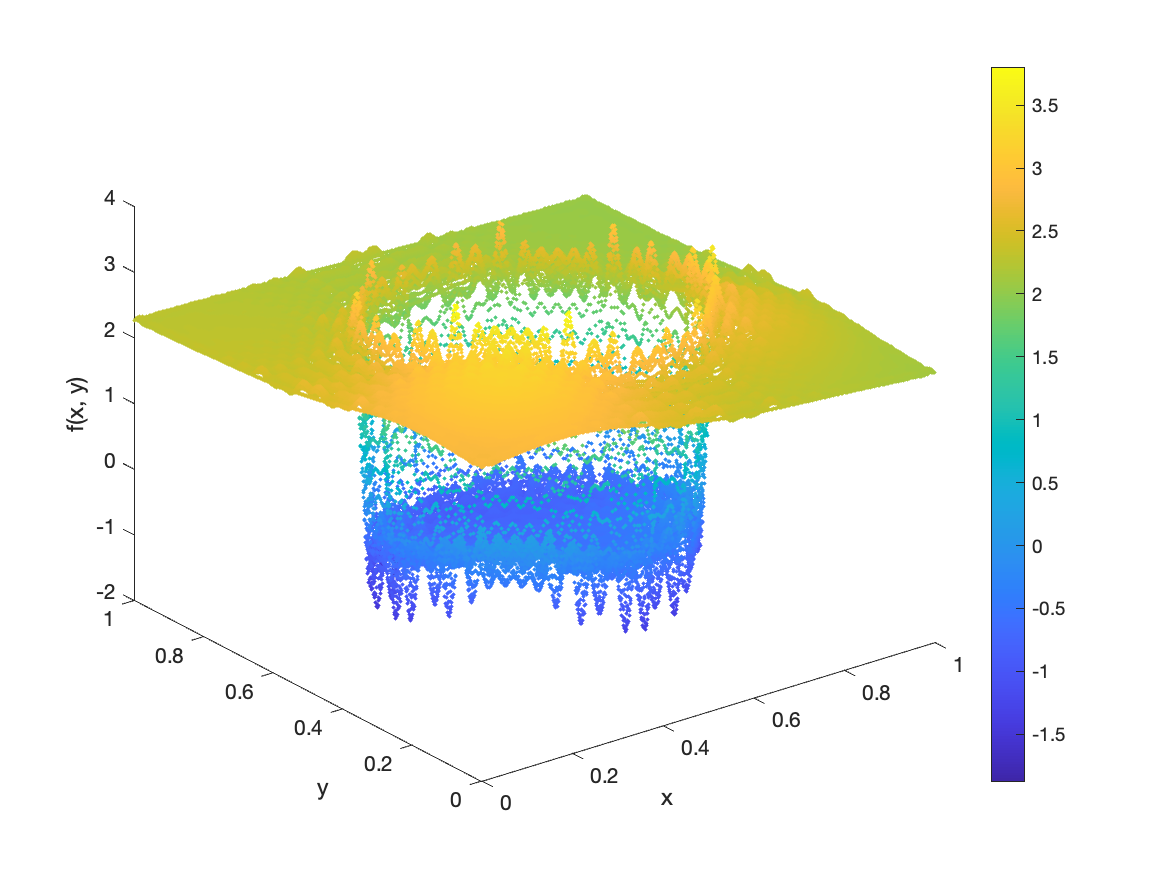} & 	\includegraphics[width=6cm]{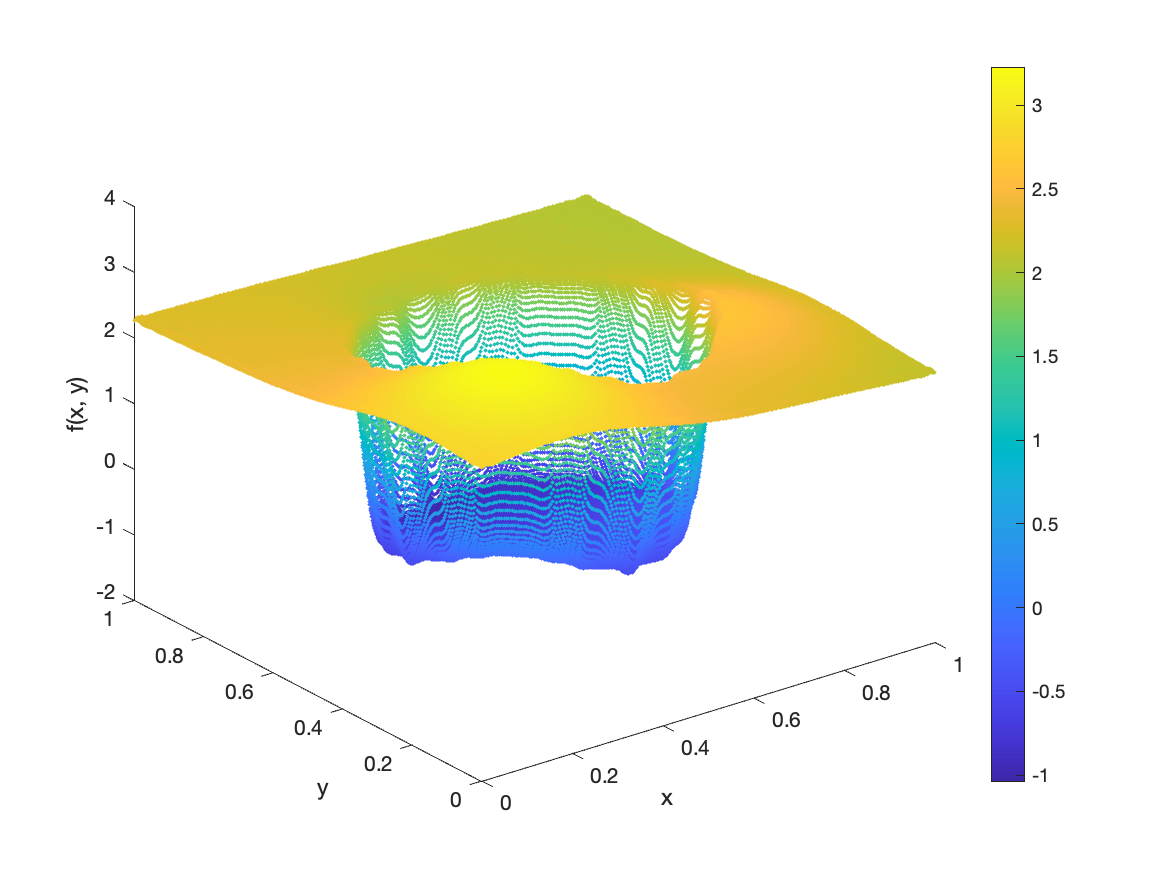}\\
	\includegraphics[width=6cm]{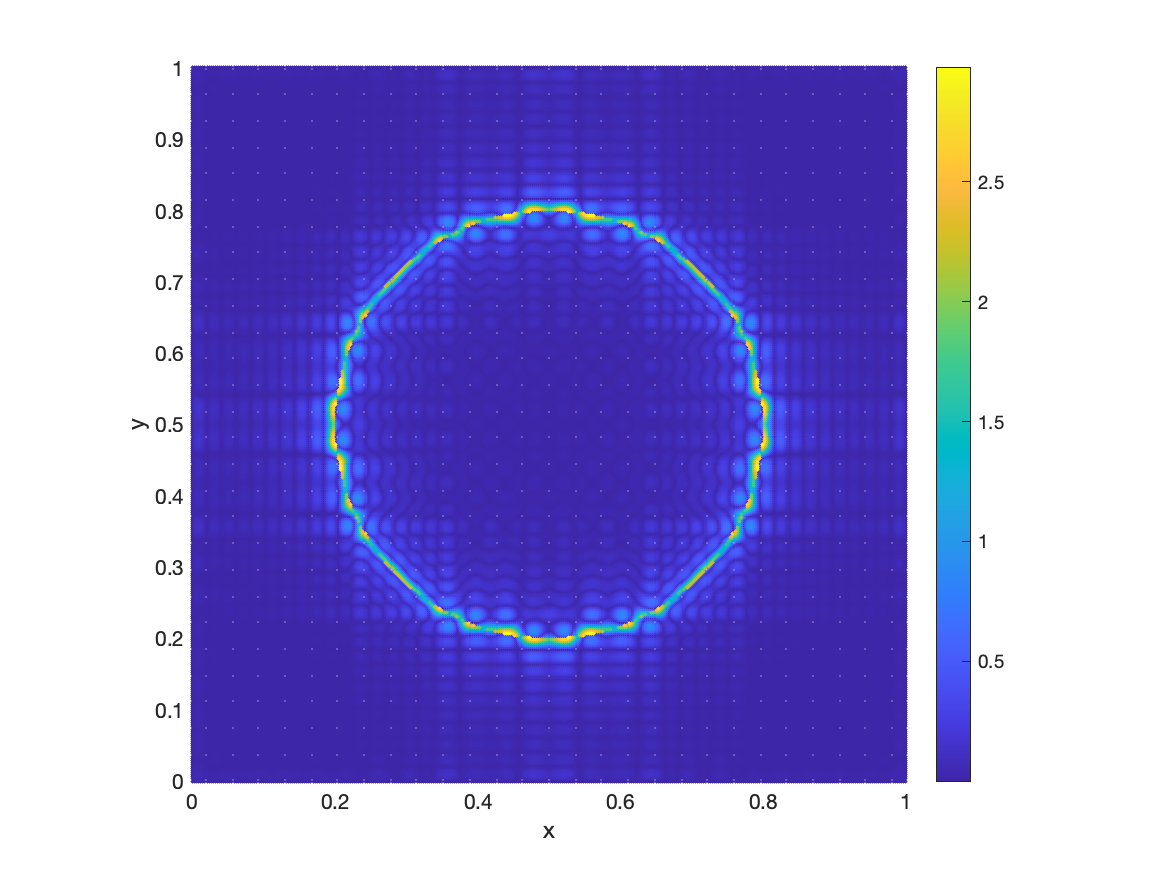} & 	\includegraphics[width=6cm]{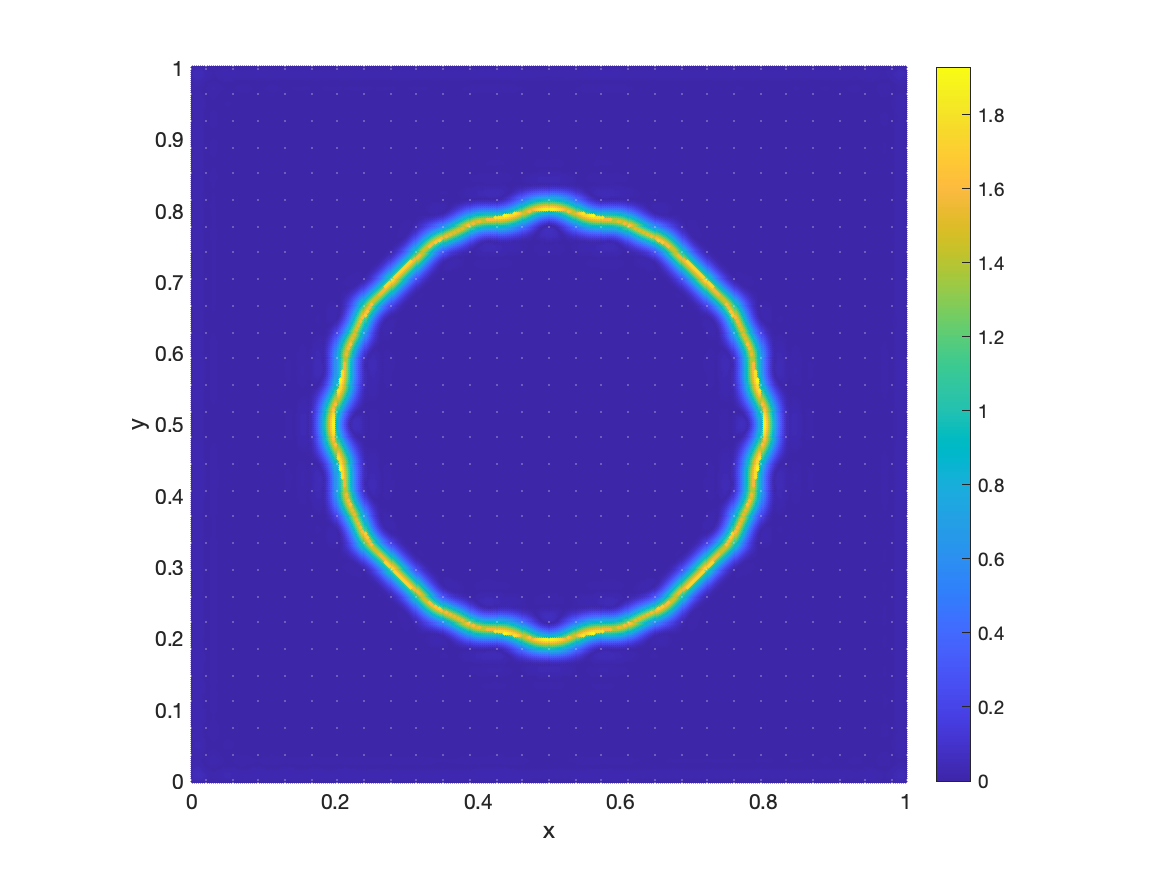}
		\end{tabular}
\end{center}
			\caption{Approximation of the function $f_1$, defined in Eq.~\eqref{frankesdisc}, using $n = 50 \times 50$ initial gridded data points. The left column shows the results obtained with the classical RBF$_{\text{G}}$ algorithm, while the right column corresponds to the data-dependent DD-RBF$_{\text{G}}$ algorithm. The first row displays the final approximations, and the second row presents the error distributions across the domain.}
		\label{exp1_2D}
	\end{figure}

	\begin{figure}[htbp!]
\begin{center}
		\begin{tabular}{cc}
	\includegraphics[width=6cm]{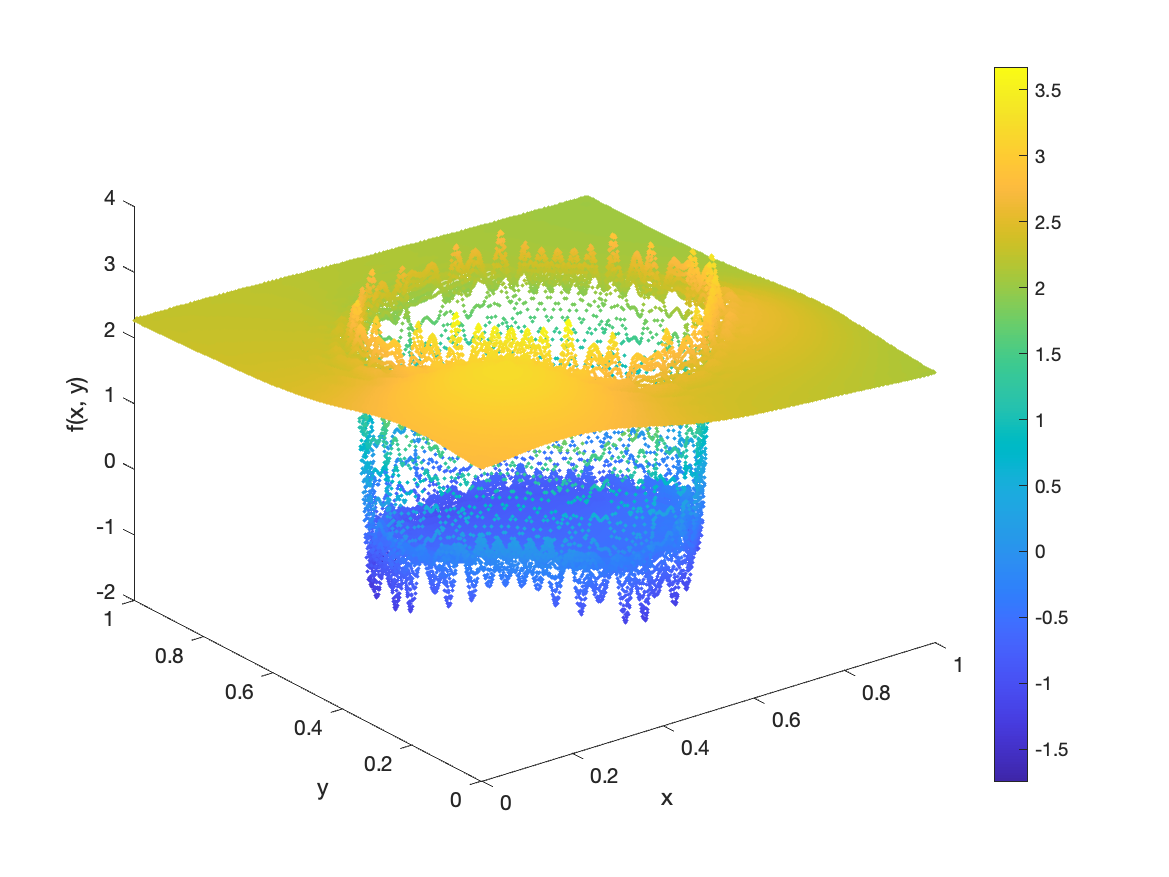} & 	\includegraphics[width=6cm]{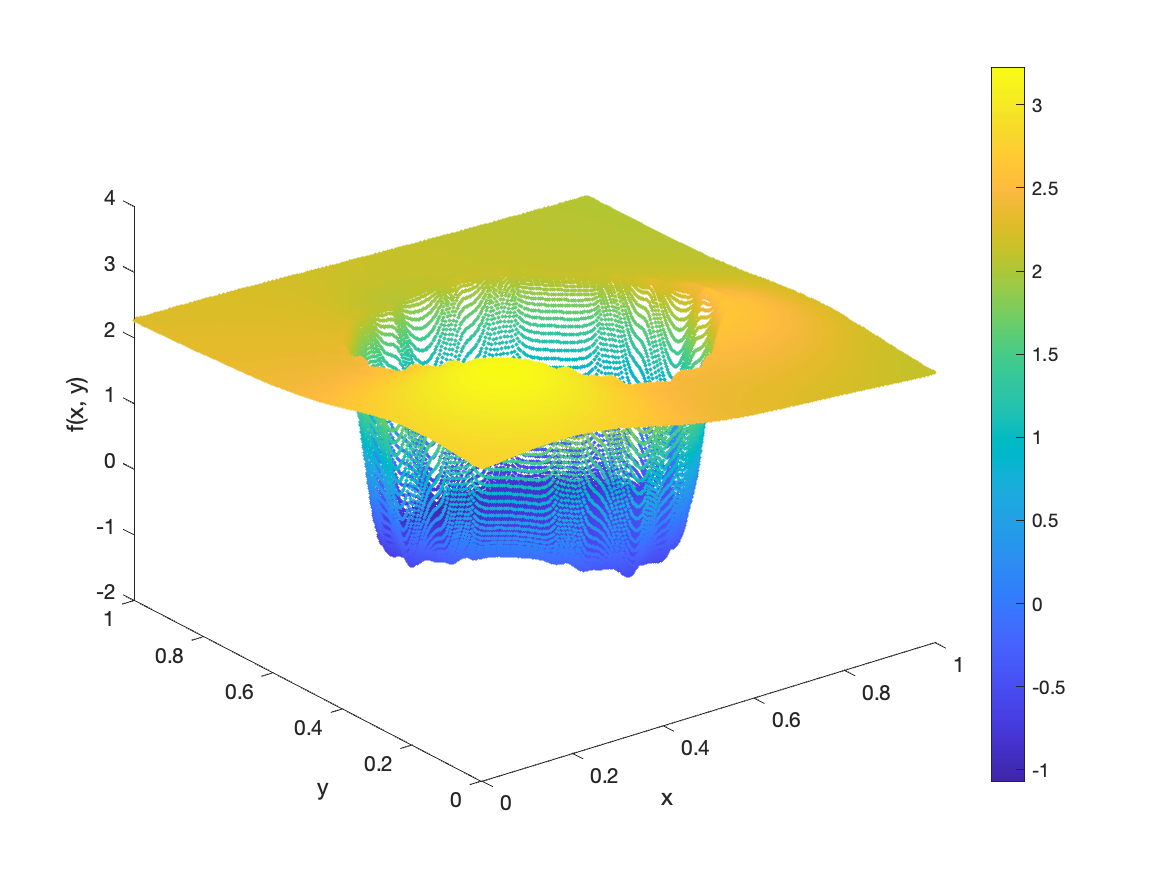}\\
	\includegraphics[width=6cm]{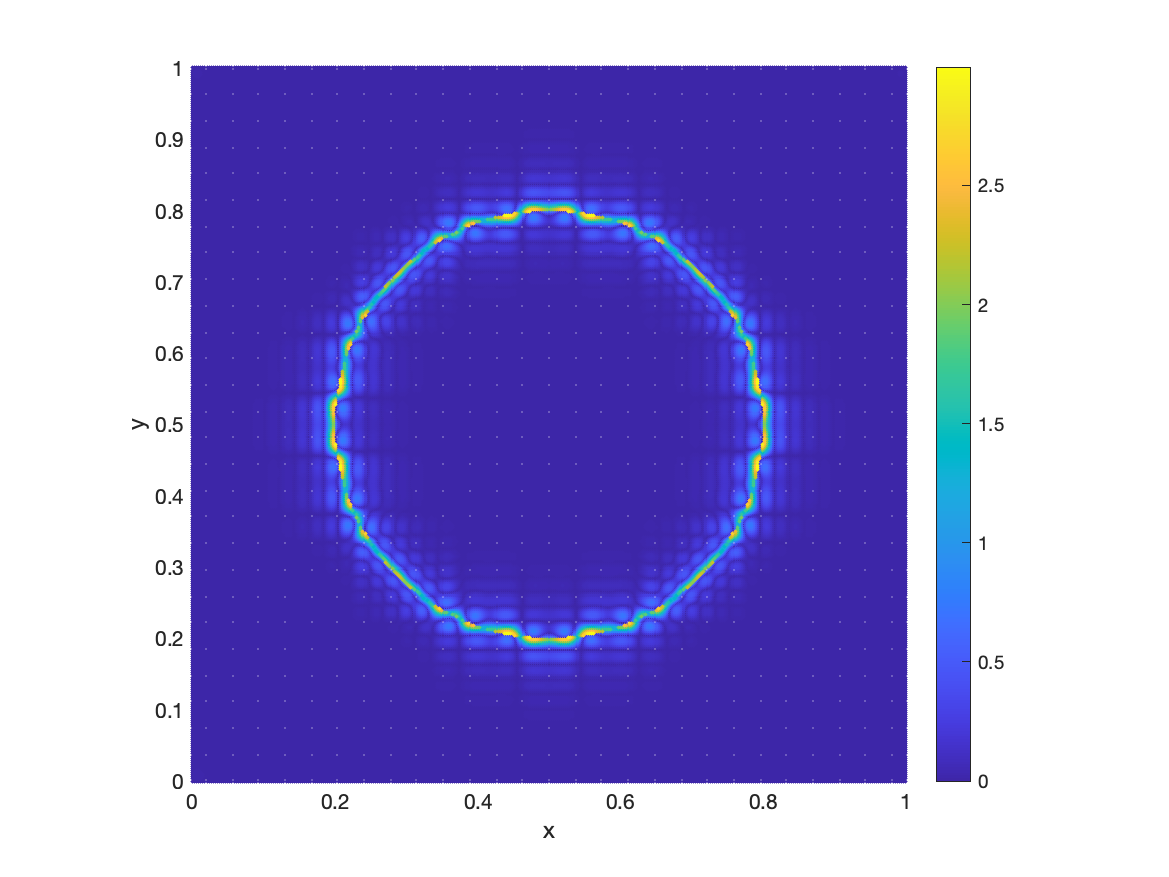} & 	\includegraphics[width=6cm]{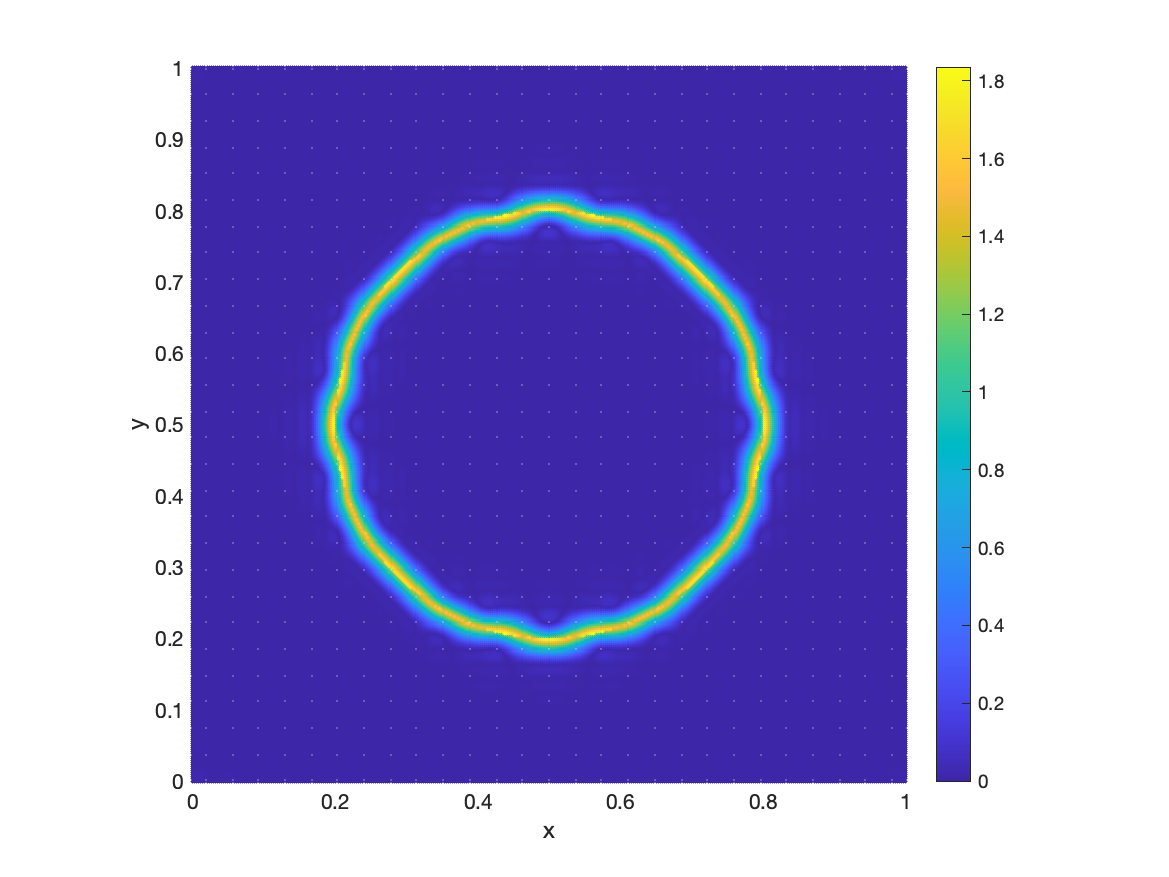}
		\end{tabular}
\end{center}
			\caption{Approximation of the function $f_1$, defined in Eq.~\eqref{frankesdisc}, using $n = 50 \times 50$ initial gridded data points. The left column shows the results obtained with the classical RBF$_{\text{IMQ}}$ algorithm, while the right column corresponds to the data-dependent DD-RBF$_{\text{IMQ}}$ algorithm. The first row displays the final approximations, and the second row presents the error distributions across the domain.}
		\label{exp2_2D}
	\end{figure}

	\begin{figure}[htbp!]
\begin{center}
		\begin{tabular}{cc}
	\includegraphics[width=6cm]{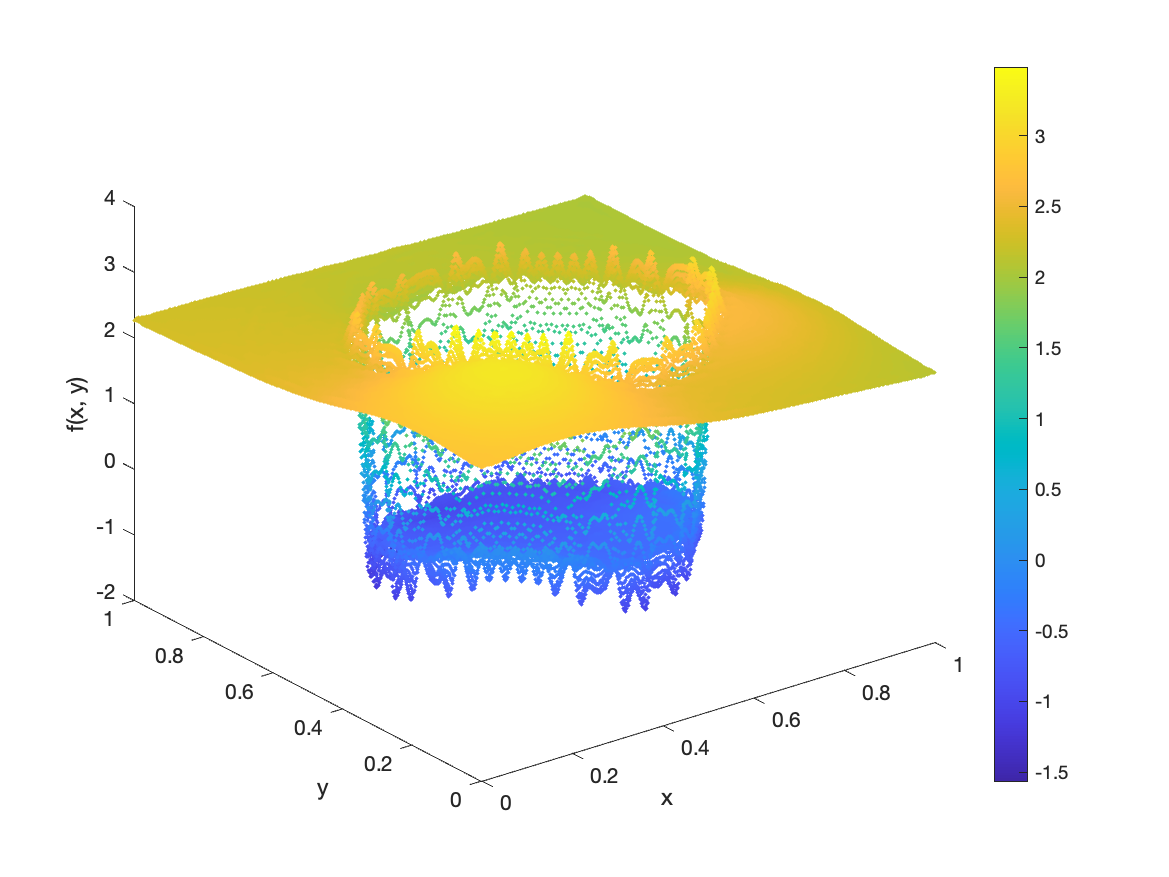} & 	\includegraphics[width=6cm]{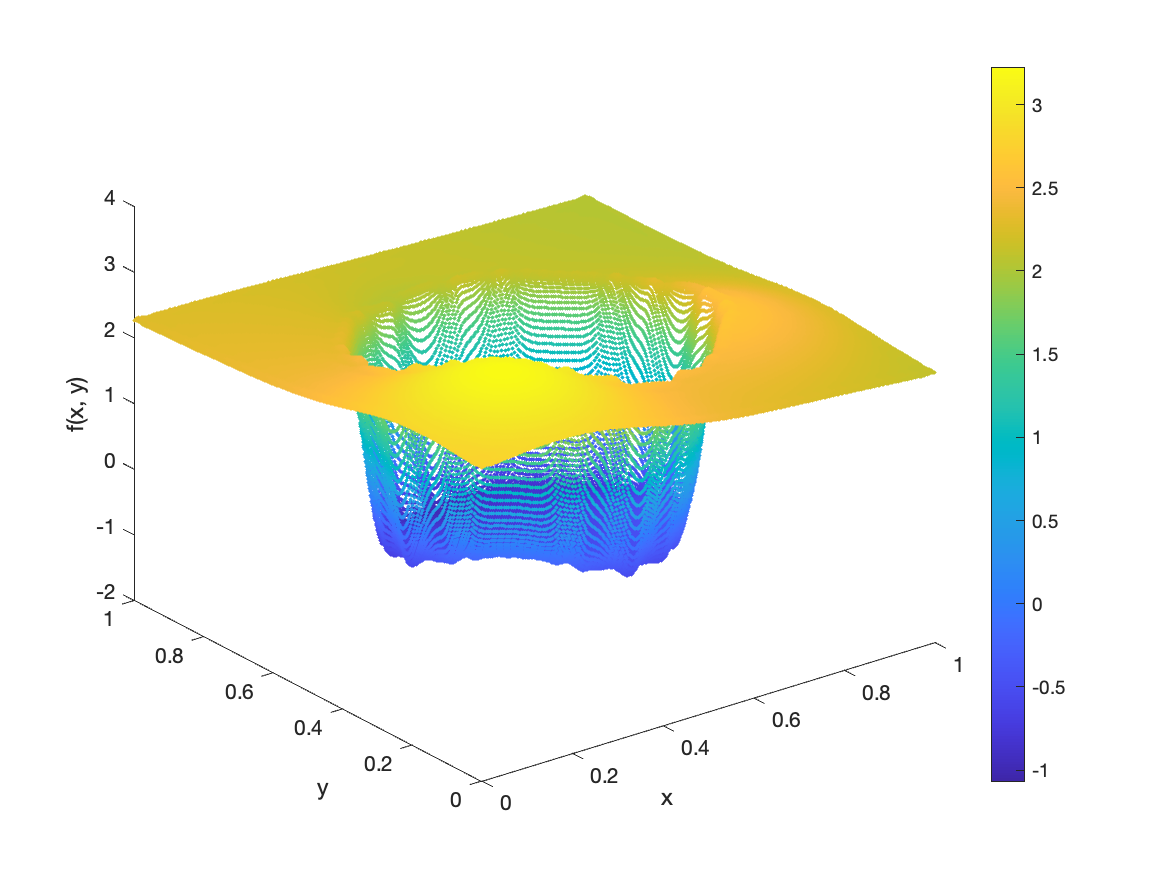}\\
	\includegraphics[width=6cm]{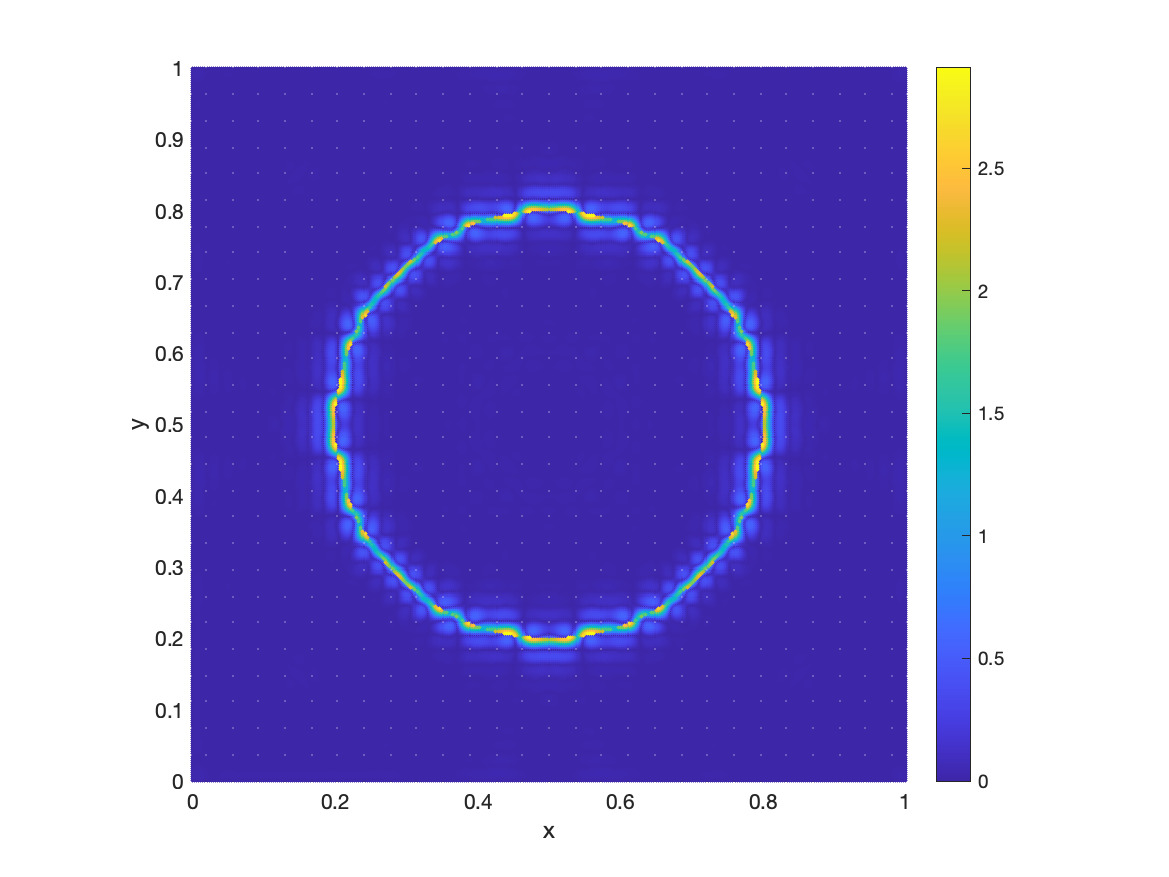} & 	\includegraphics[width=6cm]{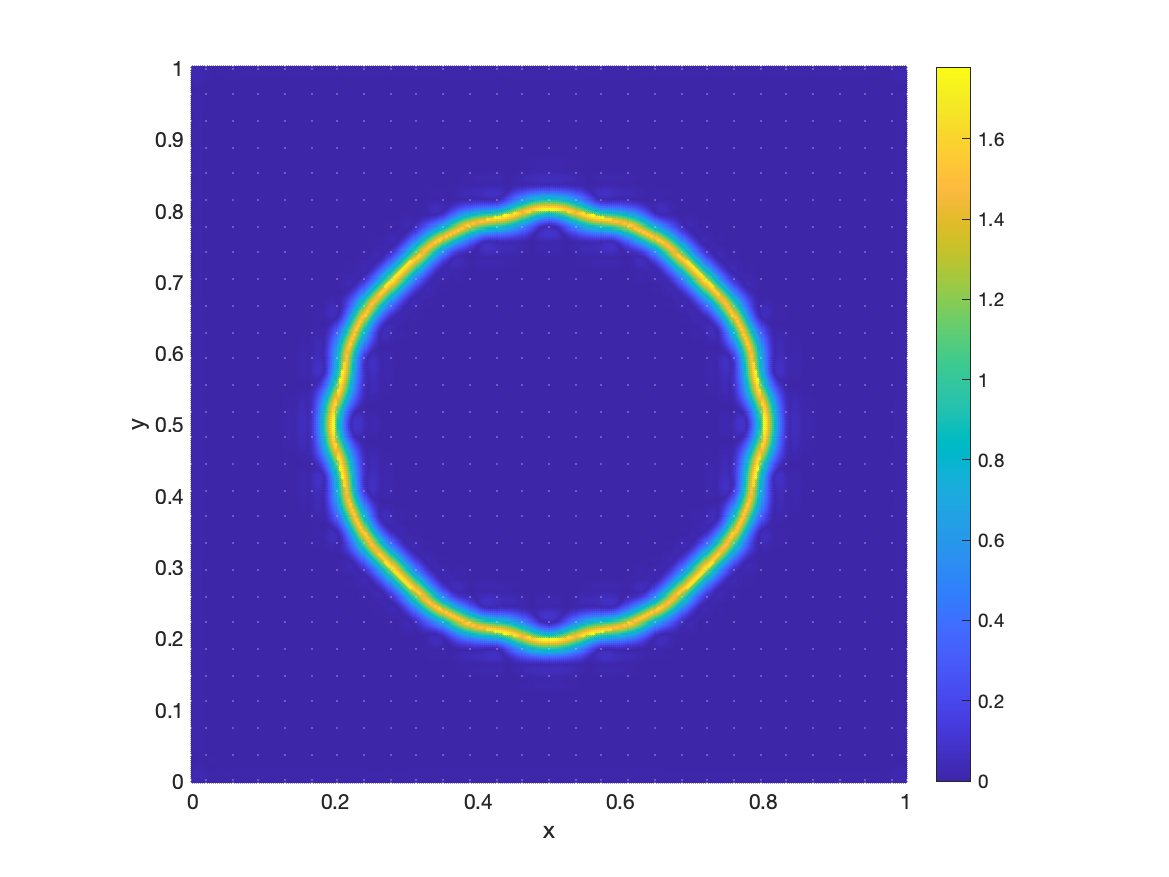}
		\end{tabular}
\end{center}
			\caption{Approximation of the function $f_1$, defined in Eq.~\eqref{frankesdisc}, using $n = 50 \times 50$ initial gridded data points. The left column shows the results obtained with the classical RBF$_{\text{W2}}$ algorithm, while the right column corresponds to the data-dependent DD-RBF$_{\text{W2}}$ algorithm. The first row displays the final approximations, and the second row presents the error distributions across the domain.}
		\label{exp3_2D}
	\end{figure}

	\begin{figure}[htbp!]
\begin{center}
		\begin{tabular}{cc}
	\includegraphics[width=6cm]{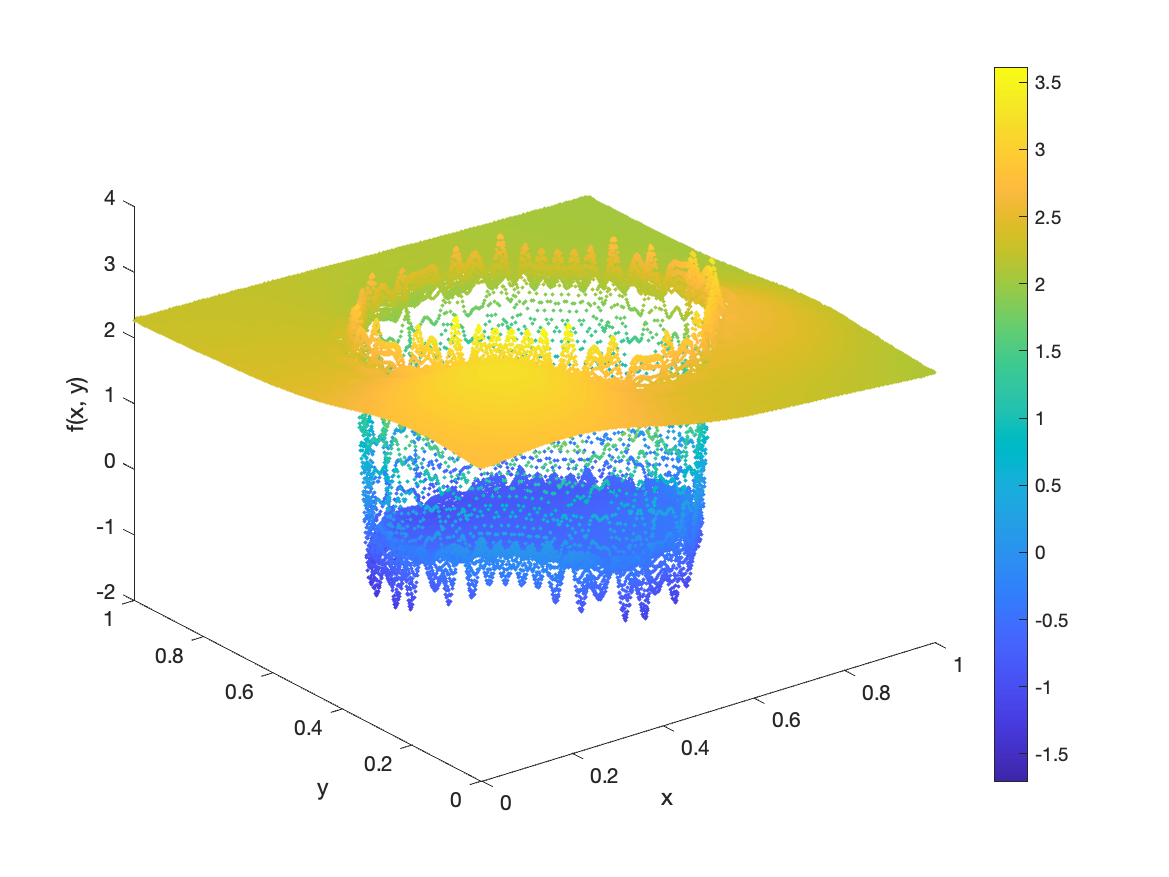} & 	\includegraphics[width=6cm]{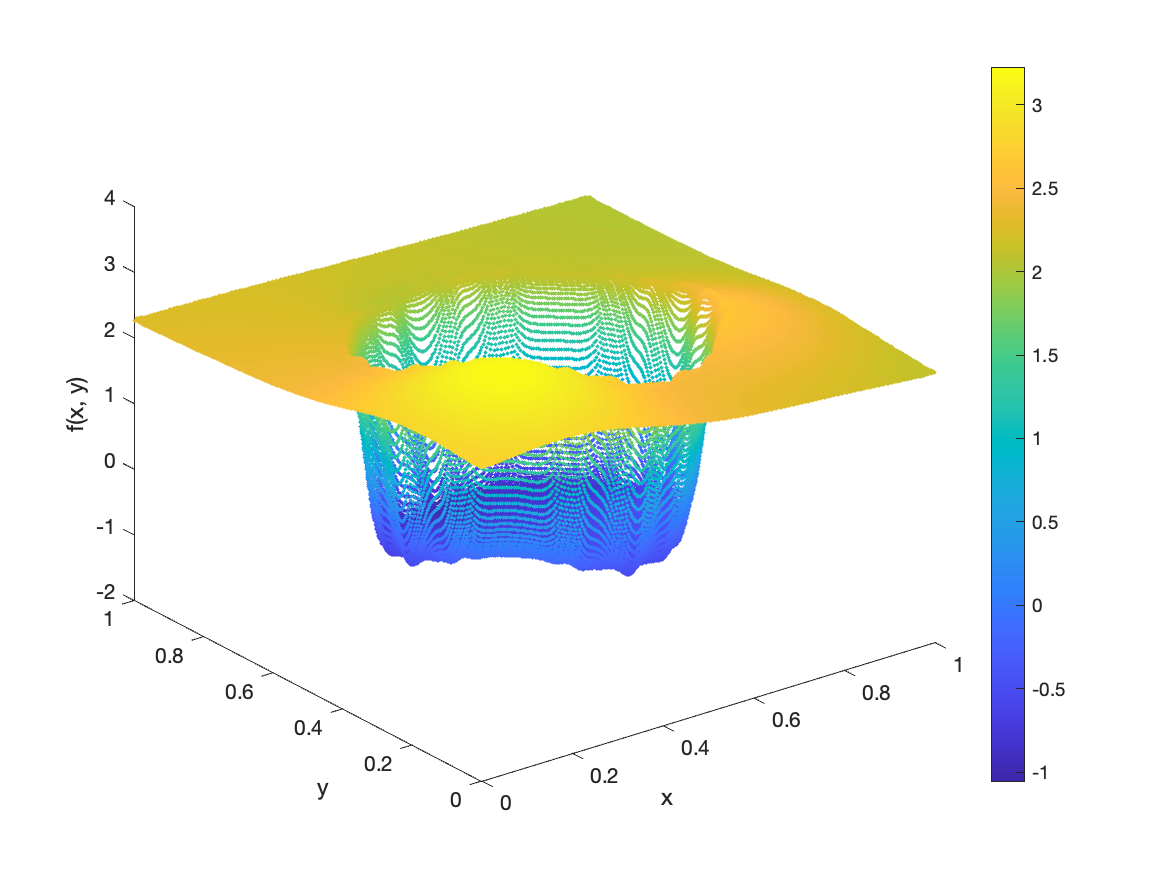}\\
	\includegraphics[width=6cm]{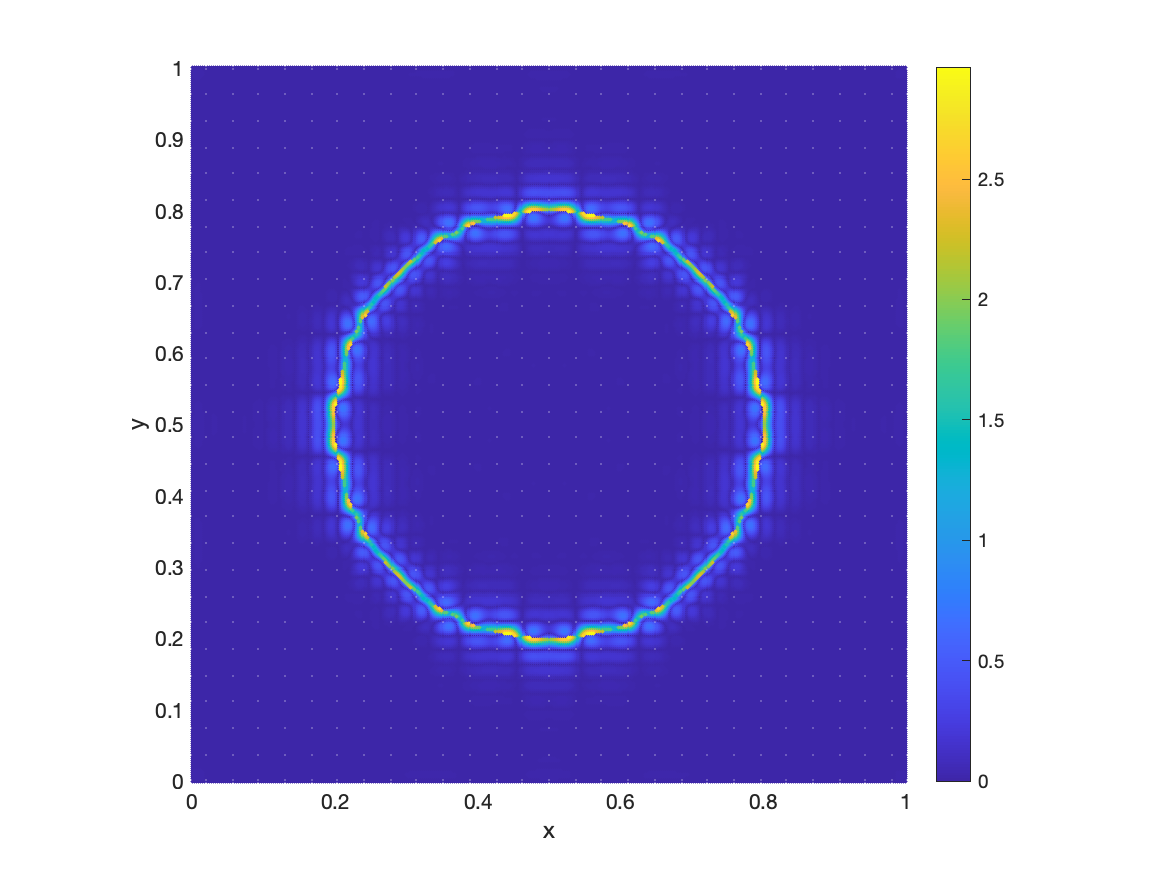} & 	\includegraphics[width=6cm]{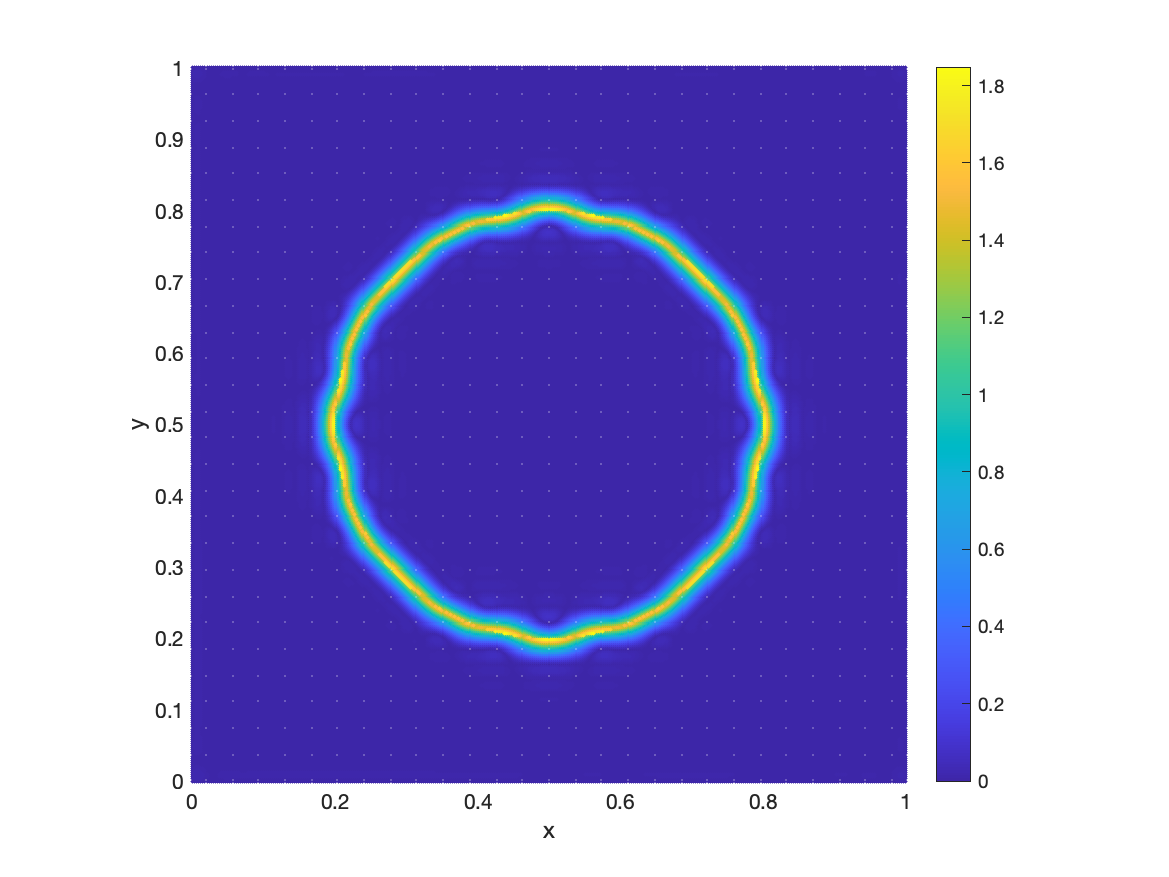}
		\end{tabular}
\end{center}
			\caption{Approximation of the function $f_1$, defined in Eq.~\eqref{frankesdisc}, using $n = 50 \times 50$ initial gridded data points. The left column shows the results obtained with the classical RBF$_{\text{W4}}$ algorithm, while the right column corresponds to the data-dependent DD-RBF$_{\text{W4}}$ algorithm. The first row displays the final approximations, and the second row presents the error distributions across the domain.}
		\label{exp4_2D}
	\end{figure}

	\begin{figure}[htbp!]
\begin{center}
		\begin{tabular}{cc}
	\includegraphics[width=6cm]{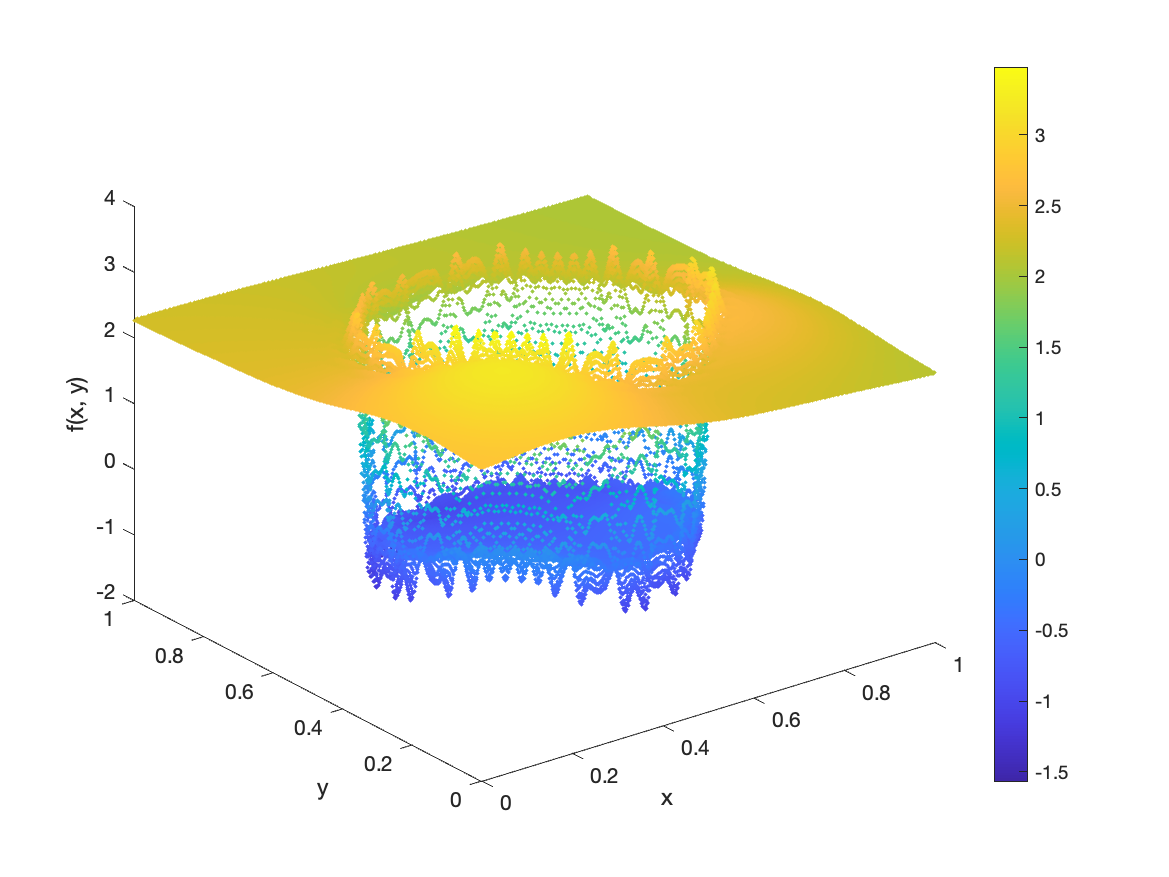} & 	\includegraphics[width=6cm]{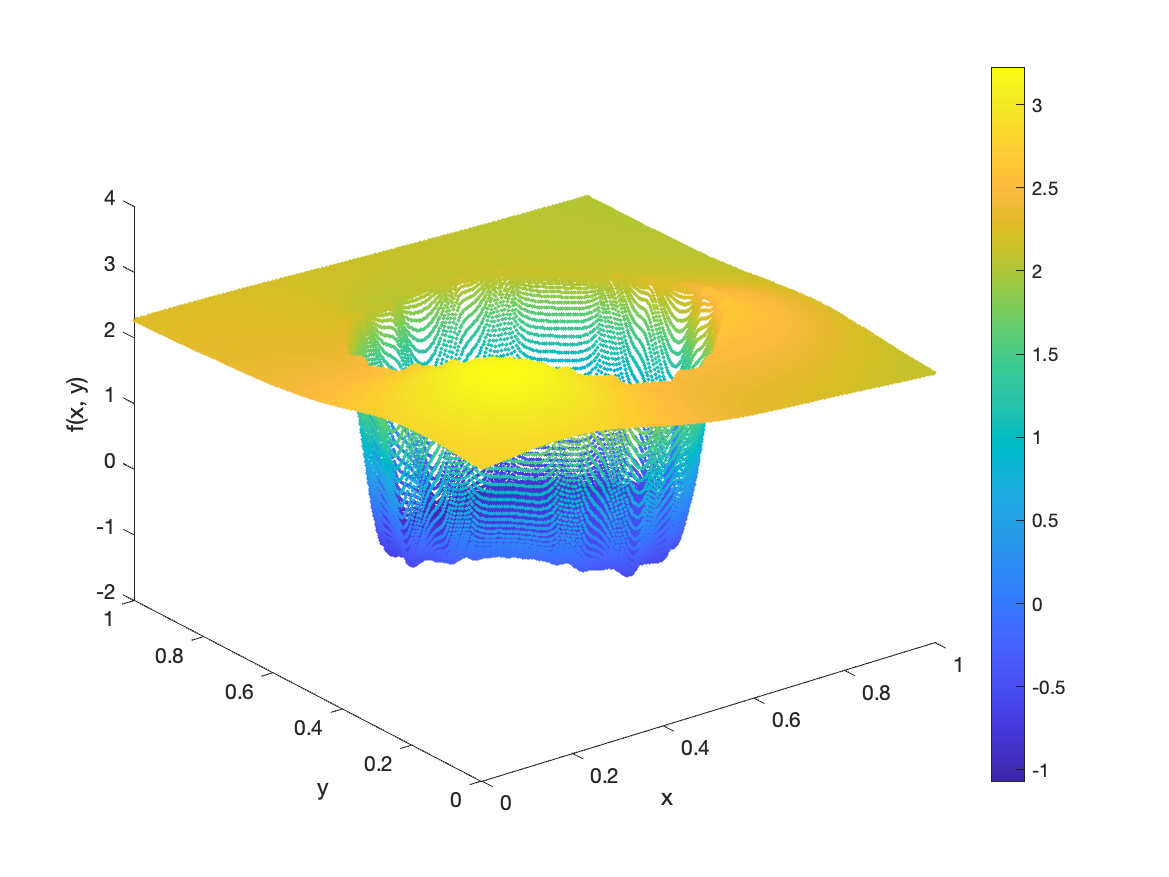}\\
	\includegraphics[width=6cm]{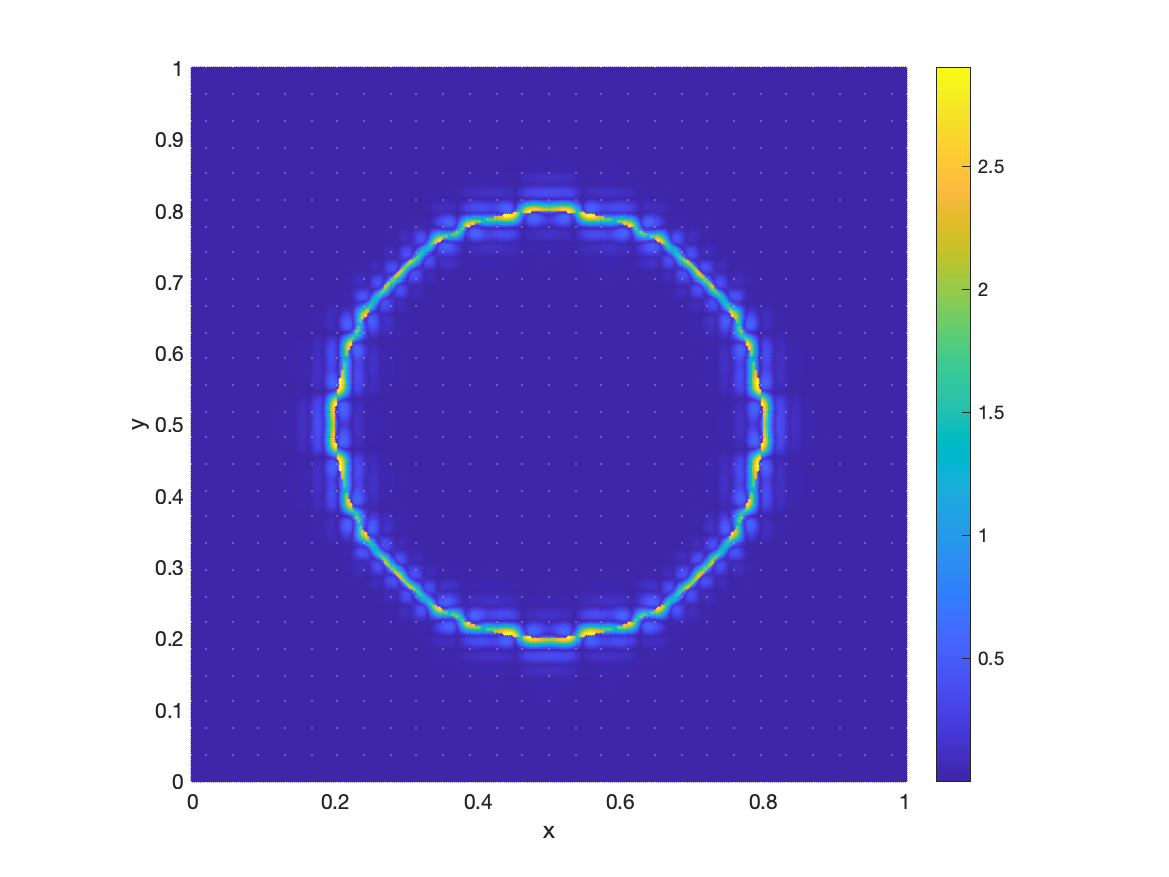} & 	\includegraphics[width=6cm]{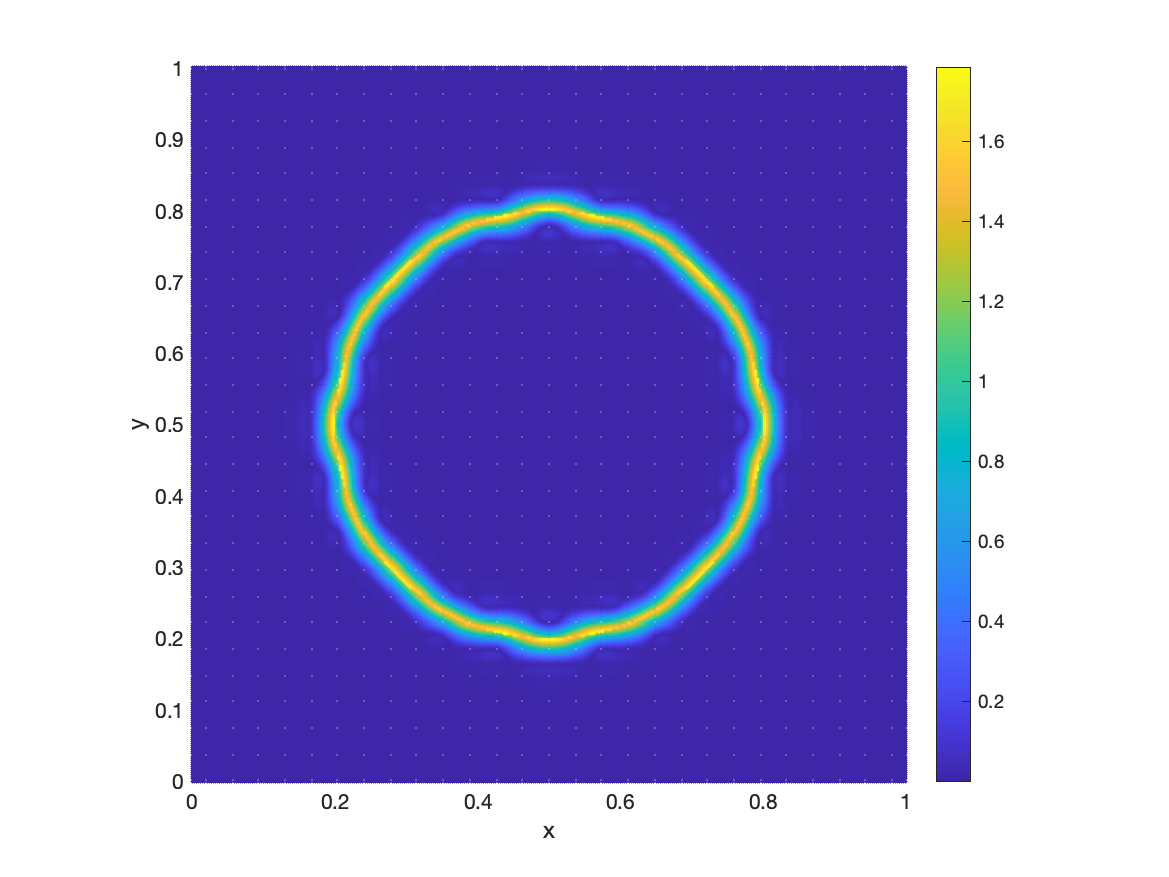}
		\end{tabular}
\end{center}
			\caption{Approximation of the function $f_1$, defined in Eq.~\eqref{frankesdisc}, using $n = 50 \times 50$ initial gridded data points. The left column shows the results obtained with the classical RBF$_{\text{M2}}$ algorithm, while the right column corresponds to the data-dependent DD-RBF$_{\text{M2}}$ algorithm. The first row displays the final approximations, and the second row presents the error distributions across the domain.}
		\label{exp5_2D}
	\end{figure}

	\begin{figure}[htbp!]
\begin{center}
		\begin{tabular}{cc}
	\includegraphics[width=6cm]{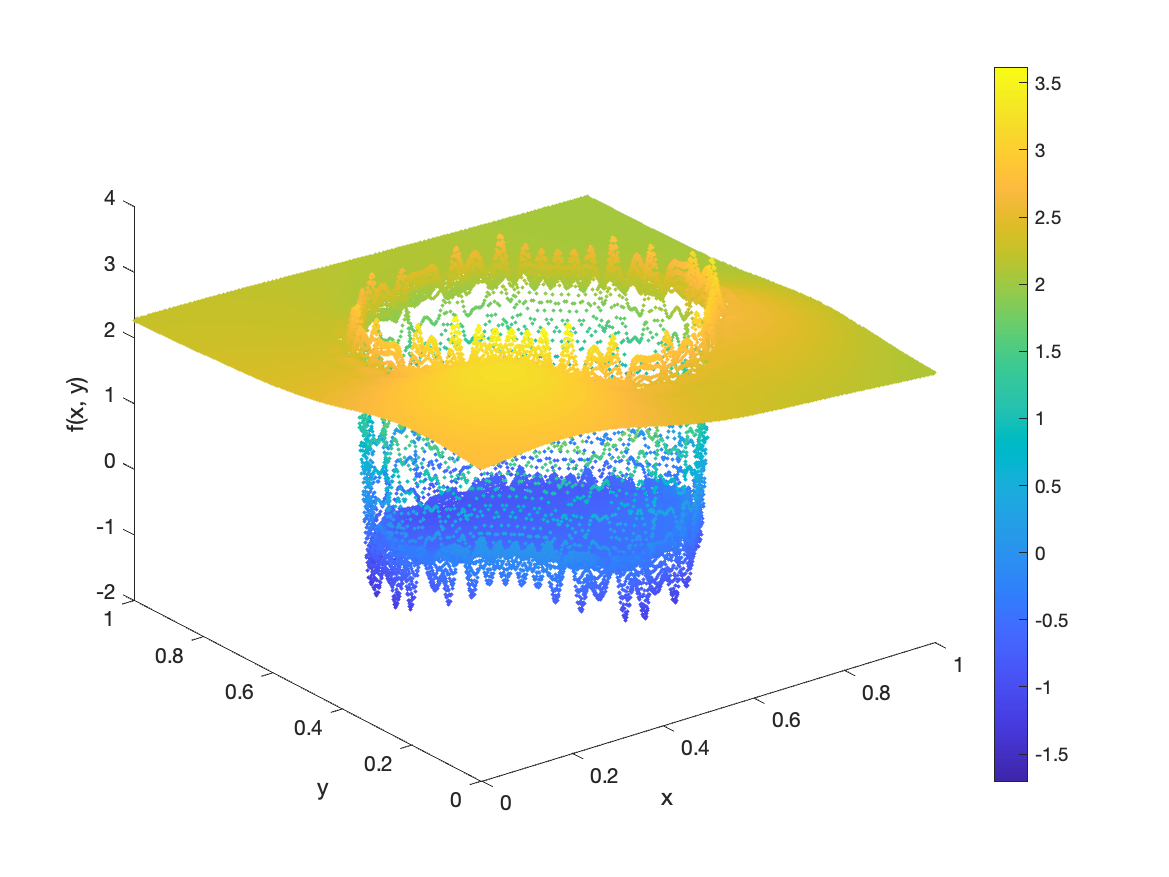} & 	\includegraphics[width=6cm]{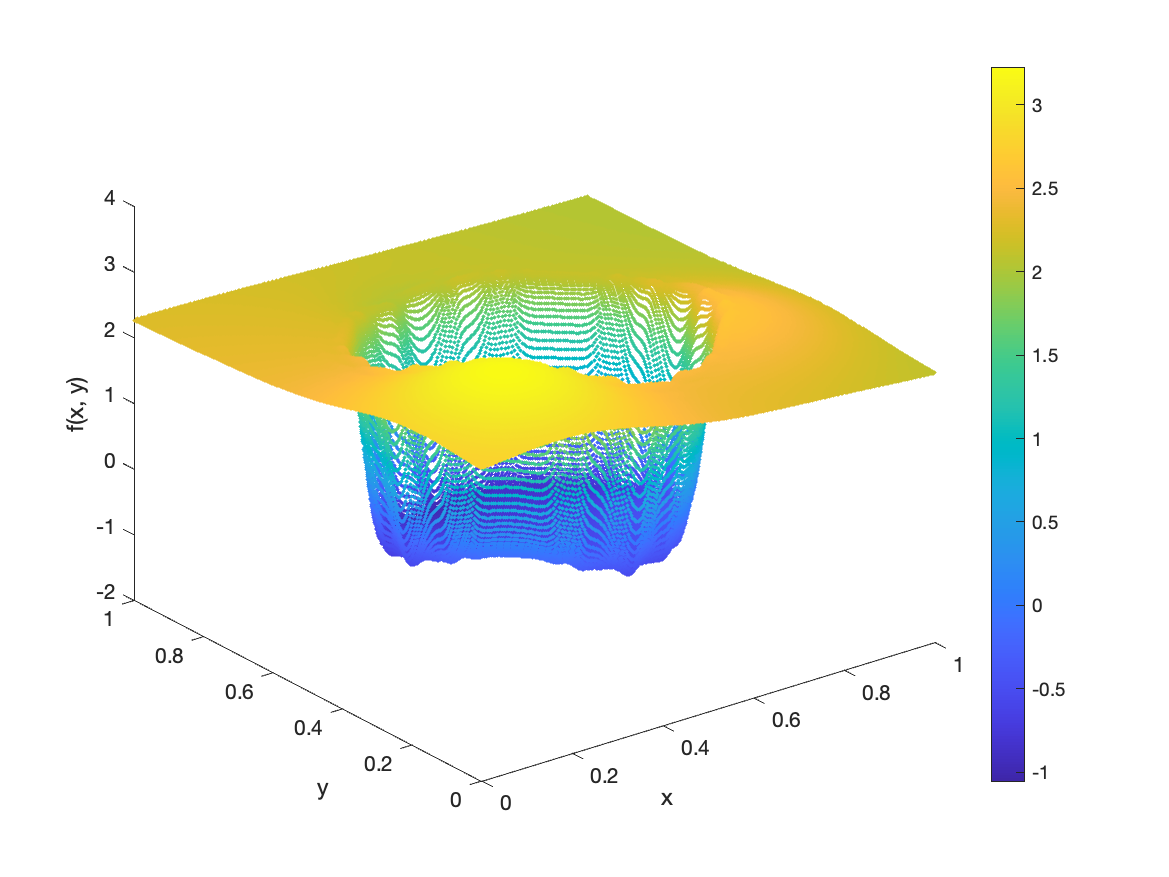}\\
	\includegraphics[width=6cm]{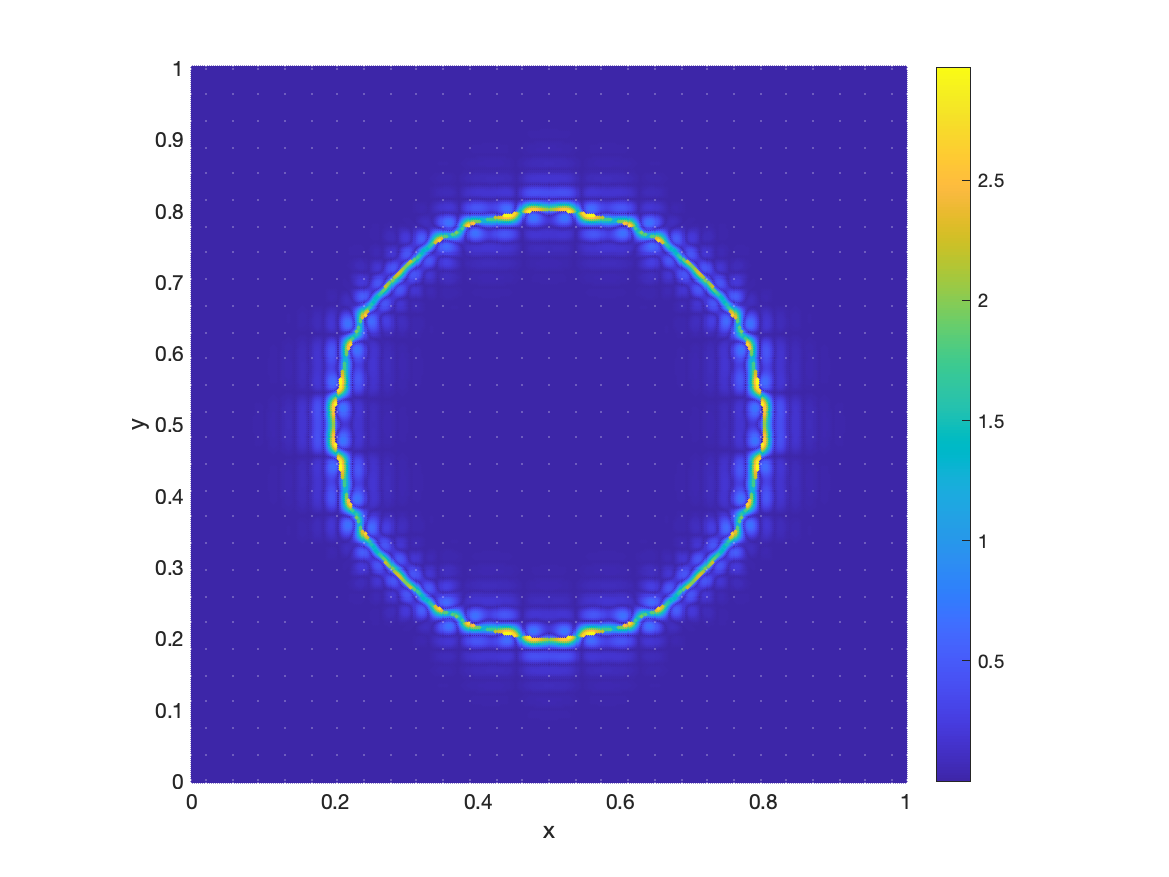} & 	\includegraphics[width=6cm]{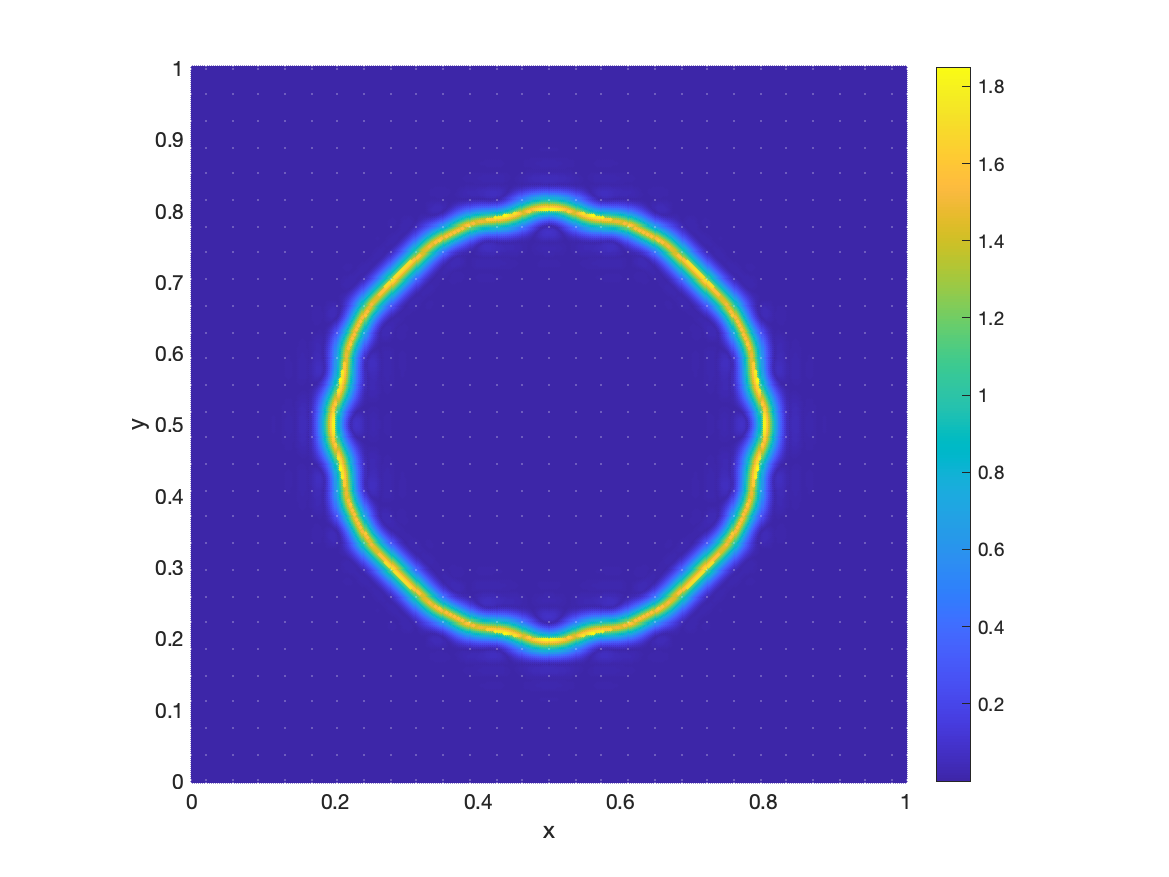}
		\end{tabular}
\end{center}
			\caption{Approximation of the function $f_1$, defined in Eq.~\eqref{frankesdisc}, using $n = 50 \times 50$ initial gridded data points. The left column shows the results obtained with the classical RBF$_{\text{M4}}$ algorithm, while the right column corresponds to the data-dependent DD-RBF$_{\text{M4}}$ algorithm. The first row displays the final approximations, and the second row presents the error distributions across the domain.}
		\label{exp6_2D}
	\end{figure}

\begin{table}[ht]
\centering
\begin{tabular}{|l|c|c|}
\hline
\textbf{Kernel} & $\kappa$ \textbf{(Classical)} & $\kappa$ \textbf{(data-dependent)} \\
\hline
RBF$_{\text{G}}$ & 7.9186e+07 & 4.1700e+07 \\
RBF$_{\text{IMQ}}$ & 1.7755e+05 & 1.6204e+05 \\
RBF$_{\text{W2}}$ & 1.6265e+04 & 1.5306e+04 \\
RBF$_{\text{W4}}$ & 1.3790e+05 & 1.2965e+05 \\
RBF$_{\text{M2}}$ & 1.9659e+07 & 1.8260e+07 \\
RBF$_{\text{M4}}$ & 3.3058e+10 & 3.0535e+10 \\
\hline
\end{tabular}
\caption{Condition numbers $\kappa$ for classical and data-dependent RBF interpolation methods across different kernels for the experiments shown in Figures \ref{exp1_2D} to \ref{exp6_2D}, where the we have used uniform gridded data.}
\label{tabla_condicion_unif_2D}
\end{table}


	\begin{figure}[htbp!]
\begin{center}
		\begin{tabular}{cc}
	\includegraphics[width=6cm]{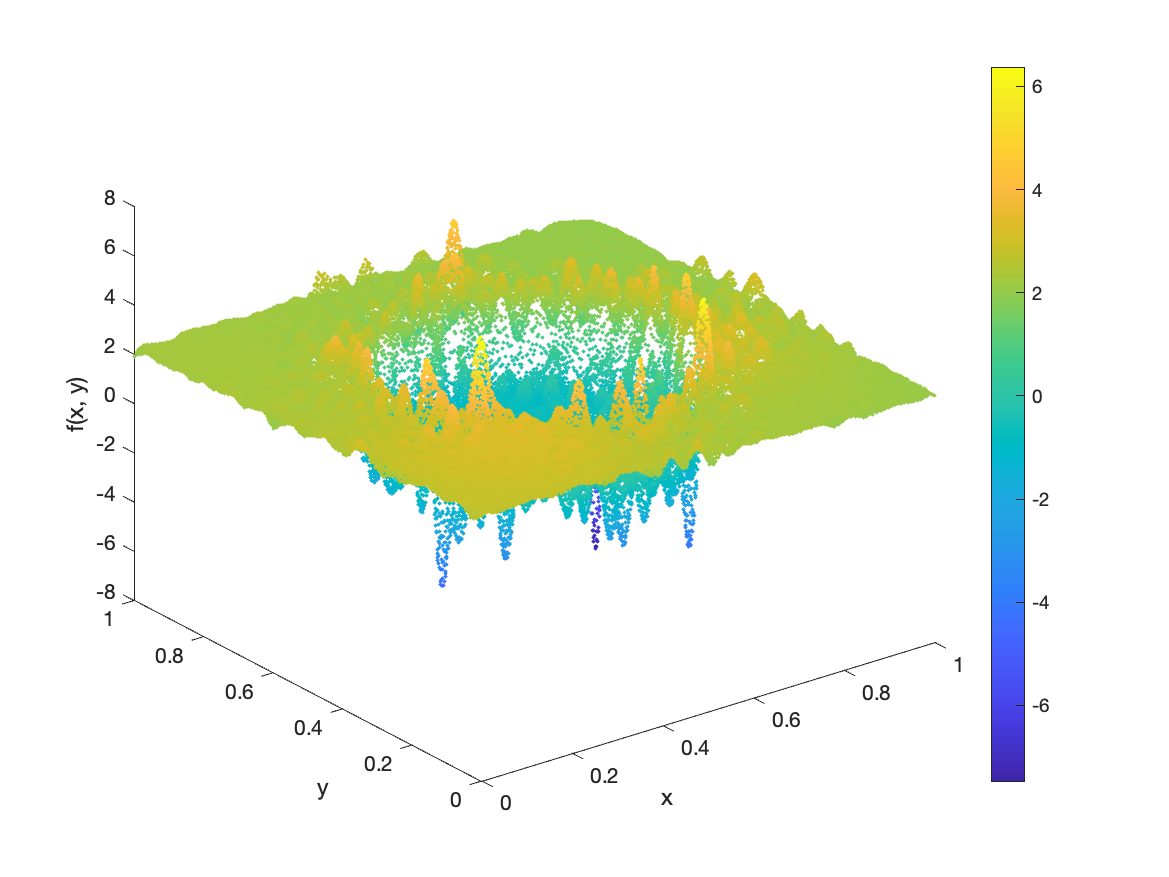} & 	\includegraphics[width=6cm]{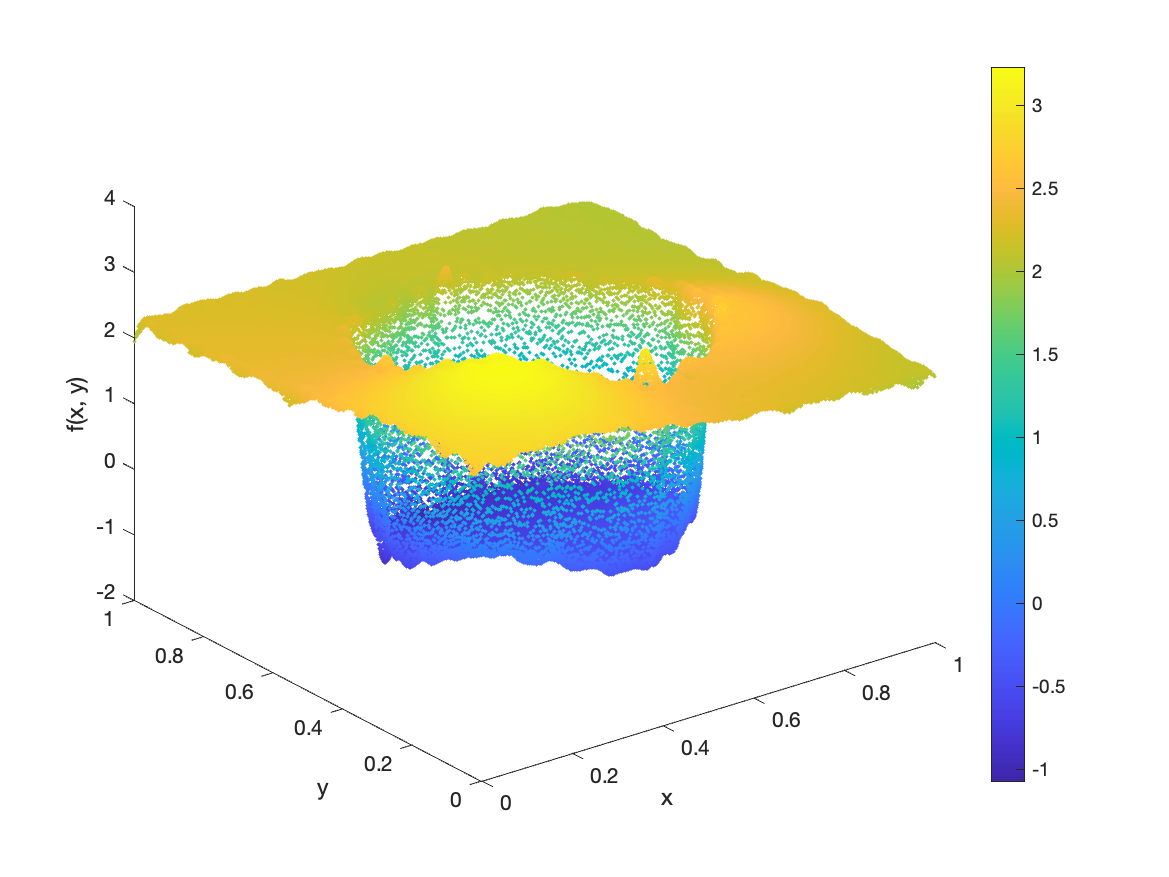}\\
	\includegraphics[width=6cm]{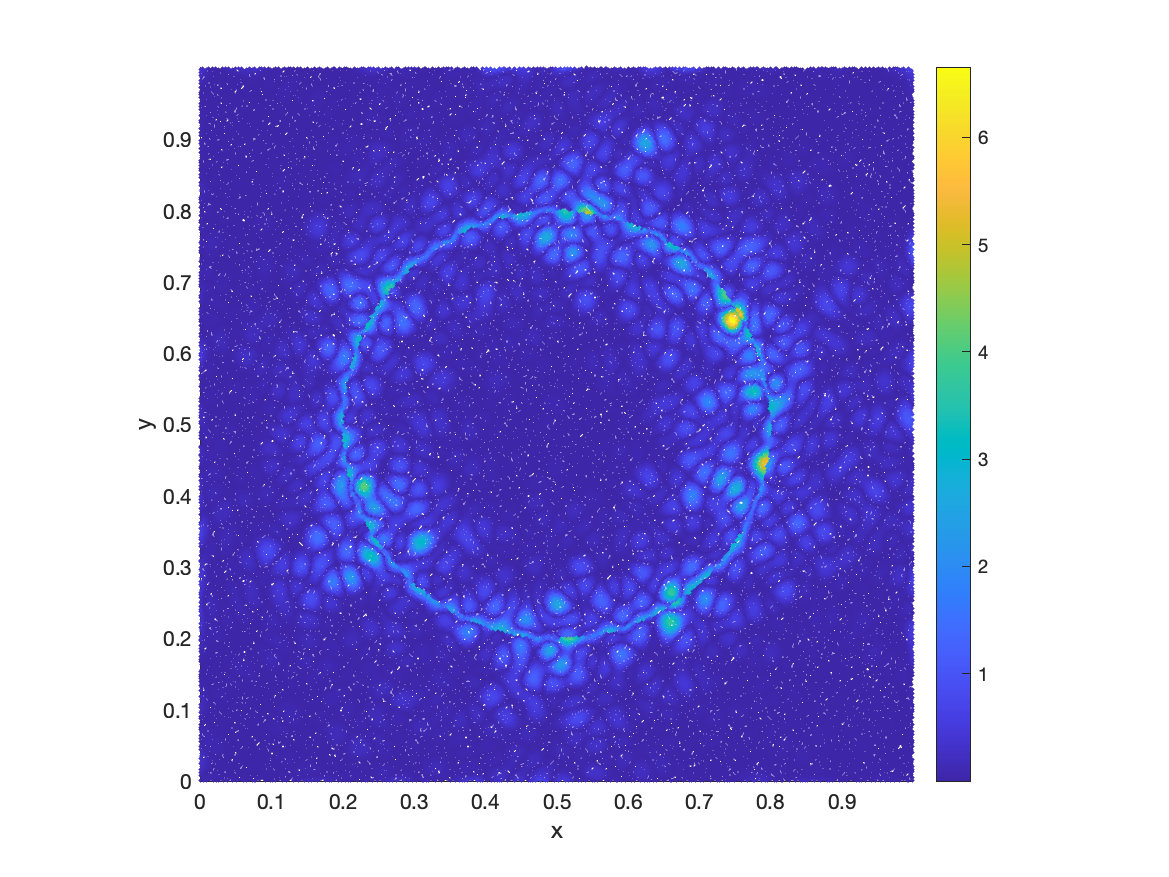} & 	\includegraphics[width=6cm]{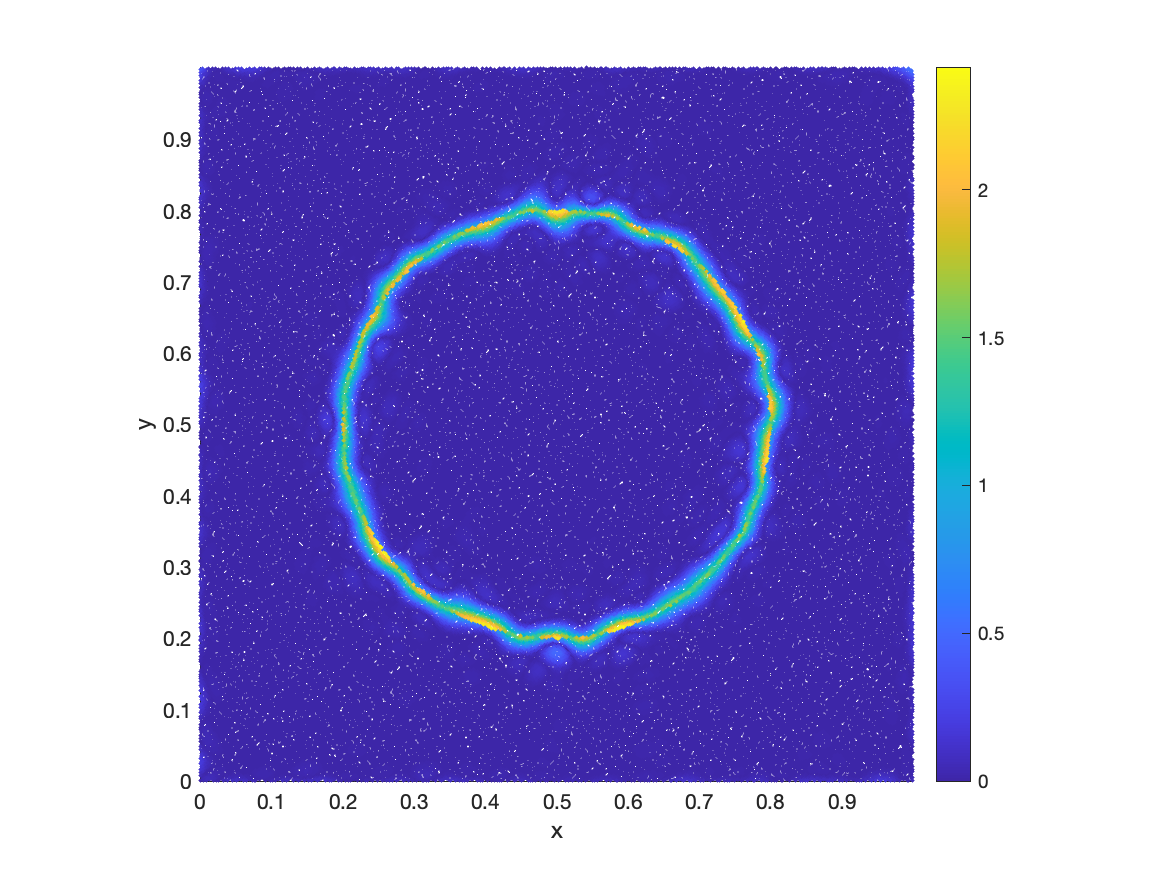}
		\end{tabular}
\end{center}
			\caption{Approximation of the function $f_1$, defined in Eq.~\eqref{frankesdisc}, using $n = 50 \times 50$ initial Halton points. The left column shows the results obtained with the classical RBF$_{\text{G}}$ algorithm, while the right column corresponds to the data-dependent DD-RBF$_{\text{G}}$ algorithm. The first row displays the final approximations, and the second row presents the error distributions across the domain.}
		\label{exp7_2D}
	\end{figure}

	\begin{figure}[htbp!]
\begin{center}
		\begin{tabular}{cc}
	\includegraphics[width=6cm]{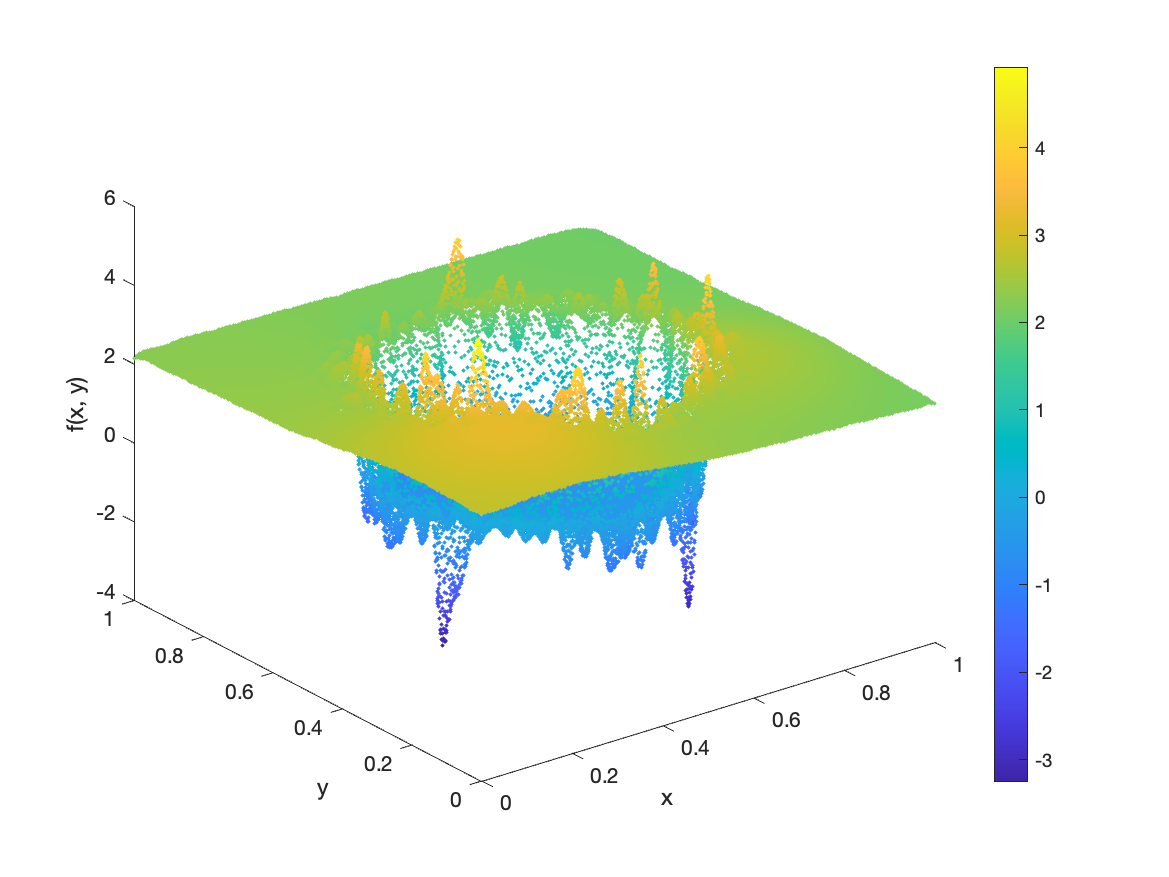} & 	\includegraphics[width=6cm]{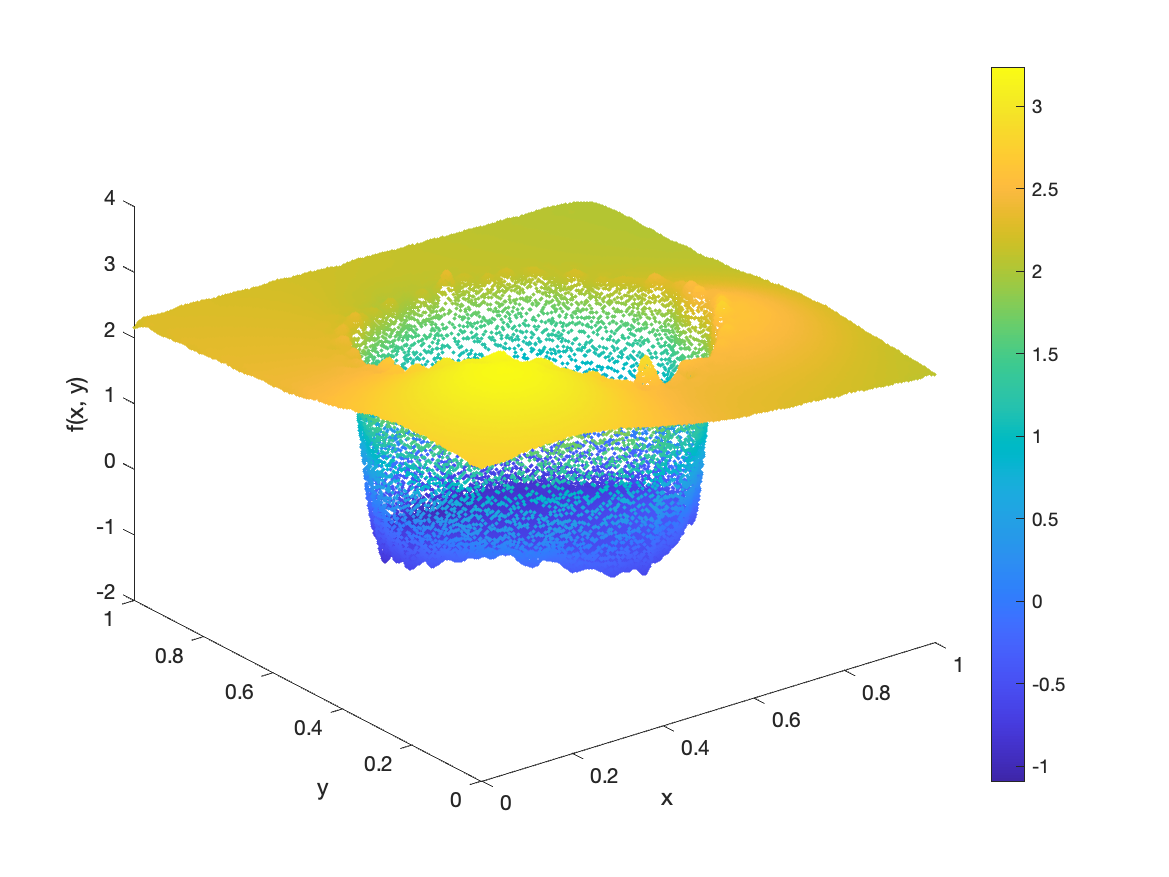}\\
	\includegraphics[width=6cm]{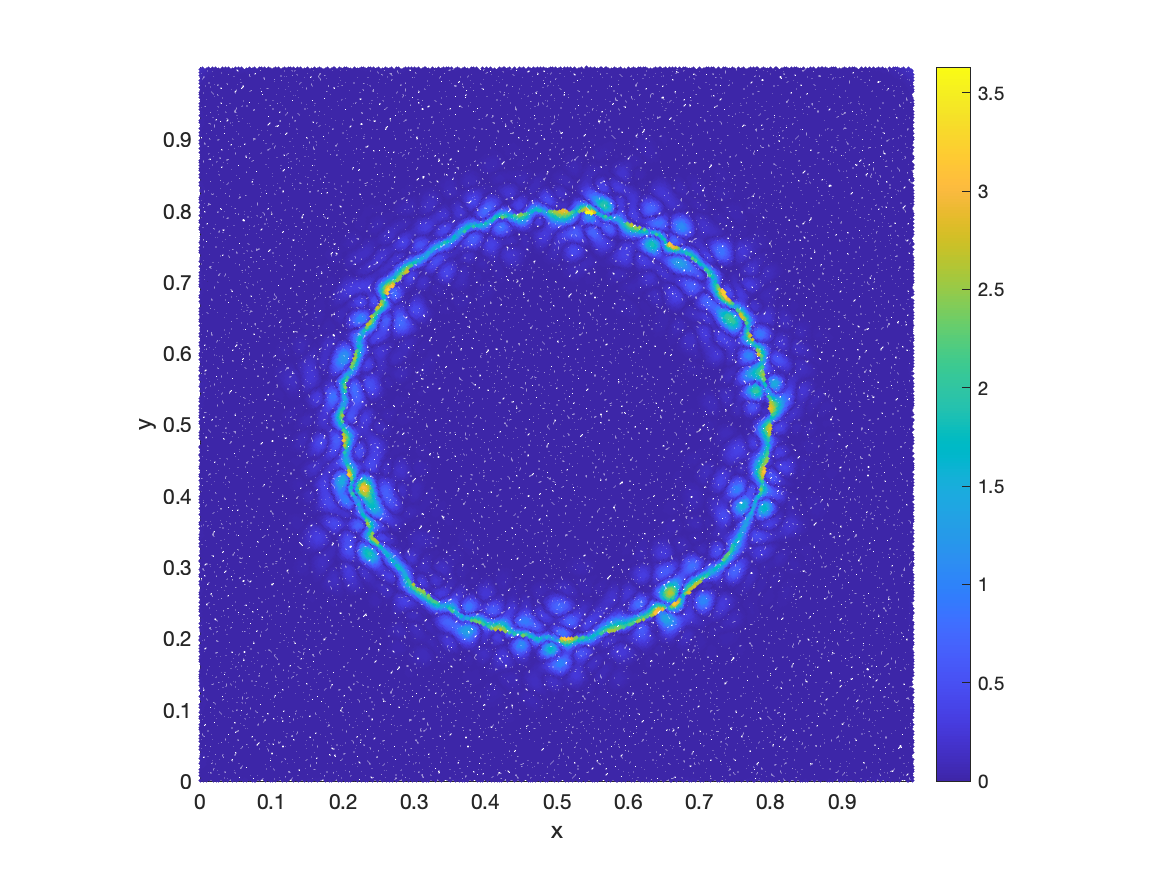} & 	\includegraphics[width=6cm]{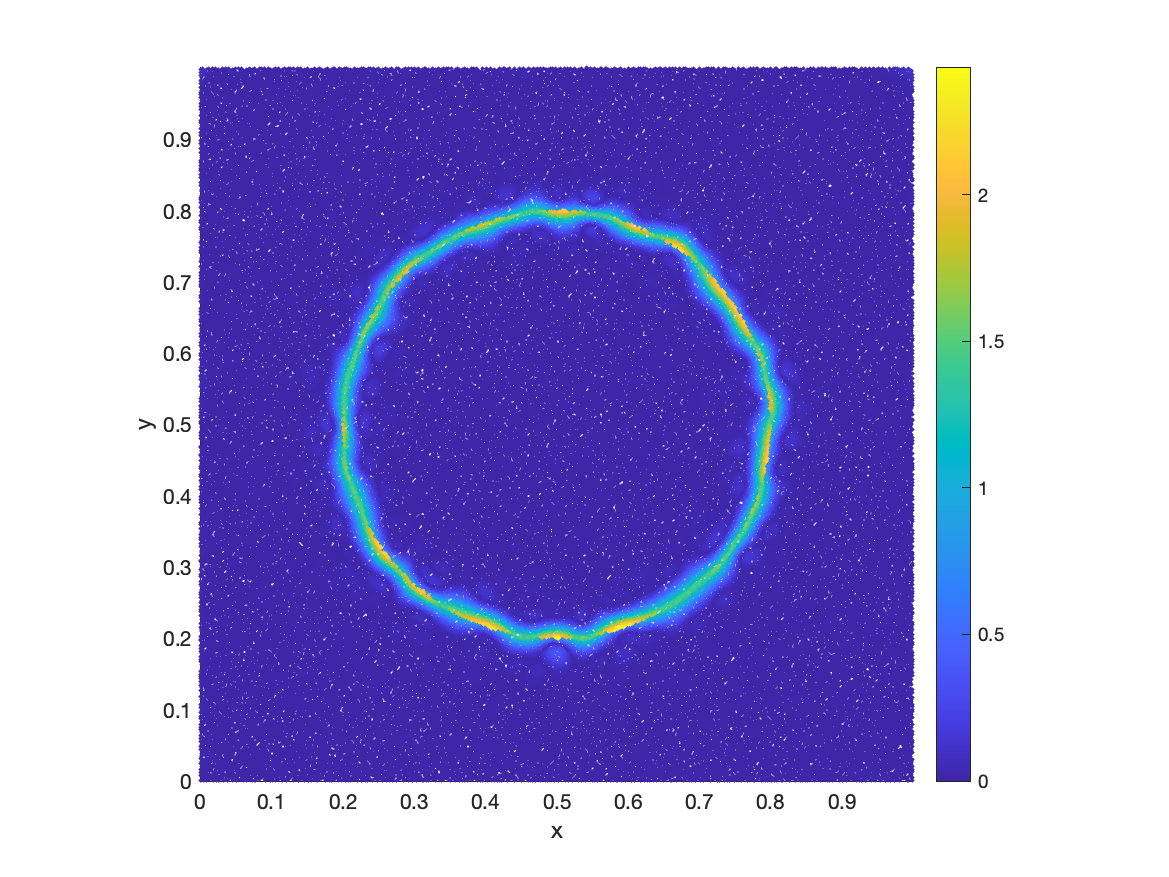}
		\end{tabular}
\end{center}
			\caption{Approximation of the function $f_1$, defined in Eq.~\eqref{frankesdisc}, using $n = 50 \times 50$ initial Halton points. The left column shows the results obtained with the classical RBF$_{\text{IMQ}}$ algorithm, while the right column corresponds to the data-dependent DD-RBF$_{\text{IMQ}}$ algorithm. The first row displays the final approximations, and the second row presents the error distributions across the domain.}
		\label{exp8_2D}
	\end{figure}

	\begin{figure}[htbp!]
\begin{center}
		\begin{tabular}{cc}
	\includegraphics[width=6cm]{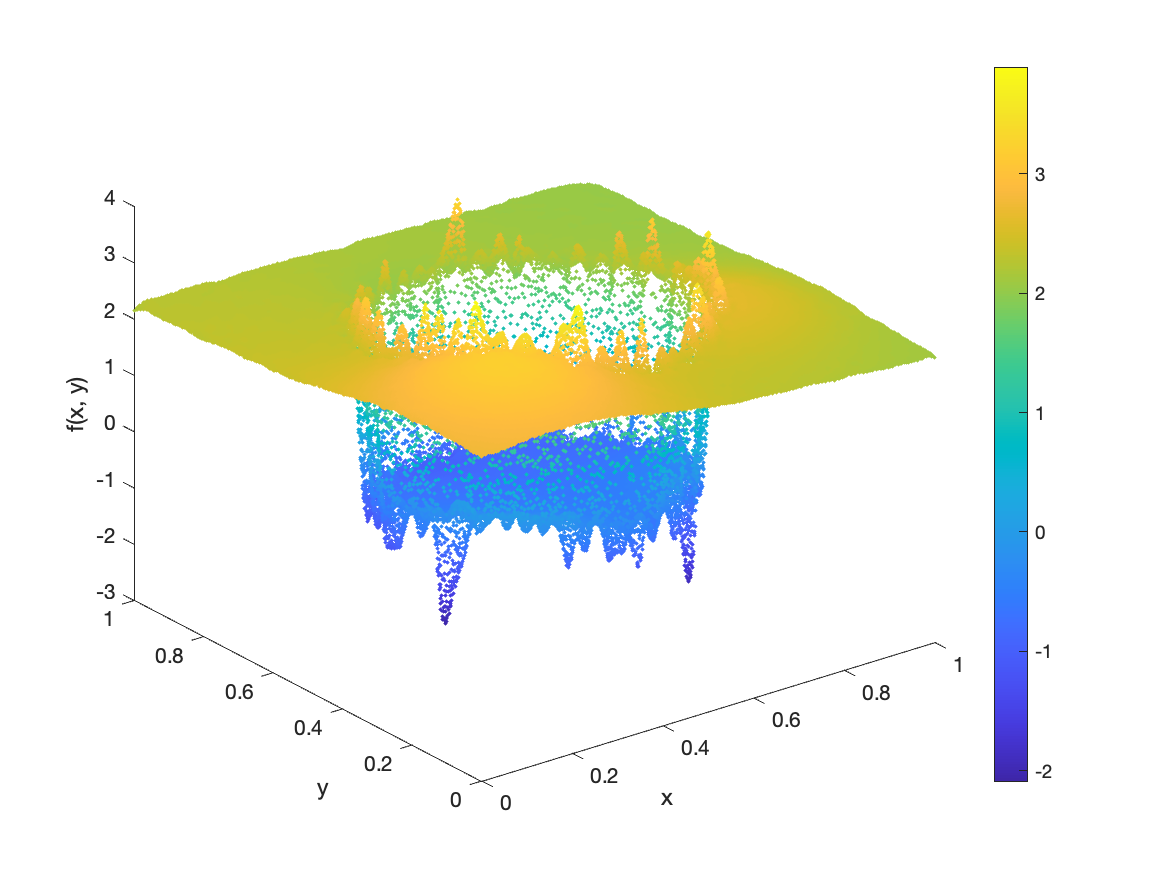} & 	\includegraphics[width=6cm]{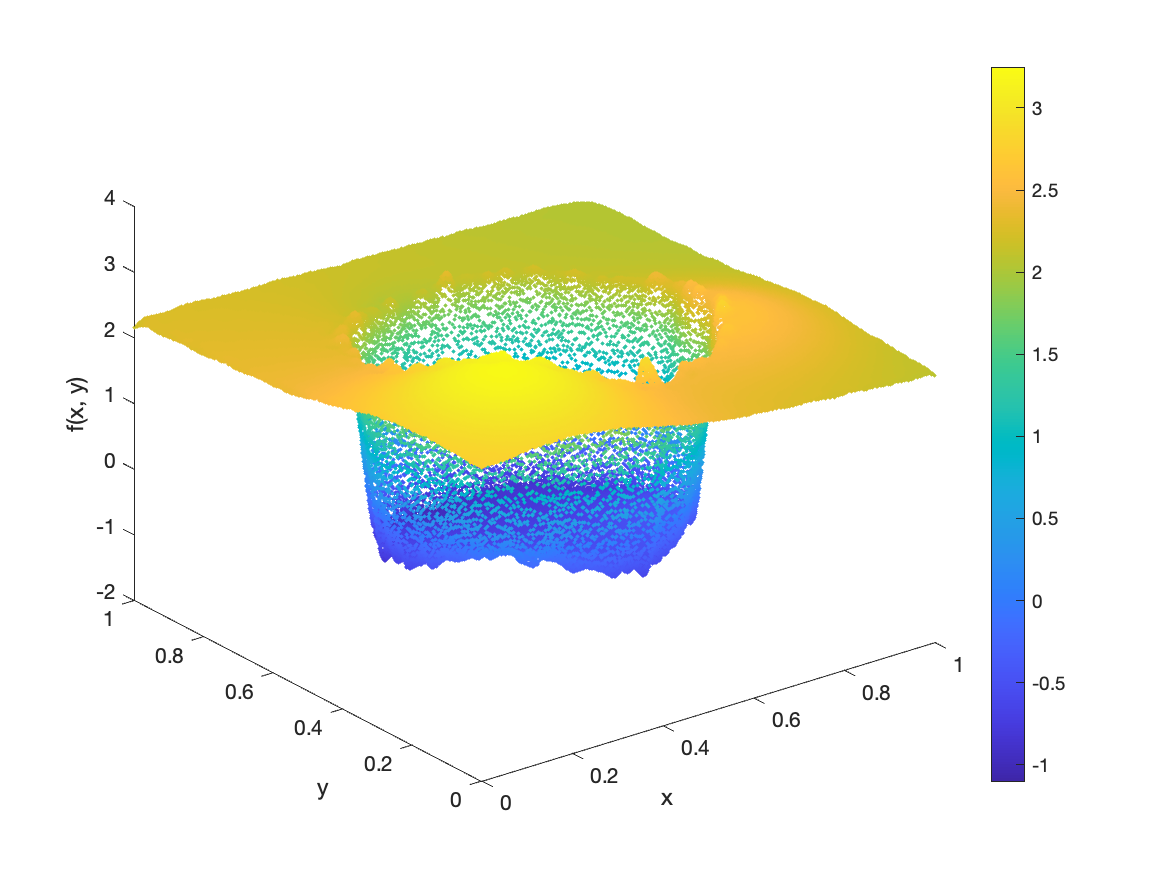}\\
	\includegraphics[width=6cm]{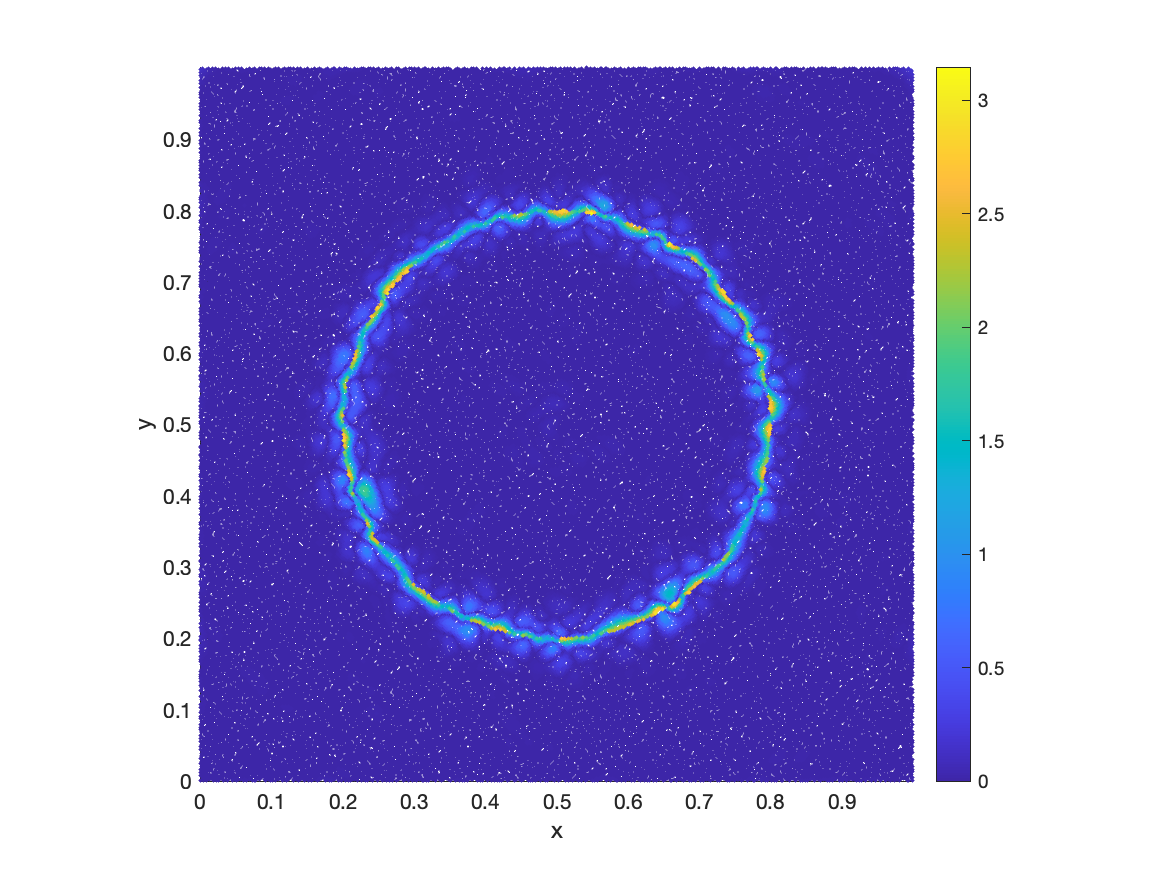} & 	\includegraphics[width=6cm]{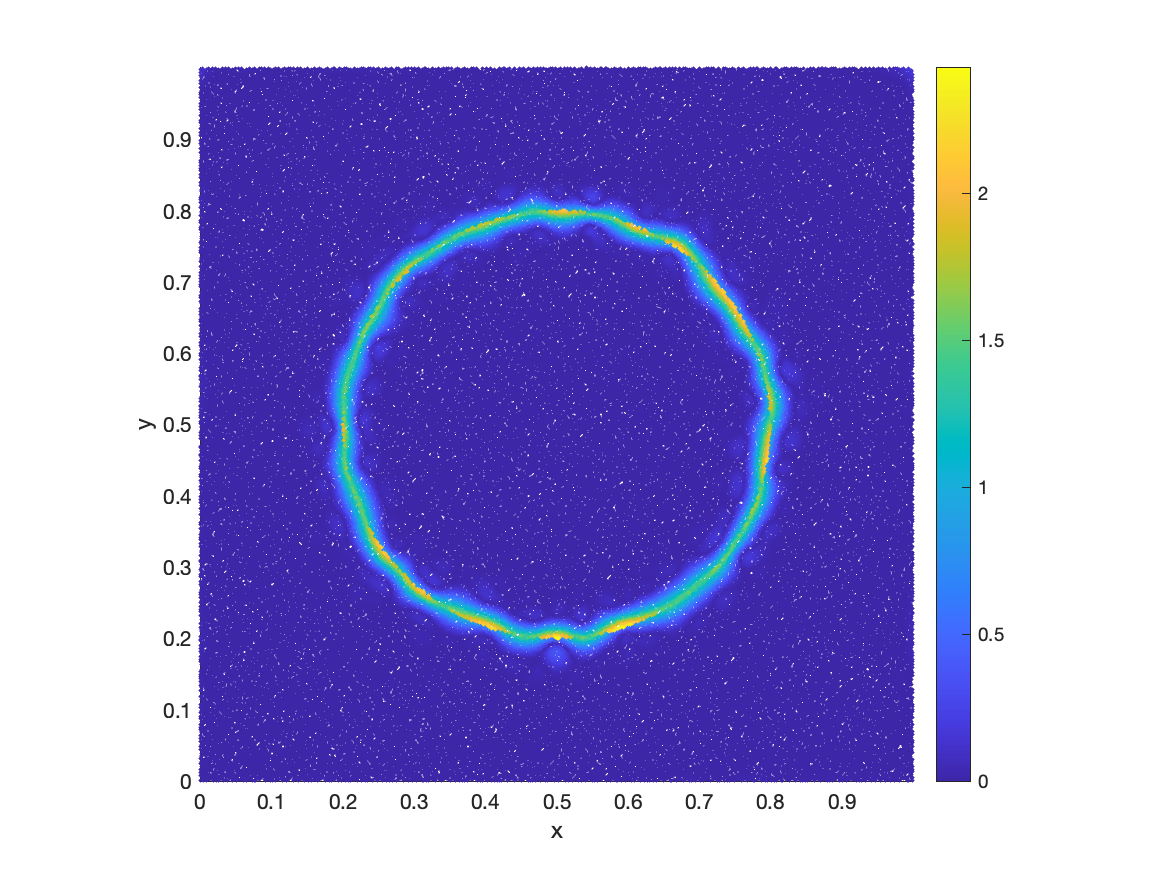}
		\end{tabular}
\end{center}
			\caption{Approximation of the function $f_1$, defined in Eq.~\eqref{frankesdisc}, using $n = 50 \times 50$ initial Halton points. The left column shows the results obtained with the classical RBF$_{\text{W2}}$ algorithm, while the right column corresponds to the data-dependent DD-RBF$_{\text{W2}}$ algorithm. The first row displays the final approximations, and the second row presents the error distributions across the domain.}
		\label{exp9_2D}
	\end{figure}

	\begin{figure}[htbp!]
\begin{center}
		\begin{tabular}{cc}
	\includegraphics[width=6cm]{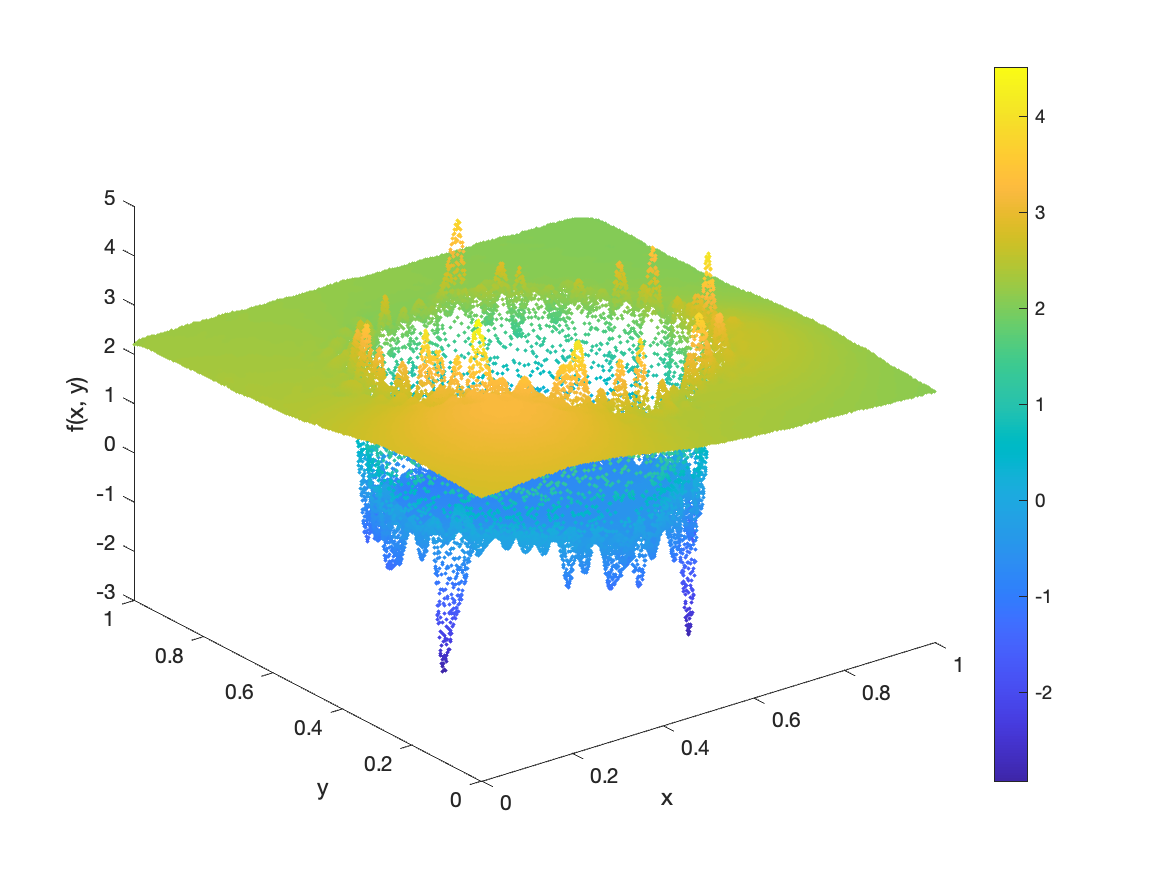} & 	\includegraphics[width=6cm]{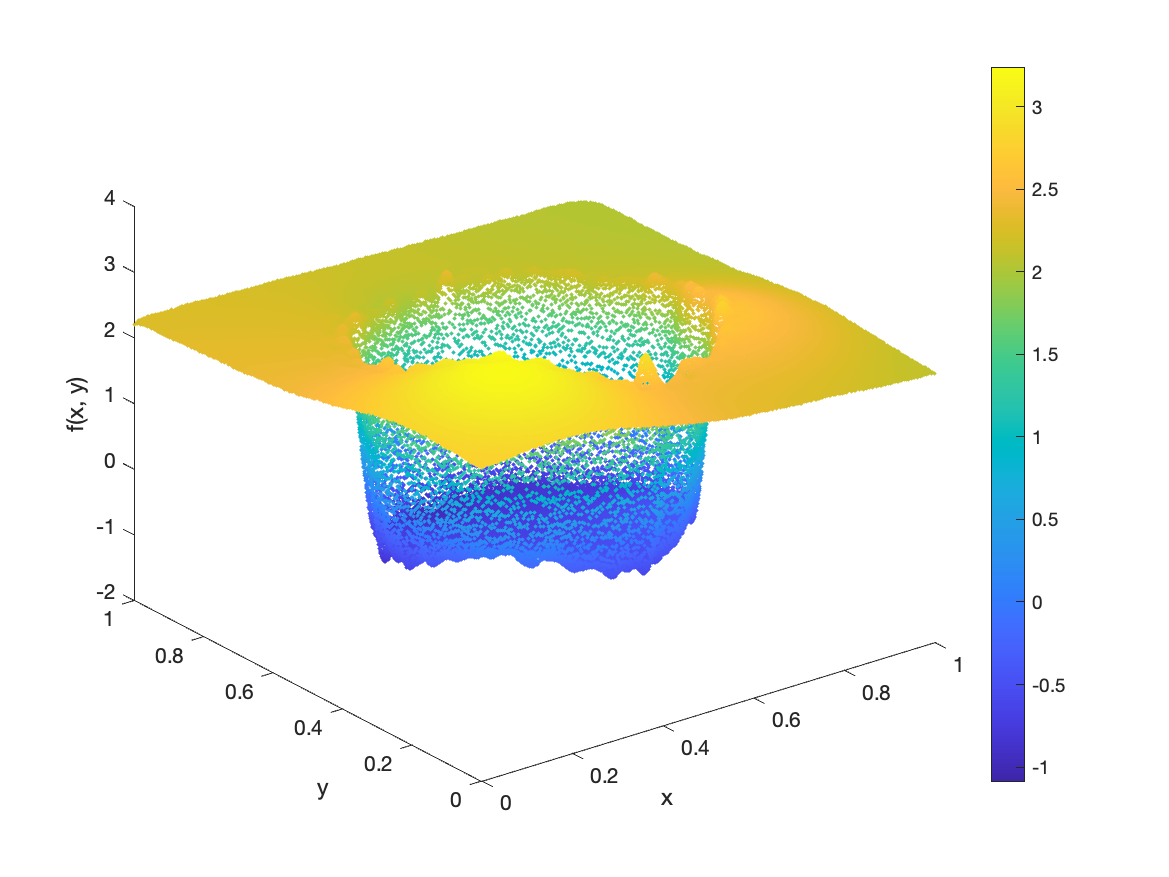}\\
	\includegraphics[width=6cm]{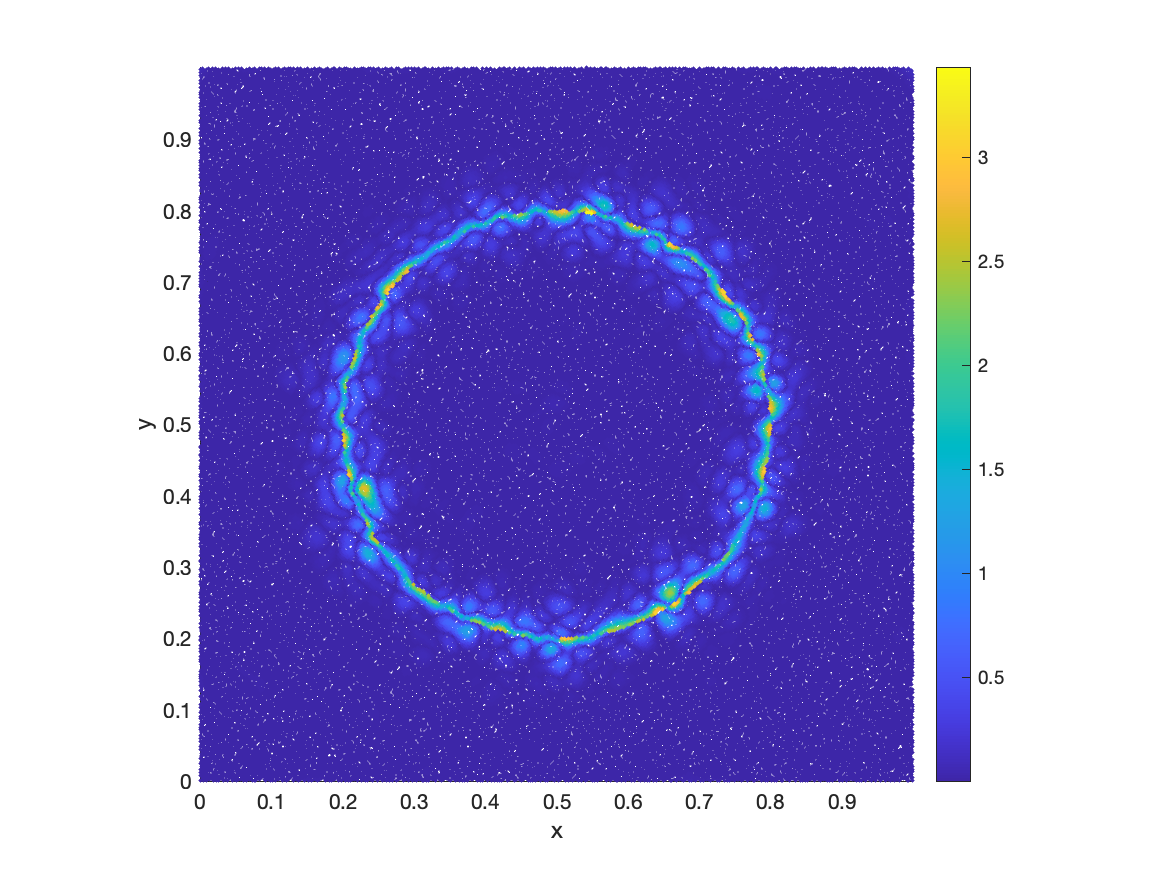} & 	\includegraphics[width=6cm]{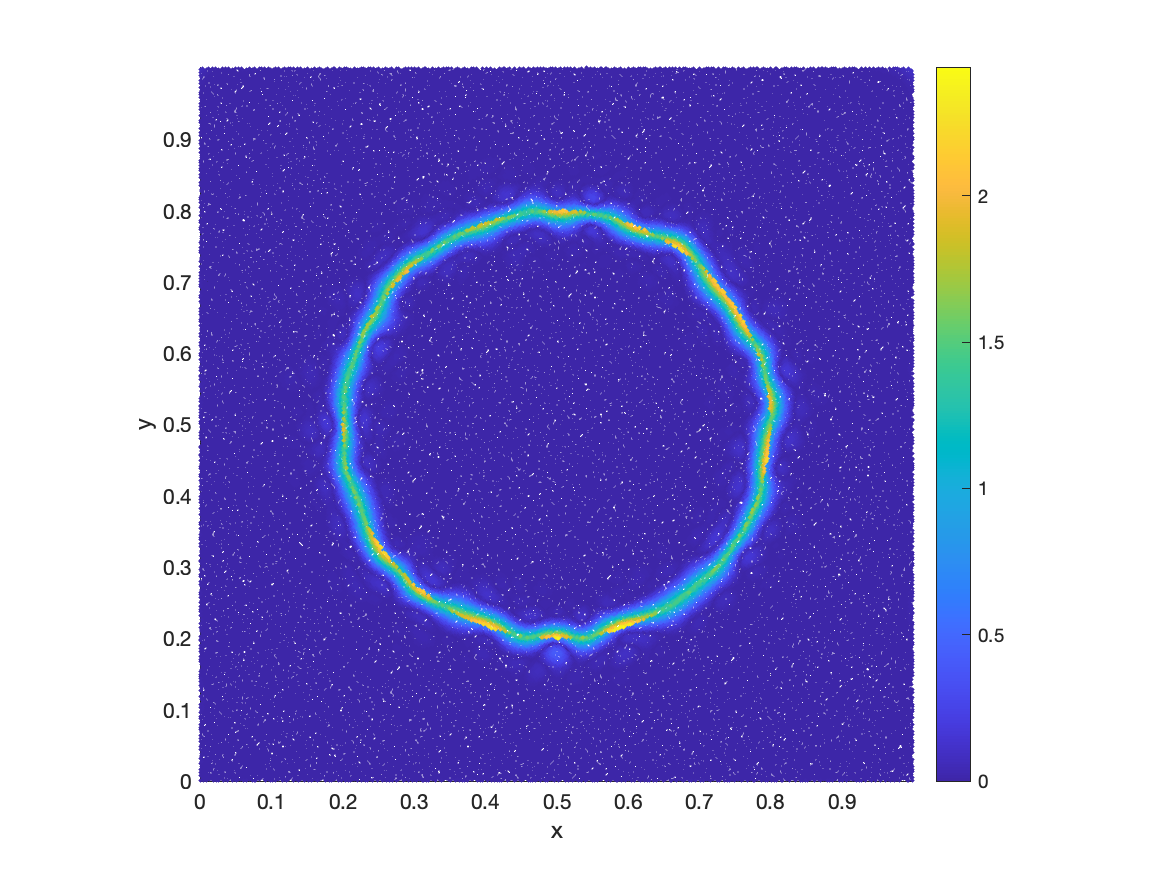}
		\end{tabular}
\end{center}
			\caption{Approximation of the function $f_1$, defined in Eq.~\eqref{frankesdisc}, using $n = 50 \times 50$ initial Halton points. The left column shows the results obtained with the classical RBF$_{\text{W4}}$ algorithm, while the right column corresponds to the data-dependent DD-RBF$_{\text{W4}}$ algorithm. The first row displays the final approximations, and the second row presents the error distributions across the domain.}
		\label{exp10_2D}
	\end{figure}

	\begin{figure}[htbp!]
\begin{center}
		\begin{tabular}{cc}
	\includegraphics[width=6cm]{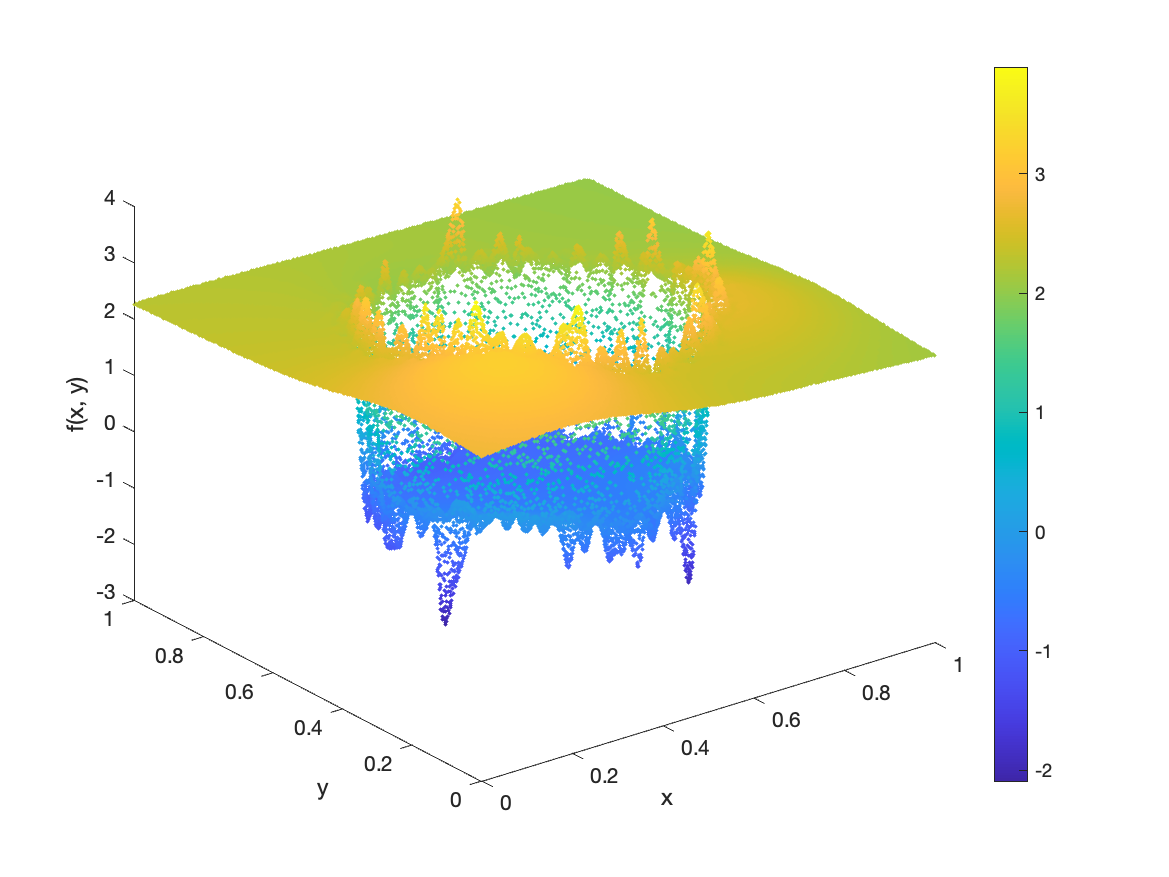} & 	\includegraphics[width=6cm]{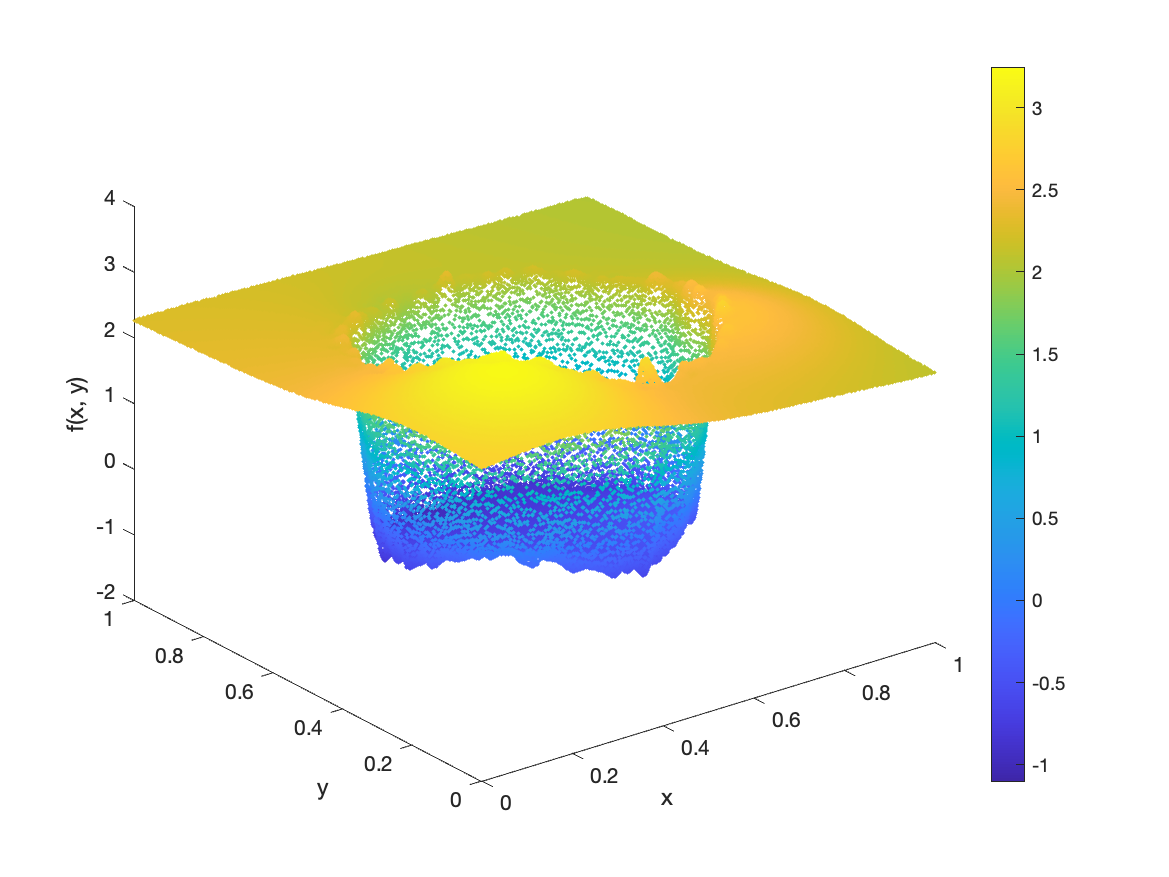}\\
	\includegraphics[width=6cm]{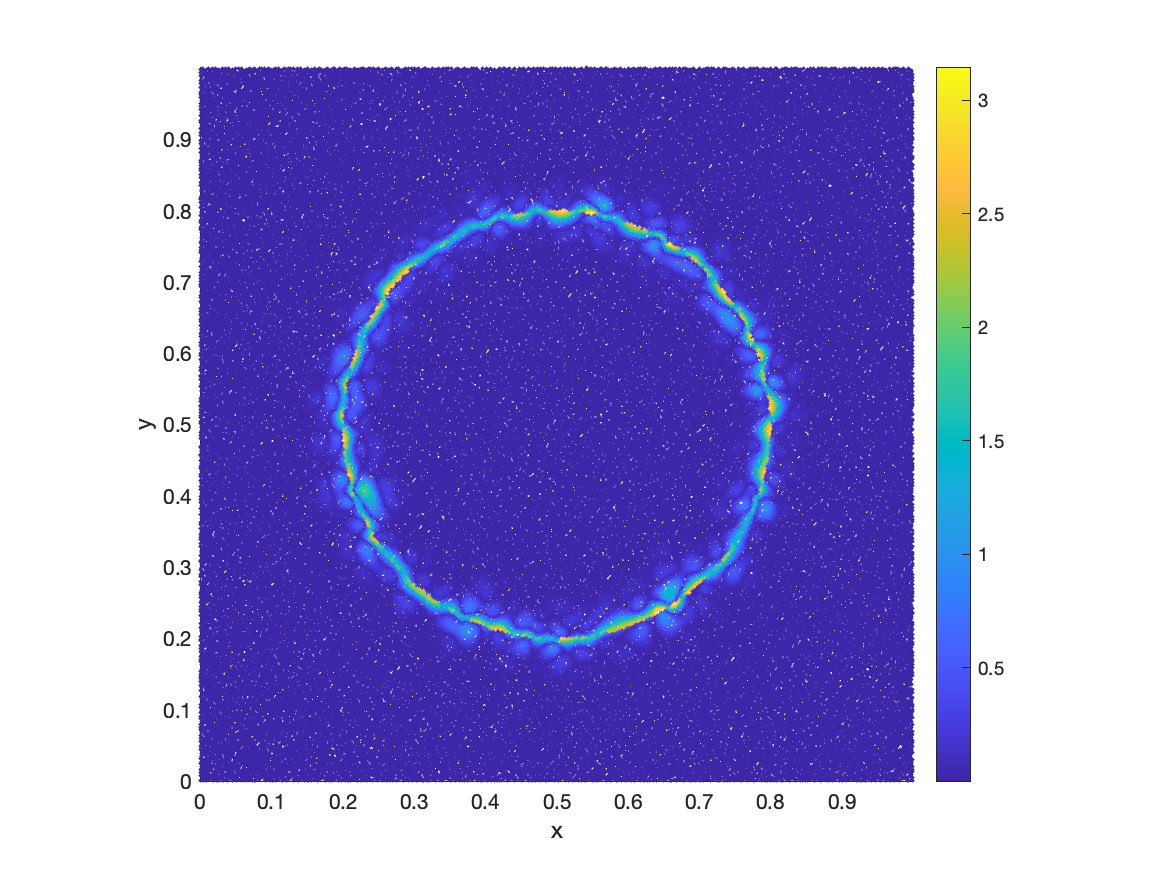} & 	\includegraphics[width=6cm]{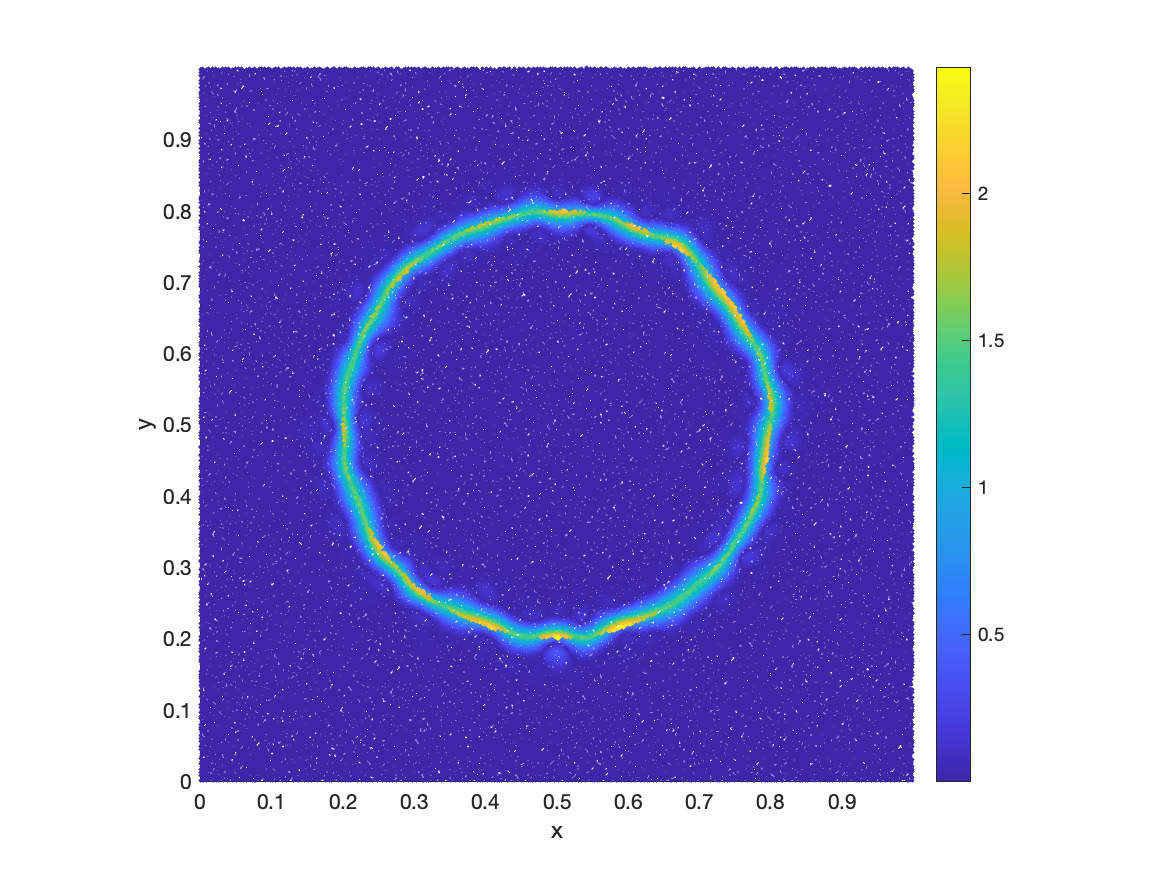}
		\end{tabular}
\end{center}
			\caption{Approximation of the function $f_1$, defined in Eq.~\eqref{frankesdisc}, using $n = 50 \times 50$ initial Halton points. The left column shows the results obtained with the classical RBF$_{\text{M2}}$ algorithm, while the right column corresponds to the data-dependent DD-RBF$_{\text{M2}}$ algorithm. The first row displays the final approximations, and the second row presents the error distributions across the domain.}
		\label{exp11_2D}
	\end{figure}

	\begin{figure}[htbp!]
\begin{center}
		\begin{tabular}{cc}
	\includegraphics[width=6cm]{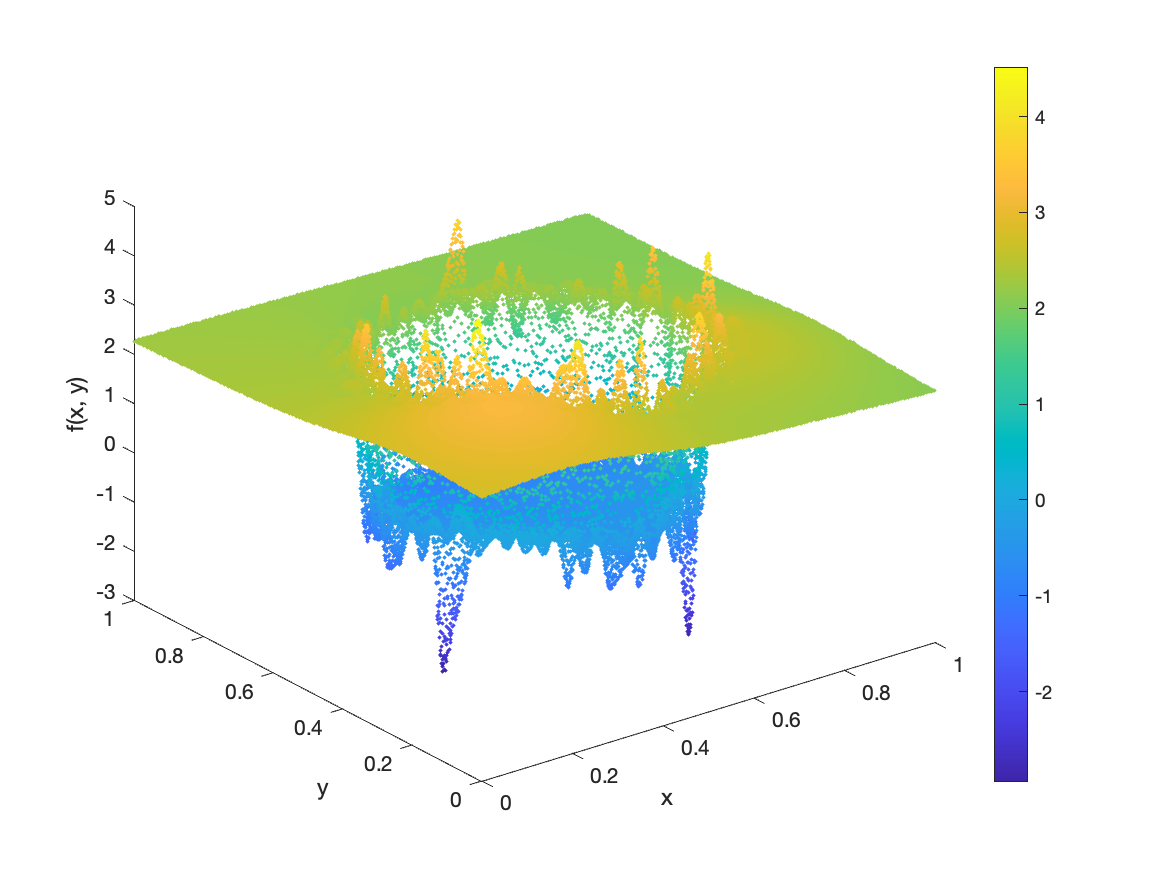} & 	\includegraphics[width=6cm]{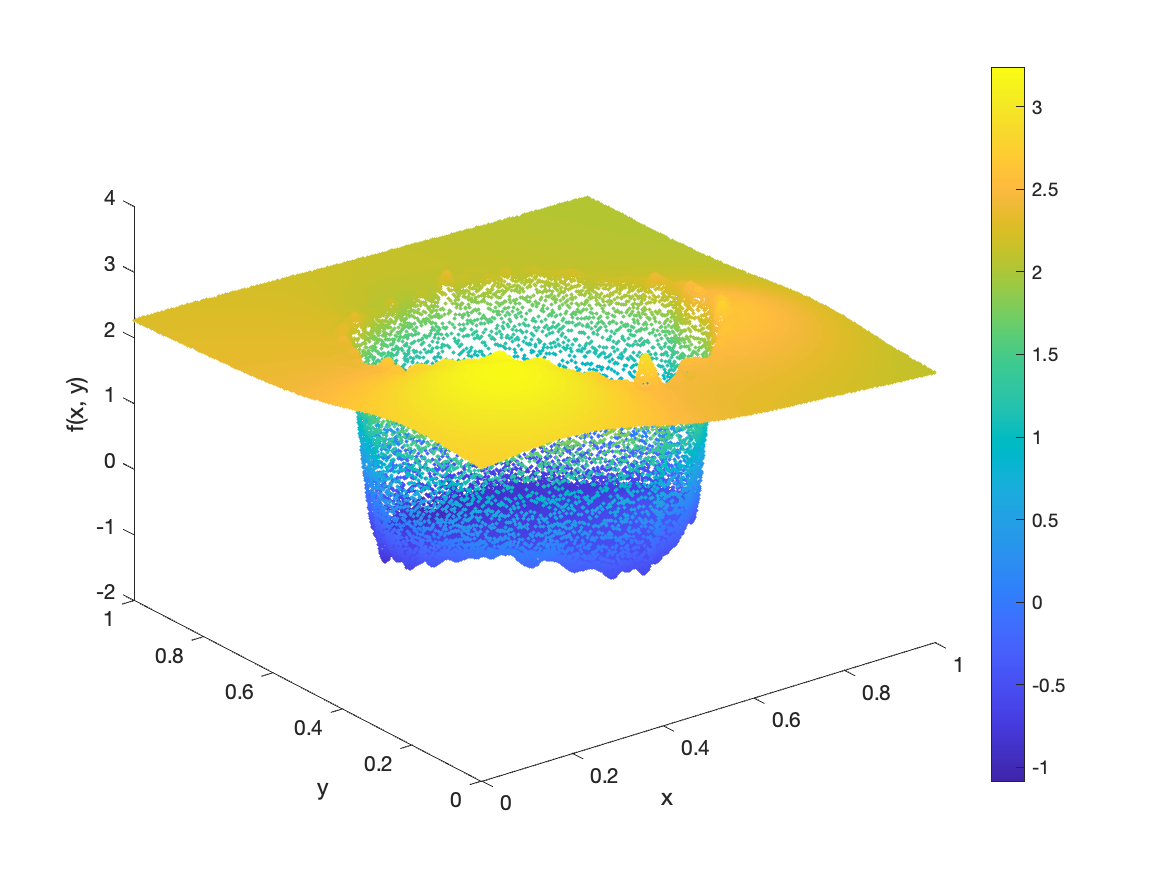}\\
	\includegraphics[width=6cm]{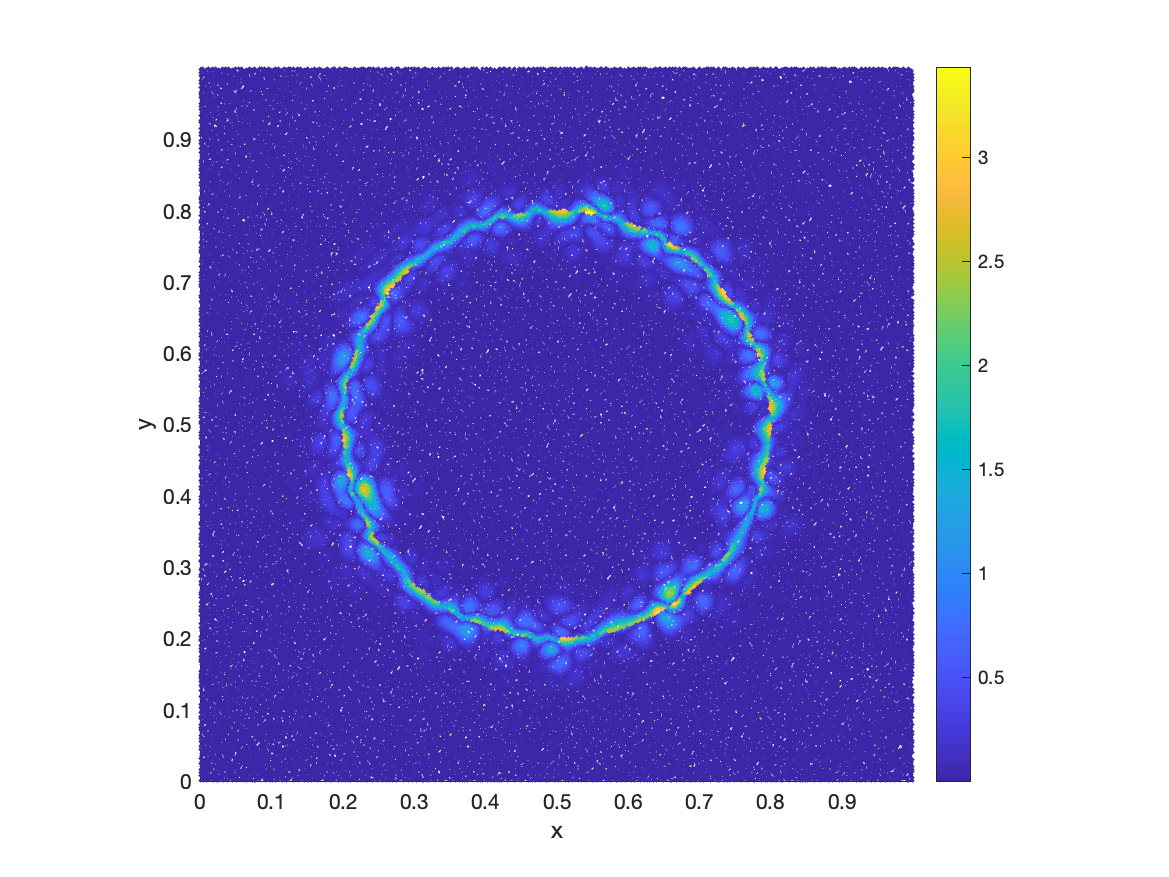} & 	\includegraphics[width=6cm]{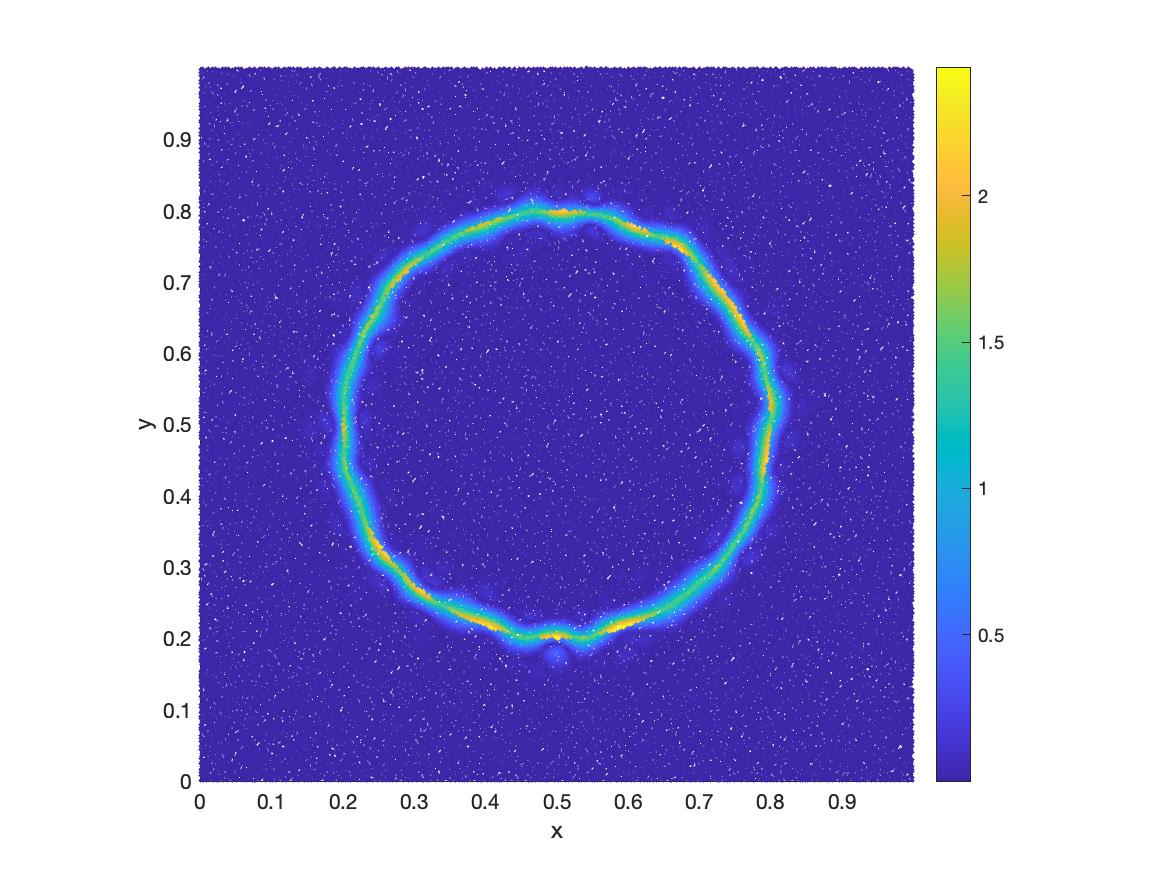}
		\end{tabular}
\end{center}
			\caption{Approximation of the function $f_1$, defined in Eq.~\eqref{frankesdisc}, using $n = 50 \times 50$ initial Halton points. The left column shows the results obtained with the classical RBF$_{\text{M4}}$ algorithm, while the right column corresponds to the data-dependent DD-RBF$_{\text{M4}}$ algorithm. The first row displays the final approximations, and the second row presents the error distributions across the domain.}
		\label{exp12_2D}
	\end{figure}

\begin{table}[ht]
\centering

\begin{tabular}{|l|c|c|}
\hline
\textbf{Kernel} & $\kappa$ \textbf{(Classical)} & $\kappa$ \textbf{(data-dependent)} \\
\hline
RBF$_{\text{G}}$ & 3.9606e+06 & 3.6626e+06 \\
RBF$_{\text{IMQ}}$ & 8.0937e+05 & 7.7192e+05 \\
RBF$_{\text{W2}}$ & 2.5546e+05 & 2.4316e+05 \\
RBF$_{\text{W4}}$ & 5.5925e+06 & 5.3440e+06 \\
RBF$_{\text{M2}}$ & 2.3368e+08 & 2.2289e+08 \\
RBF$_{\text{M4}}$ & 9.3723e+11 & 8.9575e+11 \\
\hline
\end{tabular}
\caption{Condition numbers $\kappa$ for classical and data-dependent RBF interpolation methods across different kernels for the experiments shown in Figures \ref{exp7_2D} to \ref{exp12_2D}, where the we have used Halton points.}
\label{tabla_condicion_nounif_2D}
\end{table}

\section{Conclusions}
In this work, we have introduced a data-dependent modification of classical Radial Basis Function interpolation techniques aimed at reducing oscillations near discontinuities in both one and two dimensions. By adaptively varying the shape parameter of the RBFs, forcing it to tend to infinity near discontinuities, we achieve kernel functions that locally resemble delta functions and that we lately eliminate from the interpolation, effectively suppressing spurious oscillations.

The discontinuity detection mechanism, based on smoothness indicators designed for both gridded and scattered data, plays a central role in guiding the adaptive adjustment. For gridded data, we use squared undivided second-order differences, while for scattered data, we rely on least squares approximations of the Laplacian scaled by the square of the mean local separation.

We have proven the invertibility of the resulting interpolation matrix and given a method for solving the matrix by blocks that assures that the condition number remains comparable to that of systems where discontinuity-adjacent points are excluded. 

Extensive numerical experiments confirm the effectiveness of the proposed method. We have shown that for piecewise smooth univariate and bivariate data, the data-dependent algorithm significantly reduces oscillations along discontinuity curves compared to its classical counterpart. While the classical method introduces severe oscillations that propagate far from the discontinuity, the data-dependent approach mitigates these effects, albeit with some smearing. Importantly, the condition numbers for both methods remain similar, with slightly better behavior observed for the data-dependent algorithm.

These results hold consistently across both gridded and Halton-distributed data, demonstrating the robustness of the method. Overall, the proposed data-dependent RBF interpolation technique offers a reliable and accurate tool for reconstructing functions with discontinuities, with potential applications in several fields.

\section*{ Declaration of generative AI and AI-assisted technologies in the manuscript preparation process}

During the preparation of this work, the author(s) used Microsoft Copilot (based on GPT technology) to check English grammar, improve clarity, and assist in rephrasing some sentences for better readability. After using this tool, the author(s) reviewed and edited the content as needed and take full responsibility for the content of the published article.

\end{document}